\documentclass[a4paper, 12pt]{article}

\usepackage{amsmath} 
\usepackage{theorem}
\usepackage{amssymb}
\usepackage{amsfonts}
\usepackage{latexsym}
\usepackage{amscd}
\usepackage{diagramb}

\usepackage{graphicx}
\usepackage{enumerate}
\usepackage{cancel}
\usepackage{extarrows}
\usepackage{upgreek}

\usepackage{pxfonts}

\usepackage{stmaryrd}

\input xy
\xyoption{all}

\DeclareMathAlphabet{\mathpzc}{OT1}{pzc}{m}{it}

\usepackage{xspace,colortbl}
\usepackage{color}

\definecolor{verde}{rgb}{0.,0.7,0.}
\definecolor{indigo}{rgb}{.18, .34, .78}
\definecolor{indigo1}{rgb}{.18, .24, .78}
\definecolor{indigo2}{rgb}{.18, .14, .78}
\definecolor{indigo3}{rgb}{.18, 0., .78}
\definecolor{rojo}{rgb}{1,0,0}
\definecolor{negro}{rgb}{0,0,0}
\definecolor{lila}{rgb}{.46, .16, .78}
\definecolor{lila1}{rgb}{.46, .16, .86}
\definecolor{lila2}{rgb}{.56, .16, .86}
	\definecolor{lila3}{rgb}{.63, .16, .78}
\definecolor{lila4}{rgb}{.7, .16, .78}
\definecolor{lila5}{rgb}{.78, .26, .78}
\definecolor{lila6}{rgb}{.6, 0., .78}

\theoremstyle{plain}
\theoremheaderfont{\normalfont\bfseries}

\newtheorem{thm}{Theorem}[section]

\newtheorem{lma}[thm]{Lemma}
\newtheorem{cor}[thm]{Corollary}
\newtheorem{defn}[thm]{Definition}

\theorembodyfont{\rmfamily}

\newtheorem{rem}[thm]{Remark}
\newtheorem{prop}[thm]{Proposition}
\newtheorem{ex}[thm]{Example}
\newcommand{\qed}{\hfill\quad\fbox{\rule[0mm]{0,0cm}{0,0mm}}  \par\bigskip}

\newcommand{\x}{\mbox{-}}

\newcommand{\s}{\hspace{0,06cm}}

\newcommand{\Ob}{{\mathcal Ob}}
\newcommand{\Cat}{\operatorname {Cat}}

\newcommand{\Tens}{{\rm Tens}}

\newcommand{\DblPs}{{\rm DblPs}}
\newcommand{\Fun}{{\rm Fun}}
\newcommand{\rtr}{\hspace{-0,08cm}\triangleright}
\newcommand{\ltr}{\hspace{-0,08cm}\triangleleft}
\newcommand{\Bicat}{{\rm Bicat}}
\newcommand{\I}{{\mathcal I}}
\newcommand{\Nat}{\operatorname {Nat}}

\newcommand{\Aa}{{\mathbb A}}
\newcommand{\Bb}{{\mathbb B}}
\newcommand{\Cc}{{\mathbb C}}
\newcommand{\Dd}{{\mathbb D}}

\newcommand{\Del}{\boxtimes}
\newcommand{\comp}{\circ}
\newcommand{\iso}{\cong}
\newcommand{\ot}{\otimes}

\newcommand{\C}{{\mathcal C}}
\newcommand{\Tau}{{\mathcal T}}
\newcommand{\M}{{\mathcal M}}
\newcommand{\D}{{\mathcal D}}
\newcommand{\F}{{\mathcal F}}
\newcommand{\G}{{\mathcal G}}

\newcommand{\N}{{\mathcal N}}
\newcommand{\A}{{\mathcal A}}
\newcommand{\B}{{\mathcal B}}
\newcommand{\E}{{\mathcal E}}

\newcommand{\Ll}{{\mathcal L}}

\newcommand{\crta}{\overline}

\newcommand{\Id}{\operatorname {Id}}
\newcommand{\id}{\operatorname {id}}

\newcommand{\Epsilon}{\varepsilon}

\def\K{{\mathcal K}}  

\def\Dd{{\mathbb D}}

\def\u#1{\underline{#1}}

\newcommand{\Bimod}{\operatorname{Bimod}}

\newcommand{\im}{{\rm Im}\,}

\newcommand{\cref}[1]{C.~\ref{c:#1}}

\newcommand{\exlabel}[1]{\label{ex:#1}}
\newcommand{\exref}[1]{Example~\ref{ex:#1}}
\newcommand{\lelabel}[1]{\label{le:#1}}
\newcommand{\leref}[1]{Lemma~\ref{le:#1}}
\newcommand{\eqlabel}[1]{\label{eq:#1}}
\newcommand{\equref}[1]{(\ref{eq:#1})}

\newcommand{\delabel}[1]{\label{de:#1}}
\newcommand{\deref}[1]{Definition~\ref{de:#1}}
\newcommand{\prlabel}[1]{\label{pr:#1}}
\newcommand{\prref}[1]{Proposition~\ref{pr:#1}}
\newcommand{\colabel}[1]{\label{co:#1}}
\newcommand{\coref}[1]{Corollary~\ref{co:#1}}

\newcommand{\rmlabel}[1]{\label{rm:#1}}
\newcommand{\rmref}[1]{Remark~\ref{rm:#1}}
\newcommand{\selabel}[1]{\label{se:#1}}
\newcommand{\seref}[1]{Section~\ref{se:#1}}
\newcommand{\sslabel}[1]{\label{ss:#1}}
\newcommand{\ssref}[1]{Subsection~\ref{ss:#1}}

\newcommand\Doublearrow{%
        \mathrel{\vcenter{\mathsurround0pt
                \ialign{##\crcr
                        \noalign{\nointerlineskip}$\stackrel{s}{\longrightarrow}$\crcr
                          \noalign{\nointerlineskip}$\underset{t}{\longrightarrow}$\crcr
              }%
        }}%
}

\newcommand\tripplearrow{%
        \mathrel{\vcenter{\mathsurround0pt
                \ialign{##\crcr
                        \noalign{\nointerlineskip}$\longrightarrow$\crcr
                        \noalign{\nointerlineskip}$\longleftarrow$\crcr
                        \noalign{\nointerlineskip}$\longrightarrow$\crcr
                }%
        }}%
}

\newcommand\triarrows{%
        \mathrel{\vcenter{\mathsurround0pt
                \ialign{##\crcr
                        \noalign{\nointerlineskip}$\longrightarrow$\crcr
                        \noalign{\nointerlineskip}$\longrightarrow$\crcr
                        \noalign{\nointerlineskip}$\longrightarrow$\crcr
                }%
        }}%
}

\newcommand*{\threefrac}[3]{%
  \begin{array}{@{\,}c@{\,}}%
    #1\\
    \hline
    #2\\
    \hline
    #3%
  \end{array}%
}

\newdir{:=}{{}}
\newcommand{\fit}[3]{\ar@{:=}@/{#3}/[#1] |{\Downarrow #2} }

\begin{document}


\title{Enrichment and internalization in tricategories, \\
the case of tensor categories and \\
alternative notion to intercategories}

\author{Bojana Femi\'c \vspace{6pt} \\
{\small Mathematical Institite of  \vspace{-2pt}}\\
{\small Serbian Academy of Sciences and Arts } \vspace{-2pt}\\
{\small  Knez Mihailova 35,} \vspace{-2pt}\\
{\small  11 000 Belgrade, Serbia}}

\date{}
\maketitle


\begin{center}
{\bf Dedicated to Gabi B\"ohm}
\end{center}

\begin{abstract}
This paper emerged as a result of tackling the following three issues. 
Firstly, we would like the well known embedding of bicategories into pseudo double categories to be monoidal, which it is not if one uses 
the usual notion of a monoidal pseudo double category. 
Secondly, in \cite{Gabi} the question was raised: which would be an alternative notion to intercategories of Grandis and Par\'e, so that 
monoids in B\"ohm's monoidal category $(Dbl,\ot)$ of strict double categories and strict double functors with a Gray type monoidal product be an example of it? 
We obtain and prove that precisely the monoidal structure of $(Dbl,\ot)$ resolves the first question. On the other hand, resolving the second 
question, we upgrade the category $Dbl$ to a tricategory $\DblPs$ and propose 
to consider internal categories in this tricategory. 
Apart from monoids in $(Dbl,\ot)$ - more importanlty, weak pseudomonoids in a tricategory containing $(Dbl,\ot)$ as a sub 1-category - 
most of the examples of intercategories are also examples of this gadget, the ones that escape are those 
that rely on laxness of the product on the pullback, as duoidal categories. 
For the latter purpose we define categories internal to 
tricategories (of the type of $\DblPs$), which simultaneously serves our third motive. Namely, inspired by the tricategory and $(1\times 2)$-category 
of tensor categories, we prove under mild conditions that categories enriched over certain type of tricategories 
may be made into categories internal in them. We illustrate this occurrence for tensor categories with respect to the ambient tricategory 
$2\x\Cat_{wk}$ of 2-categories, pseudofunctors, pseudonatural transformations and modifications. 
\end{abstract}

\section{Introduction}

It is well-known that 2-categories embed in strict double categories and that bicategories embed in pseudo double categories. 
However, it is not clear which of the definitions of a monoidal pseudo double category existent in the literature 
would be suitable to have a monoidal version of the result, that monoidal bicategories embed to monoidal pseudo double categories. 
This question we resolve in \ssref{mon embed}. Namely, seeing a monoidal bicategory as a one object tricategory, 
we consider the equivalent one object Gray 3-category (by the coherence of tricategories of \cite{GPS}), which is nothing but a monoid in the 
monoidal category $Gray$, {\em i.e.} a Gray monoid (see \cite{DS}, \cite[Lemma 4]{BN}). 
We then prove that $Gray$ embeds as a monoidal category in the monoidal category $Dbl$ from \cite{Gabi} of strict double categories and 
double functors. 
For the above-mentioned embedding we give an explicit description of the monoidal structure of $Dbl$
(in \cite{Gabi} only an explicit description of the structure of a monoid is given). Analogously as in \cite{Gray}, 
we introduce a {\em cubical double functor} along the way. We also show why the other notions of monoidal double categories 
(monoids in the category of strict double categories and strict double functors from \cite{BMM},  
and pseudomonoids in the 2-category $PsDbl$ of pseudo double categories, pseudo double functors and vertical transformations from 
\cite{Shul}) do not obey the embedding in question. 


Then we turn to the following question of B\"ohm: if a monoid in her monoidal category $Dbl$ could fit some framework similar to 
intercategories of Grandis and Par\'e. Observe that neither of the two notions would be more general than the other. While 
intercategories are categories internal in the 2-category $LxDbl$ of pseudo double categories, lax double functors and horizontal 
transformations, in the structure of B\"ohm's monoid in $Dbl$ the relevant objects are {\em strict} double categories and 
morphisms {\em double pseudo functors} in the sense of \cite{Shul1} (they are given by isomorphisms in both directions). 
By Strictification Theorem of \cite[Section 7.5]{GP:Limits} 
every pseudo double category is equivalent to a strict double category, thus on the level of objects in the ambient category basically 
nothing is changed. Though, going from lax double functors to double pseudo functors, one tightens in one direction and weakens in the other. 

In the search for a desired framework, we define 2-cells among 
double pseudo functors and we get that instead of a 2-category, we indeed have a tricategory structure, including modifications 
as 3-cells. This led us to propose an alternative notion for intercategories, as categories internal in this tricategory of strict double categories, which we denote by $\DblPs$. 

Contrarily to $LxDbl$, the 1-cells of the 2-category $PsDbl$ 
are particular cases of double pseudo functors. Having in mind the above Strictification Theorem and adding only the trivial 
3-cells to $PsDbl$, we also prove that thus obtained tricategory $PsDbl^*_3$ embeds in our tricategory $\DblPs$. 
As a byproduct to this proof we obtain a more general result: supposing that there is a connection (\cite{BS}) on 1v-components of 
strong vertical transformations (\cite[Section 7.4]{GP:Limits}), there is a bijection between 
strong vertical transformations and those strong horizontal transformations whose 1h-cell components are 1h-companions of some 1v-cells.
This we prove in \coref{bij}.

In the literature the notion of a category internal in a Gray category was introduced in \cite{DH}. For our announced purpose 
we introduce a notion of a category internal in a tricategory $V$ 
which is of a similar type as $\DblPs$. Most importantly, $V$ has an underlying 1-category, a property to which we refer to as 
{\em 1-strict}, and apart from the interchange law, the associativity on the 2-cells holds only up to isomorphism, 
which makes it weaker than a Gray category. 
We describe the structure of a category internal in $\DblPs$, and, similarly to intercategories, we give a 
geometric interpretation of it in the form of cubes. Moreover, we upgrade the monoidal category $Dbl$ to a 2-category 
$Dbl_2$, and this one to a (1-strict) tricategory $Dbl_3$, and show how pseudomonoids in $Dbl_2$ 
and ``weak pseudomonoids'' in $Dbl_3$, are categories internal in $\DblPs$. 
The examples of intercategories treated in \cite{GP:Fram} are also examples of categories internal in $\DblPs$, 
except from duoidal categories, which rely on lax double functors, rather than pseudo ones.

\medskip

As we mentioned above, bicategories (which are categories enriched over the 2-category $Cat_2$ of categories) embed into 
double categories (which are categories internal in $Cat_2$). We study what happens in one dimension higher. Namely, the 
standard example of the bicategory of algebras and their bimodules has its well-known analogue in one dimension higher.  
In order to formalize an allusive result,  
we first introduce the notions of tricategorical pullbacks and (co)products. For 
the tricategory $\Tens$ of tensor categories, bimodule categories over the latter, bimodule functors and bimodule natural 
transformations, we show that it is a category enriched over the tricategory $2\x\Cat_{wk}$, of 2-categories, pseudofunctors, 
weak natural transformations and modifications.
Moreover, we show that $\Tens$ is a part of the structure of a category internal in $2\x\Cat_{wk}$. This responds to 
\cite[Example 2.14]{DH}, where it was conjectured that $\Tens$ is a category internal in the Gray 3-category $2CAT_{nwk}$, which differs from 
$2\x\Cat_{wk}$ in that its 1-cells are 2-functors, rather than pseudofunctors as in $2\x\Cat_{wk}$. Motivated by this example, 
we prove in \prref{weak-int} 
that under mild conditions categories enriched over 1-strict tricategories can be made into categories internal in them. This 
generalizes to tricategories analogous results from \cite{Ehr:III} and \cite{Power}. 
Since $2\x\Cat_{wk}$ embeds into $\DblPs$, this $(1\times 2)$-category of tensor categories is also an example of our alternative 
notion to intercategories.

\bigskip

The composition of the paper is the following. In Section 2 we give a description of the monoidal structure of $Dbl$ of B\"ohm, 
we define cubical double functors and prove that $Gray$ embeds into $Dbl$. Section 3 is dedicated to the construction of our 
tricategory $\DblPs$. In \seref{Ps embeds} we prove that the tricategory $PsDbl^*_3$ embeds into $\DblPs$ and prove the bijection 
between vertical and horizontal strong transformations, supposing the mentioned connection. In \seref{3-pull(co)pr} we define 
tricategorical pullbacks and (co)products, and in the next one we define categories internal in 1-strict tricategories. 
In the subsequent section we describe the structure of a category internal 
in $\DblPs$, we show here that monoids in B\"ohm's $Dbl$, pseudomonoids in $Dbl_2$ and weak pseudomonoids in $Dbl_3$ fit this 
setting, and we present the announced geometric interpretation on cubes. In \seref{enrich-int} we define categories enriched 
in 1-strict tricategories, we prove that categories enriched over certain type of 1-strict tricategories $V$ are special cases 
of categories internal in $V$, and we discuss examples in dimensions 2 and 1. In the last section we show the enrichment and 
internal structures of $\Tens$ in $2\x\Cat_{wk}$, illustrating the mentioned result.

\section{Monoidal double categories into which monoidal bicategories embed} \selabel{mon}

Although bicategories embed into pseudo double categories, this embedding is not monoidal, 
if one takes for a definition of a monoidal double category any of the ones in \cite{BMM} 
(a monoid in the category of strict double categories and strict double functors) and in \cite{Shul} 
(a pseudomonoid in the 2-category of pseudo double categories, pseudo double functors and say vertical transformations). 
Namely, a monoidal bicategory is a one object tricategory, so its 0-cells have a product associative up to an {\em equivalence}. 
This is far from what happens in the mentioned two definitions of a monoidal double category. 
Even if we consider the triequivalence due to \cite{GPS} of a monoidal bicategory with a one object Gray-category, that is, a Gray monoid, 
one does not have monoidal embeddings, as we will show. Nevertheless, a Gray monoid, which is in fact a monoid in the monoidal category 
$(Gray,\ot)$ of 2-categories, 2-functors with the monoidal product due to Gray \cite{Gray}, can be seen as a monoid in the monoidal category 
$(Dbl,\ot)$ of strict double categories and strict double functors with the monoidal product constructed in \cite[Section 4.3]{Gabi}. 
We will show in this section that $(Gray,\ot)$ embeds monoidaly into $(Dbl,\ot)$.

\subsection{The monoidal structure in $(Dbl,\ot)$ } \sslabel{mon str}

The monoidal structure in $(Dbl,\ot)$ is constructed in the analogous way as in \cite{Gray}. For two double categories $\Aa, \Bb$ 
a double category $\llbracket\Aa,\Bb\rrbracket$ is defined in \cite[Section 2.2]{Gabi} which induces a functor $\llbracket-,-\rrbracket: Dbl^{op}\times Dbl\to Dbl$.  
Representability of the functor $Dbl(\Aa, \llbracket\Bb,-\rrbracket): Dbl\to Set$ is proved, which induces a functor 
$-\ot-: Dbl\times Dbl\to Dbl$. 
For two double categories $\Aa,\Bb$ we will give a full description of the double category $\Aa\ot\Bb$. We will do this using the natural isomorphism 
\begin{equation} \eqlabel{nat iso}
Dbl(\Aa\ot\Bb, \Cc)\iso Dbl(\Aa, \llbracket\Bb,\Cc\rrbracket),
\end{equation}
that is, characterizing a double functor $F:\Aa\to\llbracket\Bb,\Cc\rrbracket$ for another double category $\Cc$ and reading off 
the structure of the image double category $F(\Aa)(\Bb)$, setting $\Cc=\Aa\ot\Bb$. 

\medskip

Let us fix the notation in a double category $\Dd$. Objects we denote by $A,B, \dots$, horizontal 1-cells we will call 
for brevity 1h-cells and denote them by $f, f', g, F, \dots$, vertical 1-cells we will call 1v-cells and denote by $u,v, U, \dots$, 
and squares we will call just 2-cells and denote them by $\omega, \zeta, \dots$. In this section, we denote 
the horizontal identity 1-cell by $1_A$, vertical identity 1-cell by $1^A$ for an object $A\in\Dd$, 
horizontal identity 2-cell on a 1v-cell $u$ by $Id^u$, and vertical identity 2-cell on a 1h-cell $f$ by $Id_f$ (with subindices we denote 
those identity 1- and 2-cells which come from the horizontal 2-category lying in $\Dd$). 
The composition of 1h-cells as well as the horizontal composition of 2-cells we will denote by $\odot$ in this section, while the composition 
of 1v-cells as well the vertical composition of 2-cells we will denote by juxtaposition. 

We start by noticing that a strict double functor $F:\Cc\to\Dd$ is given by 1) the data: images on objects, 1h-, 1v- and 2-cells of $\Cc$, and 2) rules (in $\Dd$): 

$F(u'u)=F(u')F(u), \quad F(1^A)=1^{F(A)},$ 

$F(\omega\zeta)=F(\omega)F(\zeta), \quad F(1_f)=1_{F(f)},$ 

$F(g\odot f)=F(g)\odot F(f), \quad F(\omega\odot\zeta)=F(\omega)\odot F(\zeta),$ 

$F(1_A)=1_{F(A)},  \quad F(Id^u)=Id^{F(u)}.$

Having in mind the definition of a double category $\llbracket\Aa,\Bb\rrbracket$ from \cite[Section 2.2]{Gabi}, writing out the list of the data and relations 
that determine a double functor $F:\Aa\to\llbracket\Bb,\Cc\rrbracket$, one gets the following characterization of it:

\begin{prop} \prlabel{char df}
A double functor $F:\Aa\to\llbracket\Bb,\Cc\rrbracket$ of double categories consists of the following: \\
1. double functors 
$$(-,A):\Bb\to\Cc\quad\text{ and}\quad (B,-):\Aa\to\Cc$$ 
such that $(-,A)\vert_B=(B,-)\vert_A=(B,A)$, 
for objects $A\in\Aa, B\in\Bb$, \\
2. given 1h-cells $A\stackrel{F}{\to} A'$ and $B\stackrel{f}{\to} B'$ and 1v-cells $A\stackrel{U}{\to} \tilde A$ and $B\stackrel{u}{\to} \tilde B$ 
there are 2-cells
$$
\bfig
 \putmorphism(-150,500)(1,0)[(B,A)`(B,A')`(B,F)]{600}1a
 \putmorphism(450,500)(1,0)[\phantom{A\ot B}`(B', A') `(f, A')]{680}1a
 \putmorphism(-150,50)(1,0)[(B,A)`(B', A)`(f, A)]{600}1a
 \putmorphism(450,50)(1,0)[\phantom{A\ot B}`(B', A') `(B', F)]{680}1a
\putmorphism(-180,500)(0,-1)[\phantom{Y_2}``=]{450}1r
\putmorphism(1100,500)(0,-1)[\phantom{Y_2}``=]{450}1r
\put(350,260){\fbox{$(f,F)$}}
\efig
$$

$$
\bfig
\putmorphism(-150,50)(1,0)[(B,A)`(B,A')`(B,F)]{600}1a
\putmorphism(-150,-400)(1,0)[(\tilde B, A)`(\tilde B,A') `(\tilde B,F)]{640}1a
\putmorphism(-180,50)(0,-1)[\phantom{Y_2}``(u,A)]{450}1l
\putmorphism(450,50)(0,-1)[\phantom{Y_2}``(u,A')]{450}1r
\put(0,-180){\fbox{$(u, F)$}}
\efig
\quad
\bfig
\putmorphism(-150,50)(1,0)[(B,A)`(B',A)`(f,A)]{600}1a
\putmorphism(-150,-400)(1,0)[(B, \tilde A)`(B', \tilde A) `(f,\tilde A)]{640}1a
\putmorphism(-180,50)(0,-1)[\phantom{Y_2}``(B,U)]{450}1l
\putmorphism(450,50)(0,-1)[\phantom{Y_2}``(B',U)]{450}1r
\put(0,-180){\fbox{$(f,U)$}}
\efig
$$

$$
\bfig
 \putmorphism(-150,500)(1,0)[(B,A)`(B,A) `=]{600}1a
\putmorphism(-180,500)(0,-1)[\phantom{Y_2}`(B, \tilde A) `(B,U)]{450}1l
\put(0,50){\fbox{$(u,U)$}}
\putmorphism(-150,-400)(1,0)[(\tilde B, \tilde A)`(\tilde B, \tilde A) `=]{640}1a
\putmorphism(-180,50)(0,-1)[\phantom{Y_2}``(u,\tilde A)]{450}1l
\putmorphism(450,50)(0,-1)[\phantom{Y_2}``(\tilde B, U)]{450}1r
\putmorphism(450,500)(0,-1)[\phantom{Y_2}`(\tilde B, A) `(u,A)]{450}1r
\efig
$$
of which $(f,F)$ is vertically invertible  and $(u,U)$ is horizontally invertible, which satisfy: 
\begin{enumerate} [a)]
\item (11)\quad $(1_B,F)=Id_{(B,F)}\quad\text{and}\quad (f,1_A)=Id_{(f,A)}$ 

(21)\quad  $(1^B,F)=Id_{(B,F)}\quad\text{and}\quad (u,1_A)=Id^{(u,A)}$ 

(12)\quad  $(1_B,U)=Id^{(B,U)}\quad\text{and}\quad (f,1^A)=Id_{(f,A)}$ 

(22)\quad  $(1^B,U)=Id^{(B,U)}\quad\text{and}\quad (u,1^A)=Id^{(u,A)};$

\item (11)  
$(f'f, F)=
\bfig

 \putmorphism(-150,0)(1,0)[(B,A)`(B,A')`(B,F)]{600}1a
 \putmorphism(450,0)(1,0)[\phantom{A\ot B}`(B', A') `(f,A')]{680}1a

 \putmorphism(-150,-450)(1,0)[(B,A)`(B',A)`(f,A)]{600}1a
 \putmorphism(450,-450)(1,0)[\phantom{A\ot B}`(B', A') `(B',F)]{680}1a
 \putmorphism(1100,-450)(1,0)[\phantom{A'\ot B'}`(B'', A') `(f', A')]{660}1a

\putmorphism(-180,0)(0,-1)[\phantom{Y_2}``=]{450}1r
\putmorphism(1100,0)(0,-1)[\phantom{Y_2}``=]{450}1r
\put(350,-240){\fbox{$(f,F)$}}
\put(1000,-700){\fbox{$(f',F)$}}

 \putmorphism(450,-900)(1,0)[(B', A)` (B'', A) `(f', A)]{680}1a
 \putmorphism(1100,-900)(1,0)[\phantom{A''\ot B'}`(B'', A') ` (B'', F)]{660}1a

\putmorphism(450,-450)(0,-1)[\phantom{Y_2}``=]{450}1l
\putmorphism(1750,-450)(0,-1)[\phantom{Y_2}``=]{450}1r
\efig
$ \\
and \\
$(f,F\s'F)=
\bfig
 \putmorphism(450,500)(1,0)[(B, A') `(B, A'') `(B, F')]{680}1a
 \putmorphism(1140,500)(1,0)[\phantom{A\ot B}`(B', A'') ` (f, A'')]{680}1a

 \putmorphism(-150,50)(1,0)[(B, A) `(B, A')`(B, F)]{600}1a
 \putmorphism(450,50)(1,0)[\phantom{A\ot B}`(B', A') `(f, A')]{680}1a
 \putmorphism(1130,50)(1,0)[\phantom{A\ot B}`(B', A'') ` (B', F\s')]{680}1a

\putmorphism(450,500)(0,-1)[\phantom{Y_2}``=]{450}1r
\putmorphism(1750,500)(0,-1)[\phantom{Y_2}``=]{450}1r
\put(1020,270){\fbox{$ (f,F\s')$}}

 \putmorphism(-150,-400)(1,0)[(B, A)`(B', A) `(f,A)]{640}1a
 \putmorphism(480,-400)(1,0)[\phantom{A'\ot B'}`(B', A') `(B', F)]{680}1a

\putmorphism(-180,50)(0,-1)[\phantom{Y_2}``=]{450}1l
\putmorphism(1120,50)(0,-1)[\phantom{Y_3}``=]{450}1r
\put(310,-200){\fbox{$ (f,F)$}}

\efig
$ \\

(21)\quad $(u'u, F)=(u', F)(u,F)\quad\text{and}\quad (u, F\s'F)=(u,F\s')\odot(u,F)$ 

(12)\quad $(f'f, U)=(f', U)\odot(f, U)\quad\text{and}\quad (f,U'U)=(f,U')(f,U)$ 

(22)\quad 
$$(u,U'U)=
\bfig
 \putmorphism(-150,500)(1,0)[(B,A)`(B,A) `=]{600}1a
\putmorphism(-180,500)(0,-1)[\phantom{Y_2}`(B, \tilde A) `(B,U)]{450}1l
\put(0,50){\fbox{$(u,U)$}}
\putmorphism(-150,-400)(1,0)[(\tilde B, \tilde A)`(\tilde B, \tilde A) `=]{640}1a
\putmorphism(-180,50)(0,-1)[\phantom{Y_2}``(u,\tilde A)]{450}1l
\putmorphism(450,50)(0,-1)[\phantom{Y_2}``(\tilde B, U)]{450}1r
\putmorphism(450,500)(0,-1)[\phantom{Y_2}`(\tilde B, A) `(u,A)]{450}1r
\putmorphism(-820,50)(1,0)[(B, \tilde A)``=]{520}1a
\putmorphism(-820,50)(0,-1)[\phantom{(B, \tilde A')}``(B,U')]{450}1l
\putmorphism(-820,-400)(0,-1)[(B, \tilde A')`(\tilde B, \tilde A')`(u,\tilde A')]{450}1l
\putmorphism(-820,-850)(1,0)[\phantom{(B, \tilde A)}``=]{520}1a
\putmorphism(-150,-400)(0,-1)[(\tilde B, \tilde A)`(\tilde B, \tilde A') `(\tilde B, U')]{450}1r
\put(-650,-630){\fbox{$(u,U')$}}
\efig
$$
and
$$(u'u,U)=
\bfig
 \putmorphism(-150,500)(1,0)[(B,A)`(B,A) `=]{600}1a
\putmorphism(-180,500)(0,-1)[\phantom{Y_2}`(B, \tilde A) `(B,U)]{450}1l
\put(0,50){\fbox{$(u,U)$}}
\putmorphism(-150,-400)(1,0)[(\tilde B, \tilde A)` `=]{500}1a
\putmorphism(-180,50)(0,-1)[\phantom{Y_2}``(u,\tilde A)]{450}1l
\putmorphism(450,50)(0,-1)[\phantom{Y_2}`(\tilde B, \tilde A)`(\tilde B, U)]{450}1r
\putmorphism(450,500)(0,-1)[\phantom{Y_2}`(\tilde B, A) `(u,A)]{450}1r
\putmorphism(450,50)(1,0)[\phantom{(B, \tilde A)}`(\tilde B, A)`=]{620}1a
\putmorphism(1070,50)(0,-1)[\phantom{(B, \tilde A')}``(u',A)]{450}1r
\putmorphism(1070,-400)(0,-1)[(\tilde B', A)`(\tilde B', \tilde A)`(\tilde B', U)]{450}1r
\putmorphism(450,-850)(1,0)[\phantom{(B, \tilde A)}``=]{520}1a
\putmorphism(450,-400)(0,-1)[\phantom{(B, \tilde A)}`(\tilde B', \tilde A) `(U',\tilde A)]{450}1l
\put(600,-630){\fbox{$(u',U)$}}
\efig
$$
\item (11)\quad 
$$
\bfig
 \putmorphism(-150,500)(1,0)[(B,A)`(B,A')`(B,F)]{600}1a
 \putmorphism(450,500)(1,0)[\phantom{A\ot B}`(B', A') `(f, A')]{680}1a
 \putmorphism(-150,50)(1,0)[(B,A)`(B', A)`(f, A)]{600}1a
 \putmorphism(450,50)(1,0)[\phantom{A\ot B}`(B', A') `(B', F)]{680}1a

\putmorphism(-180,500)(0,-1)[\phantom{Y_2}``=]{450}1r
\putmorphism(1100,500)(0,-1)[\phantom{Y_2}``=]{450}1r
\put(350,260){\fbox{$(f,F)$}}
\put(650,-180){\fbox{$(v,F)$}}

\putmorphism(-150,-400)(1,0)[(\tilde B,A)`(\tilde B',A) `(g,A)]{640}1a
 \putmorphism(450,-400)(1,0)[\phantom{A'\ot B'}` (\tilde B',A') `(\tilde B', F)]{680}1a

\putmorphism(-180,50)(0,-1)[\phantom{Y_2}``(u,A)]{450}1l
\putmorphism(450,50)(0,-1)[\phantom{Y_2}``]{450}1l
\putmorphism(610,50)(0,-1)[\phantom{Y_2}``(v,A)]{450}0l 
\putmorphism(1120,50)(0,-1)[\phantom{Y_3}``(v,A')]{450}1r
\put(-40,-180){\fbox{$(\omega,A)$}} 

\efig
=
\bfig
\putmorphism(-150,500)(1,0)[(B,A)`(B,A')`(B,F)]{600}1a
 \putmorphism(450,500)(1,0)[\phantom{A\ot B}`(B', A') `(f, A')]{680}1a

 \putmorphism(-150,50)(1,0)[(\tilde B,A)`(\tilde B,A')`(\tilde B,F)]{600}1a
 \putmorphism(450,50)(1,0)[\phantom{A\ot B}`(\tilde B',A') `(g, A')]{680}1a

\putmorphism(-180,500)(0,-1)[\phantom{Y_2}``(u,A)]{450}1l
\putmorphism(450,500)(0,-1)[\phantom{Y_2}``]{450}1r
\putmorphism(300,500)(0,-1)[\phantom{Y_2}``(u,A')]{450}0r
\putmorphism(1100,500)(0,-1)[\phantom{Y_2}``(v,A')]{450}1r
\put(-20,260){\fbox{$(u,F)$}}
\put(660,260){\fbox{$(\omega,A')$}}

\putmorphism(-150,-400)(1,0)[(\tilde B,A)`(\tilde B',A) `(g, A)]{640}1a
 \putmorphism(450,-400)(1,0)[\phantom{A'\ot B'}` (\tilde B,A') `(\tilde B',F)]{680}1a

\putmorphism(-180,50)(0,-1)[\phantom{Y_2}``=]{450}1l
\putmorphism(1120,50)(0,-1)[\phantom{Y_3}``=]{450}1r
\put(300,-200){\fbox{$(g,F)$}}

\efig
$$
and 
$$
\bfig
 \putmorphism(-150,500)(1,0)[(B,A)`(B,A')`(B,F)]{600}1a
 \putmorphism(450,500)(1,0)[\phantom{A\ot B}`(B', A') `(f, A')]{680}1a
 \putmorphism(-150,50)(1,0)[(B,A)`(B', A)`(f, A)]{600}1a
 \putmorphism(450,50)(1,0)[\phantom{A\ot B}`(B', A') `(B', F)]{680}1a

\putmorphism(-180,500)(0,-1)[\phantom{Y_2}``=]{450}1r
\putmorphism(1100,500)(0,-1)[\phantom{Y_2}``=]{450}1r
\put(350,260){\fbox{$(f,F)$}}
\put(650,-180){\fbox{$(B', \zeta)$}}

\putmorphism(-150,-400)(1,0)[(B,\tilde A)`(B',\tilde A) `(f,\tilde A)]{640}1a
 \putmorphism(450,-400)(1,0)[\phantom{A'\ot B'}` (B',\tilde A') `(B', G)]{680}1a

\putmorphism(-180,50)(0,-1)[\phantom{Y_2}``(B,U)]{450}1l
\putmorphism(450,50)(0,-1)[\phantom{Y_2}``]{450}1l
\putmorphism(610,50)(0,-1)[\phantom{Y_2}``(B',U)]{450}0l 
\putmorphism(1120,50)(0,-1)[\phantom{Y_3}``(B',V)]{450}1r
\put(-40,-180){\fbox{$(f,U)$}} 

\efig
=
\bfig
\putmorphism(-150,500)(1,0)[(B,A)`(B,A')`(B,F)]{600}1a
 \putmorphism(450,500)(1,0)[\phantom{A\ot B}`(B', A') `(f, A')]{680}1a

 \putmorphism(-150,50)(1,0)[(B,\tilde A)`(B,\tilde A')`(B,G)]{600}1a
 \putmorphism(450,50)(1,0)[\phantom{A\ot B}`(B',\tilde A') `(f,\tilde A')]{680}1a

\putmorphism(-180,500)(0,-1)[\phantom{Y_2}``(B,U)]{450}1l
\putmorphism(450,500)(0,-1)[\phantom{Y_2}``]{450}1r
\putmorphism(300,500)(0,-1)[\phantom{Y_2}``(B,V)]{450}0r
\putmorphism(1100,500)(0,-1)[\phantom{Y_2}``(B',V)]{450}1r
\put(-20,260){\fbox{$(B,\zeta)$}}
\put(660,260){\fbox{$(F,V)$}}

\putmorphism(-150,-400)(1,0)[(B,\tilde A)`(B',\tilde A) `(f, \tilde A)]{640}1a
 \putmorphism(450,-400)(1,0)[\phantom{A'\ot B'}` (B',\tilde A') `(B',G)]{680}1a

\putmorphism(-180,50)(0,-1)[\phantom{Y_2}``=]{450}1l
\putmorphism(1120,50)(0,-1)[\phantom{Y_3}``=]{450}1r
\put(300,-200){\fbox{$(f,G)$}}

\efig
$$
(22)\quad
$$
\bfig
 \putmorphism(-150,500)(1,0)[(B,A)`(B,A) `=]{600}1a
 \putmorphism(450,500)(1,0)[(B,A)` `(f,A)]{450}1a
\putmorphism(-180,500)(0,-1)[\phantom{Y_2}`(B, \tilde A) `(B,U)]{450}1l
\put(0,50){\fbox{$(u,U)$}}
\putmorphism(-150,-400)(1,0)[(\tilde B, \tilde A)` `=]{500}1a
\putmorphism(-180,50)(0,-1)[\phantom{Y_2}``(u,\tilde A)]{450}1l
\putmorphism(450,50)(0,-1)[\phantom{Y_2}`(\tilde B, \tilde A)`(\tilde B, U)]{450}1l
\putmorphism(450,500)(0,-1)[\phantom{Y_2}`(\tilde B, A) `(u,A)]{450}1l
\put(600,260){\fbox{$(\omega,A)$}}
\putmorphism(450,50)(1,0)[\phantom{(B, \tilde A)}``(g, A)]{500}1a
\putmorphism(1070,50)(0,-1)[\phantom{(B, A')}`(\tilde B', \tilde A)`(\tilde B',U)]{450}1r
\putmorphism(1070,500)(0,-1)[(B', A)`(\tilde B', A)`(v,A)]{450}1r
\putmorphism(450,-400)(1,0)[\phantom{(B, \tilde A)}``(g, \tilde A)]{500}1a
\put(600,-170){\fbox{$(g,U)$}}
\efig=
\bfig
 \putmorphism(-150,500)(1,0)[(B,A)`(B',A) `(f,A)]{600}1a
 \putmorphism(450,500)(1,0)[\phantom{(B,A)}` `=]{450}1a
\putmorphism(-180,500)(0,-1)[\phantom{Y_2}`(B, \tilde A) `(B,U)]{450}1l
\put(620,50){\fbox{$(v,U)$}}
\putmorphism(-150,-400)(1,0)[(\tilde B, \tilde A)` `(g, \tilde A)]{500}1a
\putmorphism(-180,50)(0,-1)[\phantom{Y_2}``(u,\tilde A)]{450}1l
\putmorphism(450,50)(0,-1)[\phantom{Y_2}`(\tilde B', \tilde A)`(v,\tilde A)]{450}1r
\putmorphism(450,500)(0,-1)[\phantom{Y_2}`(B', \tilde A) `(B',U)]{450}1r
\put(0,260){\fbox{$(f,U)$}}
\putmorphism(-150,50)(1,0)[\phantom{(B, \tilde A)}``(f, \tilde A)]{500}1a
\putmorphism(1070,50)(0,-1)[\phantom{(B, A')}`(\tilde B', \tilde A)`(\tilde B',U)]{450}1r
\putmorphism(1070,500)(0,-1)[(B', A)`(\tilde B', A)`(v,A)]{450}1r
\putmorphism(450,-400)(1,0)[\phantom{(B, \tilde A)}``=]{500}1b
\put(0,-170){\fbox{$(\omega, \tilde A)$}}
\efig
$$
and
$$
\bfig
 \putmorphism(-150,500)(1,0)[(B,A)`(B,A) `=]{600}1a
 \putmorphism(550,500)(1,0)[` `(B,F)]{400}1a
\putmorphism(-180,500)(0,-1)[\phantom{Y_2}`(B, \tilde A) `(B,U)]{450}1l
\put(0,50){\fbox{$(u,U)$}}
\putmorphism(-150,-400)(1,0)[(\tilde B, \tilde A)` `=]{500}1a
\putmorphism(-180,50)(0,-1)[\phantom{Y_2}``(u,\tilde A)]{450}1l
\putmorphism(450,50)(0,-1)[\phantom{Y_2}`(\tilde B, \tilde A)`(\tilde B, U)]{450}1l
\putmorphism(450,500)(0,-1)[\phantom{Y_2}`(\tilde B, A) `(u,A)]{450}1l
\put(620,280){\fbox{$(u,F)$}}
\putmorphism(450,50)(1,0)[\phantom{(B, \tilde A)}``(\tilde B,F)]{500}1a
\putmorphism(1070,50)(0,-1)[\phantom{(B, A')}`(\tilde B, \tilde A')`(\tilde B,V)]{450}1r
\putmorphism(1070,500)(0,-1)[(B, A')`(\tilde B, A')`(u,A')]{450}1r
\putmorphism(450,-400)(1,0)[\phantom{(B, \tilde A)}``(\tilde B, G)]{500}1a
\put(620,-170){\fbox{$ (\tilde{B},\zeta)$ } } 
\efig=
\bfig
 \putmorphism(-150,500)(1,0)[(B,A)`(B,A') `(B,F)]{600}1a
 \putmorphism(450,500)(1,0)[\phantom{(B,A)}` `=]{450}1a
\putmorphism(-180,500)(0,-1)[\phantom{Y_2}`(B, \tilde A) `(B,U)]{450}1l
\put(620,50){\fbox{$(u,V)$}}
\putmorphism(-150,-400)(1,0)[(\tilde B, \tilde A)` `(\tilde B, G)]{500}1a
\putmorphism(-180,50)(0,-1)[\phantom{Y_2}``(u,\tilde A)]{450}1l
\putmorphism(450,50)(0,-1)[\phantom{Y_2}`(\tilde B, \tilde A')`(u,\tilde A')]{450}1r
\putmorphism(450,500)(0,-1)[\phantom{Y_2}`(B, \tilde A') `(B,V)]{450}1r
\put(0,260){\fbox{$(B,\zeta)$}}
\putmorphism(-150,50)(1,0)[\phantom{(B, \tilde A)}``(B,G)]{500}1a
\putmorphism(1070,50)(0,-1)[\phantom{(B, A')}`(\tilde B, \tilde A')`(\tilde B,V)]{450}1r
\putmorphism(1070,500)(0,-1)[(B, A')`(\tilde B, A')`(u,A')]{450}1r
\putmorphism(450,-400)(1,0)[\phantom{(B, \tilde A)}``=]{500}1b
\put(0,-170){\fbox{$(u,G)$}}
\efig
$$
for any 2-cells 
\begin{equation} \eqlabel{omega-zeta}
\bfig
\putmorphism(-150,50)(1,0)[B` B'`f]{450}1a
\putmorphism(-150,-300)(1,0)[\tilde B`\tilde B' `g]{440}1b
\putmorphism(-170,50)(0,-1)[\phantom{Y_2}``u]{350}1l
\putmorphism(280,50)(0,-1)[\phantom{Y_2}``v]{350}1r
\put(0,-140){\fbox{$\omega$}}
\efig
\quad\text{and}\quad
\bfig
\putmorphism(-150,50)(1,0)[A` A'`F]{450}1a
\putmorphism(-150,-300)(1,0)[\tilde A`\tilde A'. `G]{440}1b
\putmorphism(-170,50)(0,-1)[\phantom{Y_2}``U]{350}1l
\putmorphism(280,50)(0,-1)[\phantom{Y_2}``V]{350}1r
\put(0,-140){\fbox{$\zeta$}}
\efig
\end{equation}
in $\Bb$, respectively $\Aa$. 
\end{enumerate}
\end{prop}


In analogy to \cite[Section 4.2]{GPS} we set: 

\begin{defn} \delabel{H dbl}
The characterization in the above Proposition gives rise to an application from the Cartesian product of double categories 
$H: \Aa\times\Bb\to\Cc$ such that $H(A,-)=(-, A)$ and $H(-, B)=(B,-)$. Such an application of double categories we 
will call {\em cubical double functor}. 
\end{defn}

We may now describe a double category $\Aa\ot\Bb$ by reading off the structure of the image double category $F(\Aa)(\Bb)$  
for any double functor $F:\Aa\to\llbracket\Bb,\Aa\ot\Bb\rrbracket$ in the right hand-side of \equref{nat iso} using the 
characterization of a double functor before \prref{char df}. With the notation $F(x)(y)=(y,x)=:x\ot y$ for any 0-, 1h-, 1v- or 2-cells 
$x$ of $\Aa$ and $y$ of $\Bb$ we obtain that a double category $\Aa\ot\Bb$ consists of the following: \\
\u{objects}: $A\ot B$ for objects $A\in\Aa, B\in\Bb$; \\ 
\u{1h-cells}: $A\ot f, F\ot B$ and horizontal compositions of such (modulo associativity and unity constraints) obeying the following rules: 
$$(A\ot f')\odot(A\ot f)=A\ot (f'\odot f), \quad (F\s'\ot B)\odot(F\ot B)=(F\s'\odot F)\ot B, \quad A\ot 1_B=1_{A\ot B}=1_A\ot B$$
where $f,f'$ are 1h-cells of $\Bb$ and $F, F\s'$ 1h-cells of $\Aa$; \\
\u{1v-cells}: $A\ot u, U\ot B$ and vertical compositions of such obeying the following rules: 
$$(A\ot u')(A\ot u)=A\ot u'u, \quad (U'\ot B)(U\ot B)=U'U\ot B, \quad A\ot 1^B=1^{A\ot B}=1^A\ot B$$
where $u,u'$ are 1v-cells of $\Bb$ and $U,U'$ 1v-cells of $\Aa$; \\
\u{2-cells}: $A\ot\omega, \zeta\ot B$:
$$
\bfig
\putmorphism(-150,50)(1,0)[A\ot B`A\ot B'`A\ot f]{600}1a
\putmorphism(-150,-400)(1,0)[A\ot\tilde B`A'\ot\tilde B' `A\ot g]{640}1a
\putmorphism(-180,50)(0,-1)[\phantom{Y_2}``A\ot u]{450}1l
\putmorphism(450,50)(0,-1)[\phantom{Y_2}``A\ot v]{450}1r
\put(0,-180){\fbox{$A\ot \omega$}}
\efig
\qquad
\bfig
\putmorphism(-150,50)(1,0)[A\ot B`A'\ot B`F\ot B]{600}1a
\putmorphism(-150,-400)(1,0)[\tilde A\ot B`\tilde A'\ot B `G\ot B]{640}1a
\putmorphism(-180,50)(0,-1)[\phantom{Y_2}``U\ot B]{450}1l
\putmorphism(450,50)(0,-1)[\phantom{Y_2}``V\ot B]{450}1r
\put(0,-180){\fbox{$\zeta\ot B$}}
\efig
$$
where $\omega$ and $\zeta$ are as in \equref{omega-zeta},
and four types of 2-cells coming from the 2-cells of point 2. in \prref{char df}: 
vertically invertible globular 2-cell $F\ot f: (A'\ot f)\odot(F\ot B) \Rightarrow(F\ot B')\odot(A\ot f)$, 
horizontally invertible globular 2-cell $U\ot u: (\tilde A\ot u)(U\ot B)\Rightarrow(U\ot\tilde B)(A\ot u)$, 
2-cells $F\ot u$ and $U\ot f$, and the horizontal and vertical compositions of these 
(modulo associativity and unity constraints in the horizontal direction and the interchange law) 
subject to the rules induced by a), b) and c) of point 2. in \prref{char df} and the following ones: 
$$A\ot(\omega'\odot\omega)=(A\ot\omega')\odot(A\ot\omega), \quad (\zeta'\odot\zeta)\ot B=(\zeta'\ot B)\odot(\zeta\ot B),$$
$$A\ot(\omega'\omega)=(A\ot\omega')(A\ot\omega), \quad (\zeta'\zeta)\ot B=(\zeta'\ot B)(\zeta\ot B),$$
$$A\ot\Id_f=\Id_{A\ot f}, \quad \Id_F\ot B=\Id_{F\ot B}, \quad A\ot\Id^f=\Id^{A\ot f}, \quad \Id^F\ot B=\Id^{F\ot B}.$$

\subsection{A monoidal embedding of $(Gray,\ot)$ into $(Dbl,\ot)$ } \sslabel{mon embed}

Let $E: (Gray,\ot)\hookrightarrow(Dbl,\ot)$ denote the embedding functor which to a 2-category assigns a strict double category whose 
all vertical 1-cells are identities and whose 2-cells are vertically globular cells. Then $E$ is a left adjoint to the functor that to a strict double category 
assigns its underlying horizontal 2-category. Let us denote by $\C=(Gray,\ot)$ and by $\D=\im(E)\subseteq(Dbl,\ot)$, the image category by $E$, 
then the corestriction of $E$ to $\D$ is the identity functor 
\begin{equation}\eqlabel{id fun}
F:\C\to\D.
\end{equation}

In order to examine the monoidality of $F$ let us first consider 
an assignment $t: F(\A\ot\B)\to F(\A)* F(\B)$ for two 2-categories $\A$ and $\B$, where $*$ denotes some monoidal product in the category of 
strict double categories which a priori could be the Cartesian one or the one from the monoidal category $(Dbl,\ot)$. 

Observe that given 1-cells $f:A\to A'$ in $\A$ and $g: B\to B'$ in $\B$ the composition 1-cells 
$(f\ot B')\odot(A\ot g)$ and $(A'\ot g)\odot(f\ot B)$ in $\A\ot\B$ are not equal both in $(Gray,\ot)$ and in $(Dbl,\ot)$. 
This means that their images $F\big((f\ot B')\odot(A\ot g)\big)$ and $F\big((A'\ot g)\odot(f\ot B)\big)$ are different as 1h-cells of the double category 
$F(\A\ot\B)$. Now if we map these two images by $t$ into the Cartesian product $F(\A)\times F(\B)$, we will get in both cases the 1h-cell 
$(f,g)$. Then $t$ with the codomain in the Cartesian product is a bad candidate for the monoidal structure of the identity functor $F$. 
This shows that the Cartesian monoidal product on the category of strict double categories is not a good choice for a monoidal structure 
if one wants to embed the Gray category of 2-categories into the latter category. In contrast, if the codomain of $t$ is the monoidal product 
of $(Dbl,\ot)$, we see that $t$ is identity on these two 1-cells. 

Similar considerations and comparing the monoidal product from \cite[Theorem I.4.9]{Gray} 
in $(Gray,\ot)$ to the one after \deref{H dbl} above in $(Dbl,\ot)$, show that for the candidate for (the one part of) a monoidal structure 
on the identity functor $F$ we may take the identity $s=\Id: F(\A\ot\B)\to F(\A)\ot F(\B)$, and that it is indeed a strict double functor of 
strict double categories. 
For the other part of a monoidal structure on $F$, namely $s_0: F(*_2)\to *_{Dbl}$, where $*_2$ is the trivial 2-category with a single object, 
and similarly $*_{Dbl}$ is the trivial double category, it is clear that we again may take identity. The hexagonal and two square relations 
for the monoidality of the functor $(F, s, s_0)$ come down to checking if 
$$F(\alpha_1)=\alpha_2, \quad F(\lambda_1)=\lambda_2, \quad\text{and}\quad F(\rho_1)=\rho_2$$
where the monoidal constraints with indexes 1 are those from $\C$ and those with indexes 2 from $\D$. 

\medskip

Given any monoidal closed category $(\M, \ot, I, \alpha, \lambda, \rho)$ in \cite[Section 4.1]{Gabi} the author constructs a mate 
\begin{equation} \eqlabel{mate a}
a^C_{A,B}: [A\ot B, C]\to[A,[B,C]]
\end{equation} 
for $\alpha$ under the adjunctions $(-\ot X, [X,-])$, for $X$ taking to be $A,B$ and $A\ot B$, 
and then she constructs a mate of $a$: 
\begin{equation} \eqlabel{l-a}
l^C_{A,B}: \big([A,B]\stackrel{[\Epsilon^C_A,1]}{\rightarrow} [[C,A]\ot C, B] \stackrel{a^B}{\rightarrow} [[C,A],[C,B]]\big)
\end{equation}
where $\Epsilon$ is the counit of the adjunction. By the mate correspondence one gets: 
\begin{equation} \eqlabel{a-l}
a^C_{A,B}: \big( [A\ot B, C]\stackrel{l^B_{A\ot B,C} }{\rightarrow} [[B,A\ot B],[B,C]] \stackrel{[\eta^B_A,1]}{\rightarrow} [A,[B,C]]\big)
\end{equation}
where $\eta$ is the unit of the adjunction. As above for the monoidal constraints, let us write $l_i, a_i, \Epsilon_i$ and $\eta_i$ 
with $i=1,2$ for the corresponding 2-functors in $\C$ (with $i=1$), respectively double functors in $\D$ (with $i=2$). 
Comparing the description of $l_1$ from \cite[Section 4.7]{Gabi}, obtained as indicated above: $\alpha_1$ determines $a_1$, which in turn determines 
$l_1$ by \equref{l-a}, to the construction of $l_2$ in \cite[Section 2.4]{Gabi}, on one hand, and the well-known 2-category $[\A,\B]=\Fun(\A,\B)$ 
of 2-functors between 2-categories $\A$ and $\B$, pseudo natural transformations and modifications (see {\em e.g.} \cite[Section 5.1]{GP:Fram}, \cite{Ben}) 
to the definition of the double category $\llbracket\Aa,\Bb\rrbracket$ from \cite[Section 2.2]{Gabi} for double categories $\Aa,\Bb$, 
on the other hand, one immediately obtains:

\begin{lma}
For two 2-categories $\A$ and $\B$, functor $F$ from \equref{id fun} and $l_1$ and $l_2$ as above, it is: 
\begin{itemize}
\item $F(l_1)=l_2$,
\item $F([\A,\B])=\llbracket F(\A),F(\B)\rrbracket$.
\end{itemize}
\end{lma}

Because of the extent of the definitions and the detailed proofs we will omit them, we only record that the counits of the adjunctions 
$\Epsilon_i, i=1,2$ are basically given as evaluations and it is $F(\Epsilon_1)=\Epsilon_2$. The counits $\eta_i, i=1,2$ are defined in the natural way 
and it is also clear that $F(\eta_1)=\eta_2$. Now by the above Lemma and \equref{a-l} we get: $F(a_1)=a_2$. Then from the next Lemma 
we get that $F(\alpha_1)=\alpha_2$:

\begin{lma}
Suppose that there is an embedding functor $F:\C\to\D$ between monoidal closed categories which fulfills:
\begin{enumerate} [a)]
\item $F(X)\ot F(Y)=F(X\ot Y)$ for objects $X,Y\in\C$,
\item $F([X,Y])=[F(X),F(Y)]$,
\item $F(a_\C)=a_\D$, where the respective  $a$'s are given through \equref{mate a},
\end{enumerate} 
then it is $F(\alpha_\C)=\alpha_\D$, being $\alpha$'s the respective associativity constraints. 
\end{lma}

\begin{proof}
By the mate construction in \equref{mate a} we have a commuting diagram:
$$
\bfig
\putmorphism(-150,50)(1,0)[\C(A\ot(B\ot C),D)`\C(A, [B\ot C,D])`\iso]{1000}1a
\putmorphism(-150,-400)(1,0)[\C((A\ot B)\ot C,D) `\C(A,[B,[C,D]]) `\iso]{1040}{-1}a
\putmorphism(-120,50)(0,-1)[\phantom{Y_2}``\C(\alpha_\C,id)]{450}1l
\putmorphism(870,50)(0,-1)[\phantom{Y_2}``\C(id,a_\C)]{450}1r
\put(300,-180){\fbox{$(1)$}}
\efig
$$
Applying $F$ to it, by the assumptions $a)$ and $b)$ we obtain a commuting diagram:
$$
\bfig
\putmorphism(-150,50)(1,0)[\D(F(A)\ot(F(B)\ot F(C)),F(D))`\D(F(A), [F(B)\ot F(C),F(D)])`\iso]{1600}1a
\putmorphism(-150,-400)(1,0)[\D((F(A)\ot F(B))\ot F(C),F(D)) `\D(F(A),[F(B),[F(C),F(D)]]) `\iso]{1640}{-1}a
\putmorphism(-80,50)(0,-1)[\phantom{Y_2}``\D(F(\alpha_\C),id)]{450}1l
\putmorphism(1490,50)(0,-1)[\phantom{Y_2}``\D(id,F(a_\C))]{450}1r
\put(580,-180){\fbox{$(2)$}}
\efig
$$
Now by the assumption $c)$ and the mate construction in \equref{mate a} it follows $F(\alpha_\C)=\alpha_\D$. 
\end{proof}

So far we have proved that for the categories $\C$ and $\D$ as in \equref{id fun} we have $F(\alpha_\C)=\alpha_\D$. 
For the unity constraints $\rho_i, \lambda_i, i=1,2$ in the cases of both categories 
(see Sections 3.3 and 4.7 of \cite{Gabi}) it is: 
$$\rho_i^A=\Epsilon_{i,A}^{1_i}\comp(c_i\ot id_{1_i})\quad\text{and}\quad\lambda_i^A=\Epsilon_{i,A}^A\comp(1_A\ot id_A)$$
where $c_i: A\to[1_i,A]$ is the canonical isomorphism and $1_A: 1_i\to[A,A]$ the 2-functor (pseudofunctor) 
sending the single object of the terminal 2-category $1_1$ (double category $1_2$) to the identity 2-functor (pseudofunctor) $A\to A$ 
(here we have used the same notation for objects $A$ and inner home objects both in $\C$ and in $\D$). Then it is clear that also 
$F(\lambda_1)=\lambda_2$ and $F(\rho_1)=\rho_2$, which finishes the proof that the functor $F:\C\to\D$ is a monoidal embedding. 

\begin{prop} \prlabel{mon embed}
The category $(Gray,\ot)$ monoidally embeds into $(Dbl,\ot)$, where the respective monoidal structures are those from \cite{Gray} and \cite{Gabi}. 
Consequently, a monoid in $(Gray,\ot)$ is a monoid in $(Dbl,\ot)$, and a monoidal bicategory can be seen as a monoidal double category 
with respect to B\"ohm's tensor product. 
\end{prop}

\bigskip

\subsection{A monoid in $(Dbl,\ot)$ } \sslabel{Gabi's monoid}

In \cite[Section 4.3]{Gabi} a complete list of data 
and conditions defining the structure of a monoid $\Aa$ in $(Dbl,\ot)$ is given. As a part of this structure we have the following occurrence. 
As a monoid in $(Dbl,\ot)$, we have that $\Aa$ is equipped with a strict double functor 
$m: \Aa\ot\Aa\to\Aa$. Since in the monoidal product $\Aa\ot\Aa$ 
horizontal and vertical 1-cells of the type $(f\ot 1)(1\ot g)$ and $(1\ot g)(f\ot 1)$ are not equal (here juxtaposition denotes the 
corresponding composition of the 1-cells), one can fix a choice for how to define an image 1-cell $f\oast g$ by $m$ 
(either $m\big((f\ot 1)(1\ot g)\big)$ or $m\big((1\ot g)(f\ot 1)\big)$). Any of the two choices 
yields a {\em double pseudo} functor from the {\em Cartesian} product double category 
\begin{equation} \eqlabel{oast}
\oast:\Aa\times\Aa\to\Aa.
\end{equation} 
Let us see this. 
If we take two pairs of horizontal 1-cells $(h,k),(h',k')$ in $A\times A$, for the images under $\oast$, fixing the second choice above, 
we get $(h'h)\oast(k'k)=m\big((1\ot k'k)(h'h\ot 1)\big)=
m(1\ot k')m(1\ot k)m(h'\ot 1)m(h\ot 1)$, whereas 
$(h'\oast k')(h\oast k)=m\big((1\ot k')(h'\ot 1)\big)m\big((1\ot k)(h\ot 1)\big)=m(1\ot k')m(h'\ot 1)m(1\ot k)m(h\ot 1)$. 
So, the two images differ in the flip on the middle factors. The analogous situation happens on the vertical level, thus the functor 
$\oast$ preserves both vertical and horizontal 1-cells only up to an isomorphism 2-cell. This makes it a double pseudo functor due to 
\cite[Definition 6.1]{Shul1}. 

As outlined at the end of \cite[Section 4.3]{Gabi}, monoids in (the non-Cartesian monoidal category) $(Dbl,\ot)$ are monoids in the Cartesian 
monoidal category $(Dbl,\oast)$ of strict double categories and double pseudo functors (in the sense of \cite{Shul1}).

\subsection{Monoidal double categories as intercategories and beyond} \sslabel{beyond}

A monoidal double category in \cite{Shul} is a pseudomonoid in the 2-category $PsDbl$ of pseudo double categories, pseudo double functors 
and vertical transformations, seen as a monoidal 2-category with the Cartesian product. As such it is a particular case of an intercategory \cite{GP}. 

An intercategory is a pseudocategory ({\em i.e.} weakly internal category) in the 2-category $LxDbl$ of pseudo 
double categories, lax double functors and horizontal transformations. 
It consists of pseudodouble categories $\Dd_0$ and $\Dd_1$ and pseudo double functors 
$S,T: \Dd_1\to\Dd_0, U: \Dd_0\to\Dd_1, M:\Dd_1\times_{\Dd_0}\Dd_1\to\Dd_1$ (where $S$ and $T$ are strict) satisfying the corresponding properties. 
One may denote this structure formally by 
$$
 \Dd_1\times_{\Dd_0}\Dd_1\triarrows \Dd_1\tripplearrow \Dd_0
$$
where $\Dd_1\times_{\Dd_0}\Dd_1$ is a certain 2-pullback and the additional two arrows $\Dd_1\times_{\Dd_0}\Dd_1\to\Dd_1$ stand for the two projections. 
When $\Dd_0$ is the trivial double category 1 (the terminal object in $LxDbl$, consisting of a single object *), setting $\Dd_1=\Dd$ one has that 
$\Dd\times\Dd$ is the Cartesian product of pseudo double categories. 

As a pseudomonoid in $PsDbl$, a monoidal double category of Shulman consists of a pseudo double category $\Dd$ and pseudo 
double functors $M: \Dd\times\Dd\to\Dd$ and $U: 1\to\Dd$ which satisfy properties that make $\Dd$ precisely an intercategory 
$$ \Dd\times\Dd\triarrows \Dd\tripplearrow 1,$$
as explained in \cite[Section 3.1]{GP:Fram}. 

\medskip

Nevertheless, if one would try to make a monoid in $(Dbl,\ot)$, which is a monoid in $(Dbl,\oast)$, into an intercategory, one would 
need a lax double functor on the Cartesian product $\Dd\times\Dd$ (the pullback). However, as we showed in the last subsection, on the 
Cartesian product one has a {\em double pseudo} functor $\oast$. So, as observed at the end of \cite[Section 4.3]{Gabi}, there seems 
to be no easy way to regard a monoid in $(Dbl,\ot)$ as a suitably degenerated intercategory. Motivated by this and \prref{mon embed}, 
we want now to upgrade the Cartesian monoidal category $(Dbl,\oast)$ from the end of \ssref{Gabi's monoid} to a 2-category, so to obtain 
an intercategory-type notion which would include monoidal double categories due to B\"ohm. 

In the next section we introduce 2-cells and ``unfortunately'' rather than a 2-category we will obtain a 
tricategory of strict double categories whose 1-cells are double pseudo functors of Shulman. Since 1-cells of the 2-category $LxDbl$ 
(considered by Grandis and Par\'e to define intercategories) are lax double functors (they are lax in one and strict in the other direction), 
and our 1-cells are double pseudo functors, that is, they are given by isomorphisms in both directions, we can not 
generalize intercategories this way, rather, we will propose an alternative notion to intercategories which will include 
most of the examples of intercategories treated in \cite{GP:Fram}, but not duoidal categories, as they rely on lax functors, rather than pseudo ones.




\section{Tricategory of strict double categories and double pseudo functors} 

Let us denote the tricategory from the title of this section by $\DblPs$. 
As we are going to use double pseudo functors of \cite{Shul1}, which preserve compositions of 1-cells and identity 1-cells in 
both horizontal and vertical direction up to an isomorphism, we have to introduce accordingly horizontal and vertical transformations. 
A pair consisting of a horizontal and a vertical pseudonatural transformation, which we define next, will be a part of the data 
constituting a 2-cell of the tricategory $\DblPs$. 
Note that while $PsDbl$ usually denotes a category or a 2-category of {\em pseudo} double categories and pseudo double functors, 
that is, in which 0- and 1-cells are weakened, in the notation $\DblPs$ we wish to stress that both 1- and 2-cells are weakened in both 
directions so to deal with {\em double pseudo} functors.

\subsection{Towards the 2-cells}

For the structure of a double pseudo functor we use the same notation as in \cite[Definition 6.1]{Shul1} with the only difference 
that 0-cells we denote by $A,B...$ and 1v-cells by $u,v...$. To simplify the notation, we will denote by juxtaposition the 
compositions of both 1h- and 1v-cells, from the notation of the 1-cells it will be clear which kind of 1-cells and therefore 
composition is meant. Let $\Aa,\Bb,\Cc$ be strict double categories throughout. 

\begin{defn} \delabel{hor psnat tr}
A {\em horizontal pseudonatural transformation} between double pseudo functors $F,G: \Aa\to\Bb$ consists of the following:
\begin{itemize}
\item for every 0-cell $A$ in $\Aa$ a 1h-cell $\alpha(A):F(A)\to G(A)$ in $\Bb$,
\item for every 1v-cell $u:A\to A'$ in $\Aa$ a 2-cell in $\Bb$:
$$
\bfig
\putmorphism(-150,50)(1,0)[F(A)`G(A)`\alpha(A)]{560}1a
\putmorphism(-150,-320)(1,0)[F(A')`G(A')`\alpha(A')]{600}1a
\putmorphism(-180,50)(0,-1)[\phantom{Y_2}``F(u)]{370}1l
\putmorphism(410,50)(0,-1)[\phantom{Y_2}``G(u)]{370}1r
\put(30,-110){\fbox{$\alpha_u$}}
\efig
$$
\item 
for every 1h-cell $f:A\to B$  in $\Aa$ there is a 2-cell in $\Bb$:
$$
\bfig
 \putmorphism(-170,500)(1,0)[F(A)`F(B)`F(f)]{540}1a
 \putmorphism(360,500)(1,0)[\phantom{F(f)}`G(B) `\alpha(B)]{560}1a
 \putmorphism(-170,120)(1,0)[F(A)`G(A)`\alpha(A)]{540}1a
 \putmorphism(360,120)(1,0)[\phantom{G(B)}`G(A) `G(f)]{560}1a
\putmorphism(-180,500)(0,-1)[\phantom{Y_2}``=]{380}1r
\putmorphism(940,500)(0,-1)[\phantom{Y_2}``=]{380}1r
\put(280,310){\fbox{$\delta_{\alpha,f}$}}
\efig
$$
\end{itemize}
so that the following are satisfied:
\begin{enumerate}
\item pseudonaturality of 2-cells:
for every 2-cell in $\Aa$
$\bfig
\putmorphism(-150,50)(1,0)[A` B`f]{400}1a
\putmorphism(-150,-270)(1,0)[A'`B' `g]{400}1b
\putmorphism(-170,50)(0,-1)[\phantom{Y_2}``u]{320}1l
\putmorphism(250,50)(0,-1)[\phantom{Y_2}``v]{320}1r
\put(0,-140){\fbox{$a$}}
\efig$ 
the following identity in $\Bb$ must hold:
$$
\bfig
\putmorphism(-150,500)(1,0)[F(A)`F(B)`F(f)]{600}1a
 \putmorphism(450,500)(1,0)[\phantom{F(A)}`G(B) `\alpha(B)]{640}1a

 \putmorphism(-150,50)(1,0)[F(A')`F(B')`F(g)]{600}1a
 \putmorphism(450,50)(1,0)[\phantom{F(A)}`G(B') `\alpha(B')]{640}1a

\putmorphism(-180,500)(0,-1)[\phantom{Y_2}``F(u)]{450}1l
\putmorphism(450,500)(0,-1)[\phantom{Y_2}``]{450}1r
\putmorphism(300,500)(0,-1)[\phantom{Y_2}``F(v)]{450}0r
\putmorphism(1100,500)(0,-1)[\phantom{Y_2}``G(v)]{450}1r
\put(0,260){\fbox{$F(a)$}}
\put(700,270){\fbox{$\alpha_v$}}

\putmorphism(-150,-400)(1,0)[F(A')`G(A') `\alpha(A')]{640}1a
 \putmorphism(450,-400)(1,0)[\phantom{A'\ot B'}` G(B') `G(g)]{680}1a

\putmorphism(-180,50)(0,-1)[\phantom{Y_2}``=]{450}1l
\putmorphism(1120,50)(0,-1)[\phantom{Y_3}``=]{450}1r
\put(320,-200){\fbox{$\delta_{\alpha,g}$}}

\efig
\quad=\quad
\bfig
\putmorphism(-150,500)(1,0)[F(A)`F(B)`F(f)]{600}1a
 \putmorphism(450,500)(1,0)[\phantom{F(A)}`G(B) `\alpha(B)]{680}1a
 \putmorphism(-150,50)(1,0)[F(A)`G(A)`\alpha(A)]{600}1a
 \putmorphism(450,50)(1,0)[\phantom{F(A)}`G(B) `G(f)]{680}1a

\putmorphism(-180,500)(0,-1)[\phantom{Y_2}``=]{450}1r
\putmorphism(1100,500)(0,-1)[\phantom{Y_2}``=]{450}1r
\put(350,260){\fbox{$\delta_{\alpha,f}$}}
\put(650,-180){\fbox{$G(a)$}}

\putmorphism(-150,-400)(1,0)[F(A')`G(A') `\alpha(A')]{640}1a
 \putmorphism(490,-400)(1,0)[\phantom{F(A')}` G(B') `G(g)]{640}1a

\putmorphism(-180,50)(0,-1)[\phantom{Y_2}``F(u)]{450}1l
\putmorphism(450,50)(0,-1)[\phantom{Y_2}``]{450}1l
\putmorphism(610,50)(0,-1)[\phantom{Y_2}``G(u)]{450}0l 
\putmorphism(1120,50)(0,-1)[\phantom{Y_3}``G(v)]{450}1r
\put(40,-180){\fbox{$\alpha_u$}} 

\efig
$$

\item vertical functoriality: for any composable 1v-cells $u$ and $v$ in $\Aa$:
$$
\bfig
 \putmorphism(-150,500)(1,0)[F(A)`F(A) `=]{500}1a
 \putmorphism(450,500)(1,0)[` `\alpha(A)]{380}1a
\putmorphism(-180,500)(0,-1)[\phantom{Y_2}`F(A') `F(u)]{450}1l
\put(20,250){\fbox{$F^{vu}$}}
\putmorphism(-150,-400)(1,0)[F(A'')`F(A'') `=]{500}1a
\putmorphism(-180,50)(0,-1)[\phantom{Y_2}``F(v)]{450}1l
\putmorphism(380,500)(0,-1)[\phantom{Y_2}` `F(vu)]{900}1l
\put(600,45){\fbox{$\alpha^{vu}$} }
\putmorphism(940,500)(0,-1)[G(A)`G(A'')`G(vu)]{900}1r
\putmorphism(450,-400)(1,0)[``\alpha(A'')]{360}1a
\efig= 
\bfig
 \putmorphism(-150,500)(1,0)[F(A)`G(A)  `\alpha(A)]{600}1a
 \putmorphism(450,500)(1,0)[\phantom{(B,A)}` `=]{450}1a
\putmorphism(-180,500)(0,-1)[\phantom{Y_2}`F(A') `F(u)]{450}1l
\put(620,50){\fbox{$G^{vu}$}}
\putmorphism(-150,-400)(1,0)[F(A'')` `\alpha(A'')]{500}1a
\putmorphism(-180,50)(0,-1)[\phantom{Y_2}``F(v)]{450}1l
\putmorphism(450,50)(0,-1)[\phantom{Y_2}`G(A'')`G(v)]{450}1r
\putmorphism(450,500)(0,-1)[\phantom{Y_2}`G(A') `G(u)]{450}1r
\put(40,260){\fbox{$\alpha_u$}}
\putmorphism(-150,50)(1,0)[\phantom{(B, \tilde A)}``\alpha(A')]{500}1a
\putmorphism(1000,500)(0,-1)[G(A)`G(A'')`G(vu)]{900}1r
\putmorphism(480,-400)(1,0)[\phantom{F(A)}`\phantom{F(A)}`=]{500}1b
\put(40,-170){\fbox{$\alpha_v$}}
\efig
$$
and  
$$
\bfig
 \putmorphism(-150,250)(1,0)[F(A)`F(A)`=]{600}1a 
 \putmorphism(450,250)(1,0)[\phantom{A\ot B}`G(A) `\alpha(A)]{680}1a
\putmorphism(500,250)(0,-1)[\phantom{Y_2}``F(id_A)]{450}1l

 \putmorphism(-150,-200)(1,0)[F(A)`F(A)`=]{600}1a 
 \putmorphism(450,-200)(1,0)[\phantom{A\ot B}`G(A) `\alpha(A)]{680}1a
\putmorphism(-180,250)(0,-1)[\phantom{Y_2}``=]{450}1r
\putmorphism(1100,250)(0,-1)[\phantom{Y_2}``G(id_A)]{450}1r
\put(0,0){\fbox{$F^A$}}
\put(700,30){\fbox{$\alpha_{id_A}$}}
\efig
\quad=
\bfig
 \putmorphism(-100,250)(1,0)[F(A)`G(A)`\alpha(A)]{550}1a
 \putmorphism(450,250)(1,0)[\phantom{A\ot B}`G(A) `=]{680}1a
\putmorphism(500,250)(0,-1)[\phantom{Y_2}``=]{450}1l

 \putmorphism(-150,-200)(1,0)[F(A)`G(A)`\alpha(A)]{600}1a
 \putmorphism(450,-200)(1,0)[\phantom{F( B)}`G(A) `=]{680}1a 
\putmorphism(-180,250)(0,-1)[\phantom{Y_2}``=]{450}1r
\putmorphism(1100,250)(0,-1)[\phantom{Y_2}``G(id_A)]{450}1r
\put(0,0){\fbox{$\Id_{\alpha(A)}$}}
\put(700,30){\fbox{$G^A$}}
\efig
$$

\item horizontal functoriality for $\delta_{\alpha,-}$: for any composable 1h-cells $f$ and $g$ in $\Aa$ the 2-cell 
$\delta_{\alpha,gf}$ is given by: 
$$
\bfig
 \putmorphism(-130,500)(1,0)[F(A)`F(C)`F(gf)]{580}1a
 \putmorphism(450,500)(1,0)[\phantom{F(B)}`G(C) `\alpha(C)]{580}1a
 \putmorphism(-130,135)(1,0)[F(A)`G(A)`\alpha(A)]{580}1a
 \putmorphism(450,135)(1,0)[\phantom{F(B)}`G(C) `G(gf)]{580}1a
\putmorphism(-180,520)(0,-1)[\phantom{Y_2}``=]{380}1r
\putmorphism(1030,520)(0,-1)[\phantom{Y_2}``=]{380}1r
\put(330,300){\fbox{$\delta_{\alpha,gf}$}}
\efig= 
\bfig
\putmorphism(-150,900)(1,0)[F(A)`F(C)`F(gf)]{1200}1a

 \putmorphism(-150,450)(1,0)[F(A)`F(B)`F(f)]{600}1a
 \putmorphism(450,450)(1,0)[\phantom{F(B)}`F(C) `F(g)]{680}1a
 \putmorphism(1120,450)(1,0)[\phantom{F(B)}`G(C) `\alpha(C)]{600}1a

\putmorphism(-180,900)(0,-1)[\phantom{Y_2}``=]{450}1r
\putmorphism(1060,900)(0,-1)[\phantom{Y_2}``=]{450}1r
\put(350,650){\fbox{$F_{gf}$}}
\put(1000,200){\fbox{$\delta_{\alpha,g}$}}

  \putmorphism(-150,0)(1,0)[F(A)` F(B) `F(f)]{600}1a
\putmorphism(450,0)(1,0)[\phantom{F(A)}` G(B) `\alpha(B)]{680}1a
 \putmorphism(1120,0)(1,0)[\phantom{F(A)}`G(C) ` G(g)]{620}1a

\putmorphism(450,450)(0,-1)[\phantom{Y_2}``=]{450}1l
\putmorphism(1710,450)(0,-1)[\phantom{Y_2}``=]{450}1r

 \putmorphism(-150,-450)(1,0)[F(A)`G(A)`\alpha(A)]{600}1a
 \putmorphism(450,-450)(1,0)[\phantom{F(B)}`G(B) `G(f)]{680}1a
 \putmorphism(1120,-450)(1,0)[\phantom{F(B)}`G(C) `G(g)]{620}1a

\putmorphism(-180,0)(0,-1)[\phantom{Y_2}``=]{450}1r
\putmorphism(1040,0)(0,-1)[\phantom{Y_2}``=]{450}1r
\put(350,-240){\fbox{$\delta_{\alpha,f}$}}
\put(1000,-660){\fbox{$G^{-1}_{gf}$}}

 \putmorphism(450,-900)(1,0)[G(A)` G(C) `G(gf)]{1300}1a

\putmorphism(450,-450)(0,-1)[\phantom{Y_2}``=]{450}1l
\putmorphism(1750,-450)(0,-1)[\phantom{Y_2}``=]{450}1r
\efig
$$ 
and unit:
$$
\bfig
 \putmorphism(-150,420)(1,0)[F(A)`F(A)`=]{500}1a
\putmorphism(-180,420)(0,-1)[\phantom{Y_2}``=]{370}1l
\putmorphism(320,420)(0,-1)[\phantom{Y_2}``=]{370}1r
 \putmorphism(-150,50)(1,0)[F(A)`F(A)`F(id_A)]{500}1a
 \put(-80,250){\fbox{$(F_A)^{-1}$}} 
\putmorphism(330,50)(1,0)[\phantom{F(A)}`G(A) `\alpha(A)]{560}1a
 \putmorphism(-170,-350)(1,0)[F(A)`G(A)`\alpha(A)]{520}1a
 \putmorphism(350,-350)(1,0)[\phantom{F(A)}`G(A) `G(id_A)]{560}1a

\putmorphism(-180,50)(0,-1)[\phantom{Y_2}``=]{400}1r
\putmorphism(910,50)(0,-1)[\phantom{Y_2}``=]{400}1r
\put(240,-170){\fbox{$\delta_{\alpha,id_A}$}}
\put(540,-550){\fbox{$G_A$}}

\putmorphism(350,-350)(0,-1)[\phantom{Y_2}``=]{350}1l
\putmorphism(920,-350)(0,-1)[\phantom{Y_3}``=]{350}1r
 \putmorphism(350,-700)(1,0)[G(A)` G(A) `=]{580}1a
\efig
\quad=\quad
\bfig
\putmorphism(-150,50)(1,0)[F(A)` G(A) `\alpha(A)]{450}1a
\putmorphism(-150,-300)(1,0)[F(A)` G(A) `\alpha(A)]{440}1b
\putmorphism(-170,50)(0,-1)[\phantom{Y_2}``=]{350}1l
\putmorphism(280,50)(0,-1)[\phantom{Y_2}``=]{350}1r
\put(-80,-140){\fbox{$\Id_{\alpha(A)}$}}
\efig
$$
\end{enumerate}
\end{defn}

\begin{rem} \rmlabel{hor tr as subcase}
Recall horizontal transformations from \cite[Section 2.2]{GP:Adj} and their version when $R=S=\Id, 
\Aa=\Bb, \Cc=\Dd$, which constitute the 2-cells of the 2-category $\mathcal{L}x\mathcal{D}bl$ from \cite{GP}. Considering them 
as acting between pseudo double functors of strict double categories, one has that they are particular cases of our 
horizontal pseudonatural transformations so that the 2-cells $\delta_{\alpha,f}$ and $F_{gf}, F_A$ are identities. 
Similarly, vertical transformations from \cite{Shul} and \cite{GG} are particular cases of the vertical analogon of 
\deref{hor psnat tr}. 
\end{rem}

\begin{rem}
Our above definition generalizes also horizontal pseudotransformations from \cite[Section 2.2]{Gabi} to the case of {\em double 
pseudo functors} instead of strict double functors. In \cite[Section 2.2]{Gabi} horizontal pseudotransformations appear 
as 1h-cells of the double category $\llbracket\Aa,\Bb\rrbracket$ defined therein, which we mentioned in \ssref{mon str}. 
On the other hand, our definition of horizontal pseudotransformations differs from {\em strong horizontal transformations} from 
\cite[Section 7.4]{GP:Limits} in that therein the authors work with {\em pseudo} double categories (whereas we work here with strict ones), 
and they work with double functors which are lax in one direction and strict in the other, whereas we work with double functors 
which are pseudo in both directions. 
\end{rem}

The following results are straightforwardly proved:

\begin{lma} \lelabel{delta F(alfa)}
For a double pseudofunctor $H:\Bb\to\Cc$ and a horizontal pseudonatural transformation $\alpha: F\to G$ of double pseudofunctors 
$F,G: \Aa\to\Bb$, $H(\alpha)$ is a horizontal pseudonatural transformation with $(H(\alpha))_u=H(\alpha_u)$ and $\delta_{H(\alpha),f}$ 
satisfying:
$$
\bfig
\putmorphism(-150,500)(1,0)[HF(A)`HG(B)`H(\alpha(B)F(f))]{1240}1a

 \putmorphism(-150,50)(1,0)[HF(A)`HF(B)`HF(f)]{600}1a
 \putmorphism(450,50)(1,0)[\phantom{F(A)}`HG(B) `H(\alpha(B))]{640}1a

\putmorphism(-180,500)(0,-1)[\phantom{Y_2}``=]{450}1l
\putmorphism(1100,500)(0,-1)[\phantom{Y_2}``=]{450}1r
\put(320,270){\fbox{$H_{\alpha(B),F(f)}$}}

\putmorphism(-150,-400)(1,0)[HF(A)`HG(A) `H(\alpha(A))]{640}1a
 \putmorphism(450,-400)(1,0)[\phantom{A'\ot B'}` HG(B) `HG(f)]{680}1a

\putmorphism(-180,50)(0,-1)[\phantom{Y_2}``=]{450}1l
\putmorphism(1120,50)(0,-1)[\phantom{Y_3}``=]{450}1r
\put(320,-200){\fbox{$\delta_{H(\alpha),f}$}}

\efig
\quad=\quad
\bfig
\putmorphism(-580,500)(0,-1)[\phantom{Y_2}``=]{450}1l
\putmorphism(1520,500)(0,-1)[\phantom{Y_2}``=]{450}1r
\putmorphism(-550,500)(1,0)[HF(A)`\phantom{HF(A)}`=]{550}1a
\putmorphism(-550,50)(1,0)[HF(A)`\phantom{HF(A)}`=]{550}1a
\put(-500,260){\fbox{$H^{F(A)}$}}
\put(1060,260){\fbox{$(H^{G(B)})^{-1}$}}

 \putmorphism(980,500)(1,0)[\phantom{HF(A)}`HG(B) `=]{550}1a
 \putmorphism(1000,50)(1,0)[\phantom{HF(A)}`HG(B) `=]{550}1a

\putmorphism(-20,500)(1,0)[HF(A)`HG(B)`H(\alpha(B)F(f))]{1000}1a
 \putmorphism(-20,50)(1,0)[HF(A)`HG(B)`H\big(G(f)\alpha(A)\big)]{1000}1a

\putmorphism(-80,500)(0,-1)[\phantom{Y_2}``]{450}1r
\putmorphism(-100,500)(0,-1)[\phantom{Y_2}``H(id)]{450}0r

\putmorphism(1020,500)(0,-1)[\phantom{Y_2}``]{450}1l
\putmorphism(1050,500)(0,-1)[\phantom{Y_2}``H(id)]{450}0l

\put(300,310){\fbox{$H(\delta_{\alpha,f})$}}
\put(300,-180){\fbox{$H_{G(f),\alpha(A)}$}}

\putmorphism(-150,-400)(1,0)[HF(A)`HG(A) `H(\alpha(A))]{600}1a
 \putmorphism(450,-400)(1,0)[\phantom{HF(A)}` HG(B). `HG(f)]{600}1a

\putmorphism(-80,50)(0,-1)[\phantom{Y_2}``=]{450}1l
\putmorphism(1020,50)(0,-1)[\phantom{Y_3}``=]{450}1r

\efig
$$
\end{lma}

\begin{lma} \lelabel{horiz comp hor.ps.tr.}
Horizontal composition of two horizontal pseudonatural transformations $\alpha_1: F\Rightarrow G: \Aa\to\Bb$ and $\beta_1: F\s'\Rightarrow G':\Bb\to\Cc$, denoted by $\beta_1\comp\alpha_1$, is well-given by: 
\begin{itemize}
\item for every 0-cell $A$ in $\Aa$ a 1h-cell in $\Cc$:
$$(\beta_1\comp\alpha_1)(A)=\big( F\s'F(A)\stackrel{F\s'(\alpha_1(A))}{\longrightarrow}F\s'G(A) \stackrel{\beta_1(G(A))}{\longrightarrow} G'G(A) \big),$$ 
\item for every 1v-cell $u:A\to A'$ in $\Aa$ a 2-cell in $\Cc$:
$$(\beta_1\comp\alpha_1)_u=
\bfig
\putmorphism(-320,500)(1,0)[F\s'F(A)`F\s'G(A)`F\s'(\alpha_1(A))]{770}1a
 \putmorphism(500,500)(1,0)[\phantom{F(A)}`G'G(A) `\beta_1(G(A))]{720}1a

 \putmorphism(-320,50)(1,0)[F\s'F(A')`F\s'G(A')`F\s'(\alpha_1(A'))]{770}1a
 \putmorphism(520,50)(1,0)[\phantom{F(A)}`G'G(A') `\beta_1(G(A'))]{730}1a

\putmorphism(-280,500)(0,-1)[\phantom{Y_2}``F\s'F(u)]{450}1l
\putmorphism(450,500)(0,-1)[\phantom{Y_2}``]{450}1r
\putmorphism(300,500)(0,-1)[\phantom{Y_2}``F\s'G(u)]{450}0r
\putmorphism(1180,500)(0,-1)[\phantom{Y_2}``G'G(u)]{450}1r
\put(-160,290){\fbox{$F\s'((\alpha_1)_u)$}}
\put(700,300){\fbox{$(\beta_1)_{G(u)}$}}
\efig
$$
\item for every 1h-cell $f:A\to B$  in $\Aa$ a 2-cell in $\Cc$:
$$\delta_{\beta_1\comp\alpha_1,f}=
\bfig

 \putmorphism(-250,0)(1,0)[F\s'F(A)`F\s'F(B)` F\s'F(f)]{700}1a
 \putmorphism(450,0)(1,0)[\phantom{F\s'F(B)}`F\s'G(B) `F\s'(\alpha_1(B))]{760}1a

 \putmorphism(-250,-450)(1,0)[F\s'F(A)`F\s'G(A)`F\s'(\alpha_1(A))]{700}1a
 \putmorphism(480,-450)(1,0)[\phantom{A\ot B}`F\s'G(B) `F\s'G(f)]{740}1a
 \putmorphism(1200,-450)(1,0)[\phantom{A'\ot B'}` G'G(B) `\beta_1(G(B))]{760}1a

\putmorphism(-280,0)(0,-1)[\phantom{Y_2}``=]{450}1r
\putmorphism(1180,0)(0,-1)[\phantom{Y_2}``=]{450}1r
\put(350,-240){\fbox{$ \delta_{F\s'(\alpha_1),f}  $}}
\put(1000,-700){\fbox{$\delta_{\beta_1,G(f)}$}}

 \putmorphism(450,-900)(1,0)[F\s'G(A)` G'G(A) `\beta_1(G(A))]{760}1a
 \putmorphism(1180,-900)(1,0)[\phantom{A''\ot B'}`G'G(B) ` G'G(f)]{760}1a

\putmorphism(450,-450)(0,-1)[\phantom{Y_2}``=]{450}1l
\putmorphism(1950,-450)(0,-1)[\phantom{Y_2}``=]{450}1r
\efig
$$
where $\delta_{F\s'(\alpha_1),f}$ is from \leref{delta F(alfa)}. 
\end{itemize}
\end{lma}

{\em Vertical pseudonatural transformations} between double pseudo functors $F,G: \Aa\to\Bb$ are defined in an analogous way, 
consisting of a 1v-cell $\alpha(A):F(A)\to G(A)$ in $\Bb$ for every 0-cell $A$ in $\Aa$, 
for every 1h-cell $f:A\to B$ in $\Aa$ a 2-cell on the left hand-side below and for every 1v-cell $u:A\to A'$  in $\Aa$ a 2-cell 
on the right hand-side below, both in $\Bb$:
$$
\bfig
\putmorphism(-150,180)(1,0)[F(A)`F(B)`F(f)]{560}1a
\putmorphism(-150,-190)(1,0)[G(A)`G(B)`G(f)]{600}1a
\putmorphism(-180,180)(0,-1)[\phantom{Y_2}``\alpha(A)]{370}1l
\putmorphism(410,180)(0,-1)[\phantom{Y_2}``\alpha(B)]{370}1r
\put(30,40){\fbox{$\alpha_f$}}
\efig
\qquad\qquad
\bfig
 \putmorphism(-90,500)(1,0)[F(A)`F(A) `=]{540}1a
\putmorphism(-120,500)(0,-1)[\phantom{Y_2}`F(A') `F(u)]{400}1l
\putmorphism(-90,-300)(1,0)[G(A')`G(A') `=]{540}1a
\putmorphism(-120,100)(0,-1)[\phantom{Y_2}``\alpha(A')]{400}1l
\putmorphism(450,100)(0,-1)[\phantom{Y_2}``G(u)]{400}1r
\putmorphism(450,500)(0,-1)[\phantom{Y_2}`G(A) `\alpha(A)]{400}1r
\put(60,50){\fbox{$\delta_{\alpha,u}$}}
\efig
$$
Observe that we use the same notation for the 2-cells $\alpha_{\bullet}$ and $\delta_{\alpha,\bullet}$ both for a horizontal and a vertical 
pseudonatural transformation $\alpha$, the difference is indicated by the notation for the respective 1-cell, recall that horizontal ones are denoted by $f,g..$ and vertical ones by $u,v...$. 

For vertical pseudonatural transformations results analogous to \leref{delta F(alfa)} and \leref{horiz comp hor.ps.tr.} hold,
the analogon of the latter one we state here in order to fix the structures that we use:

\begin{lma} \lelabel{horiz comp vert.ps.tr.}
Horizontal composition of two vertical pseudonatural transformations $\alpha_0: F\Rightarrow G: \Aa\to\Bb$ and 
$\beta_0: F\s'\Rightarrow G':\Bb\to\Cc$, denoted by $\beta_0\comp\alpha_0$, is well-given by: 
\begin{itemize}
\item for every 0-cell $A$ in $\Aa$ a 1v-cell on the left below, and 
for every 1h-cell $f:A\to B$ in $\Aa$ a 2-cell on the right below, both in $\Cc$:
$$(\beta_0\comp\alpha_0)(A)=
\bfig
\putmorphism(-280,500)(0,-1)[F\s'F(A)`F\s'G(A) `F\s'(\alpha_0(A))]{450}1l
 \putmorphism(-280,70)(0,-1)[\phantom{F(A)}`G'G(A) `\beta_0(G(A))]{450}1l
\efig \qquad
(\beta_0\comp\alpha_0)_f= 
\bfig
\putmorphism(-250,500)(1,0)[F\s'F(A)`F\s'F(B)` F\s'F(f)]{700}1a
 \putmorphism(-250,50)(1,0)[F\s'G(A)`F\s'G(B)` F\s'G(f)]{700}1a
 \putmorphism(-250,-400)(1,0)[G'G(A)`G'G(B)` G'G(f)]{700}1a

\putmorphism(-280,500)(0,-1)[\phantom{Y_2}``F\s'(\alpha_0(A))]{450}1l
 \putmorphism(-280,70)(0,-1)[\phantom{F(A)}` `\beta_0(G(A))]{450}1l

\putmorphism(450,500)(0,-1)[\phantom{Y_2}``F\s'(\alpha_0(B))]{450}1r
\putmorphism(450,70)(0,-1)[\phantom{Y_2}``\beta_0(G(B))]{450}1r
\put(-160,290){\fbox{$F\s'((\alpha_0)_f)$}}
\put(-120,-150){\fbox{$(\beta_0)_{G(f)}$}}
\efig
$$
\item for every 1v-cell $u:A\to A'$ in $\Aa$ a 2-cell in $\Cc$: 
$$\delta_{\beta_0\comp\alpha_0,u}=
\bfig
 \putmorphism(-150,500)(1,0)[F\s'F(A)`F\s'F(A) `=]{550}1a
\putmorphism(-180,500)(0,-1)[\phantom{Y_2}`F\s'F(A') `F\s'F(u)]{450}1l
\put(-80,-160){\fbox{$\delta_{F\s'(\alpha_0),u}$}}
\putmorphism(-150,-400)(1,0)[F\s'G(A')`F\s'G(A') `=]{550}1a
\putmorphism(-180,50)(0,-1)[\phantom{Y_2}``F\s'(\alpha_0(A'))]{450}1l
\putmorphism(380,500)(0,-1)[\phantom{Y_2}` `F\s'(\alpha_0(A))]{450}1r
\putmorphism(380,50)(0,-1)[F\s'G(A)` `F\s'G(u)]{450}1r
\putmorphism(520,60)(1,0)[`F\s'G(A)`=]{460}1a
\putmorphism(370,-850)(1,0)[\phantom{G'G(A)}``=]{440}1b
\putmorphism(960,50)(0,-1)[\phantom{(B, \tilde A')}``\beta_0(G(A))]{450}1r
\putmorphism(960,-400)(0,-1)[G'G(A)`G'G(A')` G'G(u)]{450}1r
\putmorphism(400,-400)(0,-1)[\phantom{(B, \tilde A)}`G'G( A') `\beta_0(G(A'))]{450}1l
\put(500,-630){\fbox{$\delta_{\beta_0,G(u)}$}}
\efig
$$
where $\delta_{F\s'(\alpha_0),u}$ is defined analogously as in \leref{delta F(alfa)}. 
\end{itemize}
\end{lma}

Horizontal compositions of horizontal and of vertical pseudonatural transformations are not strictly associative. 

\medskip

We proceed by defining vertical compositions of horizontal and of vertical pseudonatural transformations. From the respective definitions 
it will be clear that these vertical compositions are strictly associative.

\begin{lma} \lelabel{vert comp hor.ps.tr.}
Vertical composition of two horizontal pseudonatural transformations $\alpha_1: F\Rightarrow G: \Aa\to\Bb$ and 
$\beta_1: G\Rightarrow H:\Aa\to\Bb$, denoted by $\frac{\alpha_1}{\beta_1}$, is well-given by: 
\begin{itemize}
\item for every 0-cell $A$ in $\Aa$ a 1h-cell in $\Bb$:
$$(\frac{\alpha_1}{\beta_1})(A)=\big( F(A)\stackrel{\alpha_1(A)}{\longrightarrow}G(A) \stackrel{\beta_1(A)}{\longrightarrow} H(A) \big),$$ 
\item for every 1v-cell $u:A\to A'$ in $\Aa$ a 2-cell in $\Bb$:
$$(\frac{\alpha_1}{\beta_1})(u)=
\bfig
\putmorphism(-150,50)(1,0)[F(A)`G(A)`\alpha_1(A)]{560}1a
 \putmorphism(430,50)(1,0)[\phantom{F(A)}`H(A) `\beta_1(A)]{620}1a

\putmorphism(-150,-320)(1,0)[F(A')`G(A')`\alpha_1(A')]{600}1a
\putmorphism(-180,50)(0,-1)[\phantom{Y_2}``F(u)]{370}1l
\putmorphism(380,50)(0,-1)[\phantom{Y_2}``G(u)]{370}1r
\put(-30,-130){\fbox{$(\alpha_1)_u$}}

 \putmorphism(460,-320)(1,0)[\phantom{F(A)}`H(A') `\beta_1(A')]{630}1a

\putmorphism(1060,50)(0,-1)[\phantom{Y_2}``H(u)]{370}1r
\put(640,-110){\fbox{$(\beta_1)_u$}}
\efig
$$
\item for every 1h-cell $f:A\to B$  in $\Aa$ a 2-cell in $\Bb$:
$$\delta_{\frac{\alpha_1}{\beta_1},f}=
\bfig

 \putmorphism(-200,-50)(1,0)[F(A)`F(B)` F(f)]{650}1a
 \putmorphism(430,-50)(1,0)[\phantom{F(B)}`G(B) `\alpha_1(B)]{700}1a

 \putmorphism(-200,-450)(1,0)[F(A)`G(A)`\alpha_1(A)]{650}1a
 \putmorphism(430,-450)(1,0)[\phantom{A\ot B}`G(B) `G(f)]{700}1a
 \putmorphism(1050,-450)(1,0)[\phantom{A'\ot B'}` H(B) `\beta_1(B)]{700}1a

\putmorphism(-230,-50)(0,-1)[\phantom{Y_2}``=]{400}1r
\putmorphism(1050,-50)(0,-1)[\phantom{Y_2}``=]{400}1r
\put(300,-240){\fbox{$ \delta_{\alpha_1,f}  $}}
\put(1000,-660){\fbox{$\delta_{\beta_1,f}$}}

 \putmorphism(450,-850)(1,0)[G(A)` H(A) `\beta_1(A)]{700}1a
 \putmorphism(1080,-850)(1,0)[\phantom{A''\ot B'}`G(B). ` H(f)]{700}1a

\putmorphism(450,-450)(0,-1)[\phantom{Y_2}``=]{400}1l
\putmorphism(1750,-450)(0,-1)[\phantom{Y_2}``=]{400}1r
\efig
$$
\end{itemize}
\end{lma}

\bigskip

\begin{lma} \lelabel{vert comp vert.ps.tr.}
Vertical composition of two vertical pseudonatural transformations $\alpha_0: F\Rightarrow G: \Aa\to\Bb$ and 
$\beta_0: G\Rightarrow H:\Aa\to\Bb$, denoted by $\frac{\alpha_0}{\beta_0}$, is well-given by: 
\begin{itemize}
\item for every 0-cell $A$ in $\Aa$ a 1v-cell on the left below, and for every 1h-cell $f:A\to B$ in $\Aa$ a 2-cell on the right below, both in $\Bb$:
$$(\frac{\alpha_0}{\beta_0})(A)=
\bfig
\putmorphism(-280,500)(0,-1)[F(A)`G(A) `\alpha_0(A)]{450}1l
 \putmorphism(-280,70)(0,-1)[\phantom{F(A)}`H(A) `\beta_0(A)]{450}1l
\efig \qquad
(\frac{\alpha_0}{\beta_0})(f)= 
\bfig
\putmorphism(-250,500)(1,0)[F(A)`F(B)` F(f)]{550}1a
 \putmorphism(-250,50)(1,0)[G(A)`G(B)` G(f)]{550}1a
 \putmorphism(-250,-400)(1,0)[H(A)`H(B)` H(f)]{550}1a

\putmorphism(-280,500)(0,-1)[\phantom{Y_2}``\alpha_0(A)]{450}1l
 \putmorphism(-280,70)(0,-1)[\phantom{F(A)}` `\beta_0(A)]{450}1l

\putmorphism(300,500)(0,-1)[\phantom{Y_2}``\alpha_0(B)]{450}1r
\putmorphism(300,70)(0,-1)[\phantom{Y_2}``\beta_0(B)]{450}1r
\put(-120,290){\fbox{$(\alpha_0)_f$}}
\put(-120,-150){\fbox{$(\beta_0)_f$}}
\efig
$$ 
\item for every 1v-cell $u:A\to A'$ in $\Aa$ a 2-cell in $\Bb$: 
$$\delta_{\frac{\alpha_0}{\beta_0},u}=
\bfig
 \putmorphism(-150,500)(1,0)[F(A)`F(A) `=]{460}1a
\putmorphism(-130,500)(0,-1)[\phantom{Y_2}`F(A') `F(u)]{400}1l
\put(0,250){\fbox{$\delta_{\alpha_0,u}$}}
\putmorphism(-150,-300)(1,0)[G(A')`G(A') `=]{460}1a
\putmorphism(-130,110)(0,-1)[\phantom{Y_2}``\alpha_0(A')]{400}1l
\putmorphism(380,500)(0,-1)[\phantom{Y_2}` `\alpha_0(A)]{400}1r
\putmorphism(380,100)(0,-1)[G(A)` `G(u)]{400}1l
\putmorphism(480,110)(1,0)[`G(A)`=]{460}1a
\putmorphism(390,-700)(1,0)[\phantom{G(A)}`H(A').`=]{570}1a
\putmorphism(920,100)(0,-1)[\phantom{(B, \tilde A')}``\beta_0(A)]{400}1r
\putmorphism(920,-300)(0,-1)[H(A)`` H(u)]{400}1r
\putmorphism(400,-300)(0,-1)[\phantom{(B, \tilde A)}`H( A') `\beta_0(A')]{400}1l
\put(530,-190){\fbox{$\delta_{\beta_0,u}$}}
\efig
$$
\end{itemize}
\end{lma}

\subsection{2-cells of the tricategory and their compositions} 

Now we may define what will be the 2-cells of our tricategory $\DblPs$ of strict double categories and double pseudofunctors.

\begin{defn} \delabel{double 2-cells}
A {\em double pseudonatural transformation} $\alpha:F\to G$ between double pseudofunctors 
is a quadruple $(\alpha_0, \alpha_1, t^\alpha, r^\alpha)$, where: 

(T1) $\alpha_0:F\Rightarrow G$ is a vertical pseudonatural transformation, and \\ \indent 
 $\alpha_1:F\Rightarrow G$ is a horizontal pseudonatural transformation, 

(T2) the 2-cells $\delta_{\alpha_1,f}$ and $\delta_{\alpha_0,u}$ are invertible when 
$f$ is a 1h-cell component of a horizontal pseudonatural transformation, and 
$u$ is a 1v-cell component of a vertical pseudonatural transformation;

(T3) for every 1h-cell $f:A\to B$ and 1v-cell $u:A\to A'$ in $\Aa$ there are 2-cells in $\Bb$:
$$
\bfig
\putmorphism(-150,500)(1,0)[F(A)`F(B)`F(f)]{600}1a
 \putmorphism(430,500)(1,0)[\phantom{F(A)}`G(B) `\alpha_1(B)]{610}1a
 \putmorphism(-150,100)(1,0)[G(A)`G(B)`G(f)]{1200}1b
\putmorphism(-180,500)(0,-1)[\phantom{Y_2}``\alpha_0(A)]{400}1l
\putmorphism(1050,500)(0,-1)[\phantom{Y_2}``=]{400}1r
\put(350,260){\fbox{$t^\alpha_f$}}
\efig
\qquad\text{and}\qquad
\bfig
 \putmorphism(-150,500)(1,0)[F(A)`G(A)  `\alpha_1(A)]{500}1a
\putmorphism(-180,500)(0,-1)[\phantom{Y_2}`F(A') `F(u)]{400}1l
\putmorphism(-150,-300)(1,0)[G(A')` `=]{400}1a
\putmorphism(-180,100)(0,-1)[\phantom{Y_2}``\alpha_0(A')]{400}1l
\putmorphism(350,500)(0,-1)[\phantom{Y_2}`G(A') `G(u)]{800}1r
\put(50,100){\fbox{$r^\alpha_u$}}
\efig
$$
satisfying: \\
(T3-1)
$$
\bfig
\putmorphism(-150,500)(1,0)[F(A)`F(B)`F(f)]{600}1a
 \putmorphism(450,500)(1,0)[\phantom{F(A)}`G(B) `\alpha_1(B)]{600}1a

 \putmorphism(-150,50)(1,0)[F(A')`F(B')`F(g)]{600}1a
 \putmorphism(450,50)(1,0)[\phantom{F(A)}`G(B') `\alpha_1(B')]{600}1a

\putmorphism(-180,500)(0,-1)[\phantom{Y_2}``F(u)]{450}1l
\putmorphism(450,500)(0,-1)[\phantom{Y_2}``]{450}1r
\putmorphism(300,500)(0,-1)[\phantom{Y_2}``F(v)]{450}0r
\putmorphism(1050,500)(0,-1)[\phantom{Y_2}``G(v)]{450}1r
\put(0,260){\fbox{$F(a)$}}
\put(630,270){\fbox{$(\alpha_1)_v$}}

\putmorphism(-150,-400)(1,0)[G(A')`G(B') `G(g)]{1200}1a

\putmorphism(-180,50)(0,-1)[\phantom{Y_2}``\alpha_0(A')]{450}1l
\putmorphism(1050,50)(0,-1)[\phantom{Y_3}``=]{450}1r
\put(370,-140){\fbox{$t^\alpha_g$}}

\efig=
\bfig
 \putmorphism(-150,500)(1,0)[F(A)`F(A) `=]{500}1a
\putmorphism(-180,500)(0,-1)[\phantom{Y_2}`F(A') `F(u)]{450}1r
\put(0,30){\fbox{$\delta_{\alpha_0,u}$}}
\putmorphism(-150,-400)(1,0)[G(A')`G(A') `=]{520}1a
\putmorphism(-180,50)(0,-1)[\phantom{Y_2}``\alpha_0(A')]{450}1r
\putmorphism(380,500)(0,-1)[\phantom{Y_2}` `\alpha_0(A)]{450}1r
\putmorphism(380,50)(0,-1)[G(A)` `G(u)]{450}1r
\putmorphism(350,500)(1,0)[F(A)`F(B)`F(f)]{600}1a
 \putmorphism(950,500)(1,0)[\phantom{F(A)}`G(B) `\alpha_1(B)]{650}1a
 \putmorphism(470,50)(1,0)[`G(B)`G(f)]{1150}1b
\putmorphism(1570,500)(0,-1)[\phantom{Y_2}``=]{450}1r
\put(880,260){\fbox{$t^\alpha_f$}}
\putmorphism(480,-400)(1,0)[`G(B') `G(g)]{1140}1a
\putmorphism(1570,50)(0,-1)[\phantom{Y_2}``G(v)]{450}1r
\put(850,-200){\fbox{$G(a)$}}
\efig
$$
and 
$$
\bfig
 \putmorphism(-150,500)(1,0)[F(A)`F(B)`F(f)]{600}1a
 \putmorphism(450,500)(1,0)[\phantom{F(A)}`G(B) `\alpha_1(B)]{600}1a
\putmorphism(-180,500)(0,-1)[\phantom{Y_2}`F(A') `F(u)]{450}1l
\put(680,50){\fbox{$r^\alpha_v$}}
\putmorphism(-150,-400)(1,0)[G(A')` `G(g)]{500}1a
\putmorphism(-180,50)(0,-1)[\phantom{Y_2}``\alpha_0(A')]{450}1l
\putmorphism(450,50)(0,-1)[\phantom{Y_2}`G(B')`\alpha_0(B')]{450}1r
\putmorphism(450,500)(0,-1)[\phantom{Y_2}`F(B') `F(v)]{450}1r
\put(0,260){\fbox{$F(a)$}}
\putmorphism(-150,50)(1,0)[\phantom{(B, \tilde A)}``F(g)]{500}1a
\putmorphism(1000,500)(0,-1)[`G(B')`G(v)]{900}1r
\putmorphism(480,-400)(1,0)[\phantom{F(A)}`\phantom{F(A)}`=]{500}1b
\put(0,-170){\fbox{$(\alpha_0)_g$}}
\efig
=
\bfig
\putmorphism(-150,500)(1,0)[F(A)`F(B)`F(f)]{600}1a
 \putmorphism(450,500)(1,0)[\phantom{F(A)}`G(B) `\alpha_1(B)]{680}1a
 \putmorphism(-150,100)(1,0)[F(A)`G(A)`\alpha_1(A)]{600}1a
 \putmorphism(450,100)(1,0)[\phantom{F(A)}`G(B) `G(f)]{680}1a

\putmorphism(-180,500)(0,-1)[\phantom{Y_2}``=]{400}1r
\putmorphism(1100,500)(0,-1)[\phantom{Y_2}``=]{400}1r
\put(350,260){\fbox{$\delta_{\alpha_1,f}$}}
\putmorphism(-180,100)(0,-1)[\phantom{Y_2}`F(A') `F(u)]{400}1l
\putmorphism(-150,-680)(1,0)[G(A')` `=]{480}1a
\putmorphism(-180,-280)(0,-1)[\phantom{Y_2}``\alpha_0(A')]{400}1l
\putmorphism(450,100)(0,-1)[\phantom{Y_2}`G(A') `G(u)]{760}1l
\put(0,-300){\fbox{$r^\alpha_u$}}
\putmorphism(1100,100)(0,-1)[\phantom{Y_2}`G(B') `G(v)]{760}1r
 \putmorphism(450,-660)(1,0)[\phantom{F(A)}` `G(g)]{540}1a
\put(620,-300){\fbox{$G(a)$}} 

\efig
$$
for every 2-cell $a$ in $\Aa$, 

(T3-2) for every composable 1h-cells $f$ and $g$ and every composable 1v-cells $u$ and $v$ it is: 
$$
\bfig

 \putmorphism(440,400)(1,0)[F(B)`F(C)` F(g)]{550}1a
 \putmorphism(1000,400)(1,0)[\phantom{F(B)}`G(C) `\alpha_1(C)]{550}1a

 \putmorphism(-200,0)(1,0)[F(A)`F(B) `F(f)]{600}1a
 \putmorphism(320,0)(1,0)[\phantom{A'\ot B'}` G(B) `\alpha_1(B)]{640}1a
 \putmorphism(900,0)(1,0)[\phantom{A'\ot B'}` G(C) `G(g)]{680}1a

\putmorphism(400,400)(0,-1)[\phantom{Y_2}``=]{400}1l
\putmorphism(1550,400)(0,-1)[\phantom{Y_2}``=]{400}1r
\put(1000,200){\fbox{$ \delta_{\alpha_1,g}$ }}
\put(560,-210){\fbox{$t^\alpha_f$}}

 \putmorphism(-230,-400)(1,0)[G(A)` G(B) ` G(f)]{1210}1a

\putmorphism(-230,0)(0,-1)[\phantom{Y_2}``\alpha_0(A)]{400}1r
\putmorphism(960,0)(0,-1)[\phantom{Y_2}``=]{400}1r
\efig
=
\bfig
 \putmorphism(-100,200)(1,0)[F(A)`F(B)`F(f)]{520}1a
 \putmorphism(420,200)(1,0)[\phantom{F(B)}`F(C) `F(g)]{520}1a
 \putmorphism(930,200)(1,0)[\phantom{F(B)}`G(C) `\alpha_1(C)]{570}1a
\put(20,20){\fbox{$(\alpha_0)_f$}}
\putmorphism(-100,200)(0,-1)[\phantom{Y_2}``\alpha_0(A)]{400}1l
\put(1100,20){\fbox{$t^\alpha_g$}}

\putmorphism(430,200)(0,-1)[\phantom{Y_2}``\alpha_0(B)]{400}1r
\putmorphism(1510,200)(0,-1)[\phantom{Y_2}``=]{400}1r

  \putmorphism(-150,-200)(1,0)[G(A)` G(B) `G(f)]{600}1a
\putmorphism(450,-200)(1,0)[\phantom{F(A)}`G(C) ` G(g)]{1050}1a
\efig
$$ 
and
$$
\bfig
 \putmorphism(920,450)(1,0)[F(A)`F(A) `\alpha_1(A)]{460}1a
\putmorphism(920,450)(0,-1)[\phantom{Y_2}`F(A') `F(u)]{400}1l
\put(1050,250){\fbox{$r^\alpha_u$}}
\putmorphism(920,-400)(1,0)[G(A')`G(A') `=]{460}1a
\putmorphism(1380,450)(0,-1)[\phantom{Y_2}` `G(u)]{850}1r
\putmorphism(440,50)(0,-1)[F(A')`F(A'') `F(v)]{450}1l
\putmorphism(530,60)(1,0)[``=]{300}1a
\putmorphism(560,-850)(1,0)[`G(A'')`=]{430}1a
\putmorphism(920,50)(0,-1)[\phantom{(B, \tilde A')}``\alpha_0(A')]{450}1r
\putmorphism(920,-400)(0,-1)[`` G(v)]{450}1r
\putmorphism(440,-400)(0,-1)[\phantom{(B, \tilde A)}`G(A'') `\alpha_0(A'')]{450}1l
\put(530,-190){\fbox{$\delta_{\alpha_0,v}$}}
\efig
=
\bfig
 \putmorphism(-150,410)(1,0)[F(A)`G(A)  `\alpha_1(A)]{530}1a
\putmorphism(-160,400)(0,-1)[\phantom{Y_2}`F(A') `F(u)]{380}1l
\putmorphism(370,400)(0,-1)[\phantom{Y_2}`G(A') `G(u)]{380}1r
\putmorphism(-160,50)(0,-1)[\phantom{Y_2}`F(A'')`F(v)]{430}1l
\putmorphism(370,50)(0,-1)[\phantom{Y_2}`G(A')`G(v)]{820}1r
\put(-60,240){\fbox{$(\alpha_1)_u$}}
\putmorphism(-70,20)(1,0)[``\alpha_1(A')]{330}1a
\put(20,-210){\fbox{$r^\alpha_v$ }}
\putmorphism(-160,-350)(0,-1)[\phantom{Y_2}``\alpha_0(A'')]{400}1l
\putmorphism(-180,-750)(1,0)[G(A')` `=]{440}1b
\efig
$$

(T3-3) for every composable 1h-cells $f$ and $g$ and every composable 1v-cells $u$ and $v$ it is: 
$$t^\alpha_{gf}=
\bfig
\putmorphism(-150,850)(1,0)[F(A)`F(C)`F(gf)]{1200}1a

 \putmorphism(-180,450)(1,0)[F(A)`F(B)`F(f)]{600}1a
 \putmorphism(420,450)(1,0)[\phantom{F(B)}`F(C) `F(g)]{680}1a
 \putmorphism(1090,450)(1,0)[\phantom{F(B)}`G(C) `\alpha_1(C)]{570}1a
\putmorphism(-180,840)(0,-1)[\phantom{Y_2}``=]{400}1l
\putmorphism(1060,840)(0,-1)[\phantom{Y_2}``=]{400}1r
\put(350,650){\fbox{$F_{gf}$}}
\put(0,250){\fbox{$(\alpha_0)_f$}}
\putmorphism(-180,450)(0,-1)[\phantom{Y_2}``\alpha_0(A)]{400}1l
\put(1000,270){\fbox{$t^\alpha_g$}}

\putmorphism(450,450)(0,-1)[\phantom{Y_2}``\alpha_0(B)]{400}1r
\putmorphism(1660,450)(0,-1)[\phantom{Y_2}``=]{400}1r

  \putmorphism(-150,50)(1,0)[G(A)` G(B) `G(f)]{600}1a
\putmorphism(450,50)(1,0)[\phantom{F(A)}`G(C) ` G(g)]{1200}1a


\putmorphism(-180,50)(0,-1)[\phantom{Y_2}``=]{400}1r
\putmorphism(1640,50)(0,-1)[\phantom{Y_2}``=]{400}1r
\put(650,-180){\fbox{$G^{-1}_{gf}$}}

 \putmorphism(-150,-350)(1,0)[G(A)` G(C) `G(gf)]{1800}1b
\efig
$$ 
and
$$r^\alpha_{vu}=
\bfig
\putmorphism(-600,500)(0,-1)[F(A)`F(A'')`F(vu)]{850}1r
 \putmorphism(-510,500)(1,0)[` F(A)`=]{400}1a
 \putmorphism(-30,500)(1,0)[`G(A)  `\alpha_1(A)]{450}1a
 \putmorphism(400,500)(1,0)[\phantom{(B,A)}` `=]{450}1a
\putmorphism(-180,500)(0,-1)[\phantom{Y_2}`F(A') `F(u)]{450}1r
\put(650,50){\fbox{$G^{vu}$}}
 \putmorphism(-150,-750)(1,0)[G(A'')` `=]{480}1a
\putmorphism(-180,50)(0,-1)[\phantom{Y_2}``F(v)]{400}1r
\putmorphism(-180,-350)(0,-1)[F(A'')``\alpha_0(A'')]{400}1r
\putmorphism(-500,-340)(1,0)[` `=]{200}1a

\putmorphism(450,50)(0,-1)[\phantom{Y_2}`G(A'')`G(v)]{800}1r
\putmorphism(450,500)(0,-1)[\phantom{Y_2}`G(A') `G(u)]{450}1r
\put(80,260){\fbox{$(\alpha_1)_u$}}
\putmorphism(-150,50)(1,0)[\phantom{(B, \tilde A)}``\alpha_1(A')]{500}1a
\putmorphism(960,500)(0,-1)[G(A)`G(A'').`G(vu)]{1250}1r
\putmorphism(470,-750)(1,0)[\phantom{F(A)}`\phantom{F(A)}`=]{450}1a

\put(120,-230){\fbox{$r^\alpha_v$}}
\put(-580,250){\fbox{$(F^{vu})^{-1}$}}
\efig
$$
\end{defn}

By the axiom (v) of a double pseudofunctor in \cite[Definition 6.1]{Shul1} one has:

\begin{lma}
Given three composable 1h-cells $f,g,h$ and three composable 1v-cells $u,v,w$ for a double pseudonatural transformation $\alpha$ 
it is: $t^\alpha_{(hg)f}=t^\alpha_{h(gf)}$ and $r^\alpha_{(wv)u}=r^\alpha_{w(vu)}$. 
\end{lma}

\begin{rem} \rmlabel{axioms used for t} 
The horizontal and vertical compositions of $t$'s and  $r$'s are defined in the next two Propositions below. 
Axiom (T3-2) in the above definition is introduced in order for $t$'s to satisfy the interchange law (up to isomorphism). 
\end{rem}

\medskip

For every 1-cell $F$ of $\DblPs$, the identity 2-cell $\Id_F: F\Rightarrow F$ is given by the 2-cells: 
$((\Id_F)_0)_f=\Id_{F(f)}=t^{\Id_F}_f, ((\Id_F)_1)_u=\Id_{F(u)}=r^{\Id_F}_u, \delta_{(\Id_F)_0,u}=\Id_{F(u)}$ and $\delta_{(\Id_F)_1,f}=\Id_{F(f)}$, 
with $(\Id_F)_0(A)$ and $(\Id_F)_1(A)$ being the identity 1v- and 1h-cells on $F(A)$, respectively, $f$ an arbitrary 1h-cell and $u$ an arbitrary 1v-cell. 

For the horizontal and vertical compositions of double pseudonatural transformations we have: 

\begin{prop} \prlabel{horiz comp 2-cells}
A horizontal composition of two double pseudonatural transformations acting between double pseudo functors  
$(\alpha_0, \alpha_1, t^\alpha, r^\alpha): F\Rightarrow G: \Aa\to\Bb$ and 
$(\beta_0, \beta_1, t^\beta, r^\beta): F\s'\Rightarrow G': \Bb\to\Cc$, denoted by $\beta\comp\alpha$, is well-given by: 
\begin{itemize}
\item the horizontal pseudonatural transformation $\beta_1\comp\alpha_1$ from \leref{horiz comp hor.ps.tr.}, 
\item the vertical pseudonatural transformation $\beta_0\comp\alpha_0$ from \leref{horiz comp vert.ps.tr.}, 
\item for every 1h-cell $f:A\to B$ and 1v-cell $u:A\to A'$ in $\Aa$: 2-cells in $\Bb$:
$$t^{\beta\comp\alpha}_f:= t^\beta_f\comp t^\alpha_f = 
\bfig
\putmorphism(-680,450)(1,0)[F\s'F(A)` `=]{360}1a
\putmorphism(-600,450)(0,-1)[\phantom{Y_2}`F\s'G(A)`F\s'(\alpha_0(A))]{400}1l
\putmorphism(-600,50)(0,-1)[\phantom{Y_2}`F\s'G(A)`=]{400}1l
\putmorphism(-600,-320)(0,-1)[\phantom{Y_2}``\beta_0(G(A))]{400}1l
\putmorphism(-640,-700)(1,0)[G'G(A)``=]{360}1b
\put(-580,-160){\fbox{$\delta_{\beta_0,\alpha_0(A)}$}}
 \putmorphism(-180,450)(1,0)[F\s'F(A)`F\s'F(B)`F\s'F(f)]{600}1a
 \putmorphism(450,450)(1,0)[\phantom{F(B)}`F\s'G(B) `F\s'(\alpha_1(B))]{660}1a
 \putmorphism(1180,450)(1,0)[\phantom{F(B)}`G'G(B) `\beta_1(G(B))]{570}1a
\put(0,250){\fbox{$(\beta_0)_{F(f)}$}}
\putmorphism(-110,450)(0,-1)[\phantom{Y_2}``\beta_0(F(A))]{400}1l
\put(1000,270){\fbox{$t^\beta_{\alpha_1(B)}$}}

\putmorphism(450,450)(0,-1)[\phantom{Y_2}``\beta_0(F(B))]{400}1r
\putmorphism(1660,450)(0,-1)[\phantom{Y_2}``=]{400}1r

  \putmorphism(-150,50)(1,0)[G'F(A)` G'F(B) `G'F(f)]{600}1a
\putmorphism(450,50)(1,0)[\phantom{F(A)}`G'G(B) ` G'(\alpha_1(B))]{1200}1a

\putmorphism(-110,50)(0,-1)[\phantom{Y_2}``=]{400}1r
\putmorphism(1640,50)(0,-1)[\phantom{Y_2}``=]{400}1r
\put(800,-160){\fbox{$G'^{-1}_{\alpha_1(B),F(f) }$}}

 \putmorphism(-150,-350)(1,0)[G'F(A)` G'G(B) `G'(\alpha_1(B)F(f))]{1220}1a
 \putmorphism(1200,-350)(1,0)[` G'G(B) `=]{520}1a
\putmorphism(-130,-700)(1,0)[G'G(A)`G'G(B)`G'G(f)]{1200}1b
\putmorphism(1200,-700)(1,0)[`G'G(B)`=]{520}1b
\putmorphism(-110,-340)(0,-1)[\phantom{Y_2}``G'(\alpha_0(A))]{360}1l
\putmorphism(1630,-320)(0,-1)[\phantom{Y_2}``=]{400}1r
\putmorphism(1030,-320)(0,-1)[\phantom{Y_2}``G'(id)]{400}1l
\put(250,-530){\fbox{$G'(t^\alpha_f)$}}
\put(1090,-560){\fbox{$\big(G'^{G(B)}\big)^{-1}$}}

\efig
$$ 
and \vspace{-0,4cm}
$$r^{\beta\comp\alpha}_u:= r^\beta_u\comp r^\alpha_u = 
\bfig
 \putmorphism(-180,850)(1,0)[F\s'F(A)`F\s'G(A)  `F\s'(\alpha_1(A))]{680}1a
 \putmorphism(530,850)(1,0)[\phantom{(B,A)}`F\s'G(A) `=]{490}1a
 \putmorphism(1160,850)(1,0)[`G'G(A) `\beta_1(G(A))]{600}1a 
\putmorphism(-180,850)(0,-1)[\phantom{Y_2}` `=]{350}1l
\putmorphism(1660,850)(0,-1)[``=]{350}1r
\put(540,670){\fbox{$\delta_{\beta_1,\alpha_1(A)}$}}

 \putmorphism(-180,500)(1,0)[F\s'F(A)`G'F(A)  `\beta_1(F(A))]{660}1a
\putmorphism(-180,500)(0,-1)[\phantom{Y_2}`F\s'F(A') `F\s'F(u)]{450}1l
\put(580,-190){\fbox{$G'^{\bullet\bullet}$}}

 \putmorphism(490,500)(1,0)[\phantom{(B,A)}` `=]{350}1a

\putmorphism(-180,70)(0,-1)[\phantom{Y_2}`F\s'G(A')`F\s'(\alpha_0(A'))]{420}1l
\putmorphism(440,70)(0,-1)[\phantom{Y_2}``]{820}1l  
\putmorphism(440,-60)(0,-1)[\phantom{Y_2}``G'(\alpha_0(A'))]{850}0l  

\putmorphism(-180,-330)(0,-1)[\phantom{Y_2}``\beta_0(G(A'))]{420}1l 
\putmorphism(-180,-760)(1,0)[G'G(A')`G'G(A') `=]{660}1a 
\putmorphism(540,-760)(1,0)[\phantom{F(A)}`G'G(A')`=]{490}1a 
\putmorphism(1200,-760)(1,0)[` `=]{300}1a 

\putmorphism(440,500)(0,-1)[\phantom{Y_2}` `G'F(u)]{450}1r
\put(-20,300){\fbox{$(\beta_1)_{F(u)}$}}
\putmorphism(-180,50)(1,0)[\phantom{F\s'F(A')}`G'F(A')`\beta_1(F(A'))]{720}1a
\putmorphism(980,500)(0,-1)[G'F(A)``G'(\bullet\bullet)]{800}1r   
\putmorphism(980,-280)(0,-1)[``=]{470}1r

\putmorphism(1100,-300)(1,0)[G'G(A')`G'G(A')`G'(id)]{600}1a

\putmorphism(1660,-300)(0,-1)[`G'G(A').`=]{460}1r  

\putmorphism(1660,500)(0,-1)[``G'G(u)]{800}1r
 \putmorphism(1140,500)(1,0)[`G'G(A) `G'(\alpha_1(A))]{600}1a 

\put(0,-190){\fbox{$r^\beta_{\alpha_0(A')}$}}
\put(1180,270){\fbox{$G'(r^\alpha_u)$}}
\put(1190,-530){\fbox{$G'_{G(A')}$}}
\efig
$$
\end{itemize}
\end{prop}

From the axioms of \deref{hor psnat tr} and \leref{delta F(alfa)} identities (9) and (10) in \cite[Section 4.1]{Fem} are deduced, 
of which the vertical version of (10) is used to prove that the horizontal composition of $t$'s satisfies the axiom (T3-1), and the vertical version of (9) 
is used in order to show for this composition to be associative. 
Identities after \cite[Remark 4.11]{Fem} are used in order to prove that the horizontal composition of $t$'s (and $r$'s) satisfies the axiom (T3-2).

\begin{prop} \prlabel{vertic comp 2-cells}
A vertical composition of two double pseudonatural transformations acting between double pseudo functors  
$(\alpha_0, \alpha_1, t^\alpha, r^\alpha): F\Rightarrow G: \Aa\to\Bb$ and 
$(\beta_0, \beta_1, t^\beta, r^\beta): G\Rightarrow H: \Aa\to\Bb$, denoted by $\frac{\alpha}{\beta}$, is well-given by: 
\begin{itemize}
\item the horizontal pseudonatural transformation $\frac{\alpha_1}{\beta_1}$ from \leref{vert comp hor.ps.tr.}, 
\item the vertical pseudonatural transformation $\frac{\alpha_0}{\beta_0}$ from \leref{vert comp vert.ps.tr.}, 
\item for every 1h-cell $f:A\to B$ and 1v-cell $u:A\to A'$ in $\Aa$: 2-cells in $\Bb$:
$$t^{\frac{\alpha}{\beta}}_f:= \frac{t^\alpha_f}{t^\beta_f}= 
\bfig

 \putmorphism(-200,150)(1,0)[F(A)`F(B)` F(f)]{550}1a
 \putmorphism(330,150)(1,0)[\phantom{F(B)}`G(B) `\alpha_1(B)]{550}1a

 \putmorphism(-200,-250)(1,0)[G(A)`G(B) `G(f)]{1100}1a
 \putmorphism(840,-250)(1,0)[\phantom{A'\ot B'}` H(B) `\beta_1(B)]{600}1a

\putmorphism(-230,150)(0,-1)[\phantom{Y_2}``\alpha_0(A)]{400}1r
\putmorphism(850,150)(0,-1)[\phantom{Y_2}``=]{400}1r
\put(500,-40){\fbox{$ t^\alpha_f$} }
\put(1000,-480){\fbox{$t^\beta_f$}}

 \putmorphism(-230,-650)(1,0)[H(A)` H(B) ` H(f)]{1680}1a

\putmorphism(-230,-250)(0,-1)[\phantom{Y_2}``\beta_0(A)]{400}1r
\putmorphism(1420,-250)(0,-1)[\phantom{Y_2}``=]{400}1r
\efig
$$ 
and 
$$r^{\frac{\alpha}{\beta}}_u:= \frac{r^\alpha_u}{r^\beta_u}= 
\bfig
 \putmorphism(-150,500)(1,0)[F(A)`G(A) `\alpha_1(A)]{520}1a
\putmorphism(480,500)(1,0)[`H(A)`\beta_1(A)]{460}1a

\putmorphism(-130,500)(0,-1)[\phantom{Y_2}`F(A') `F(u)]{450}1l
\put(0,250){\fbox{$r^\alpha_u$}}
\putmorphism(-150,-400)(1,0)[G(A')`G(A') `=]{500}1a
\putmorphism(-130,50)(0,-1)[\phantom{Y_2}``\alpha_0(A')]{450}1l
\putmorphism(380,500)(0,-1)[\phantom{Y_2}` `G(u)]{900}1r
\putmorphism(390,-850)(1,0)[\phantom{G(A)}`H(A').`=]{570}1a
\putmorphism(920,500)(0,-1)[`` H(u)]{1350}1r
\putmorphism(400,-400)(0,-1)[\phantom{(B, \tilde A)}`H( A') `\beta_0(A')]{450}1l
\put(530,-190){\fbox{$r^\beta_u$}}
\efig
$$
\end{itemize}
\end{prop}

This composition is clearly strictly associative. The unity constraint 3-cells for the vertical composition of 2-cells will be identities. 
The unity constraints for the horizontal composition we will discuss in \ssref{sum up tricat DblPs}.

\subsection{A subclass of the class of 2-cells} \sslabel{2-cells Theta}

In \cite[Definition 6.3]{Gabi1} {\em double natural transformations} between strict double functors were used, as a 
particular case of {\em generalized natural transformations} from \cite[Definition 3]{BMM}. Adapting the former to the case of 
{\em double pseudo functors} of Shulman, we get the following weakening of \cite[Definition 6.3]{Gabi1}:

\begin{defn} \delabel{Theta}
A {\em $\Theta$-double pseudonatural transformation} between two double pseudofunctors $F\Rightarrow G:\Aa\to\Bb$ is a tripple 
$(\alpha_0, \alpha_1, \Theta^\alpha)$, which we will denote shortly by $\Theta^\alpha$, where: 
\begin{itemize} 
\item $\alpha_0$ is a vertical and $\alpha_1$ a horizontal pseudonatural transformation (from \deref{hor psnat tr} 
and the analogous one), 
\item the axiom (T2) from \deref{double 2-cells} holds, and 
\item for every 0-cell $A$ in $\Aa$ there are 2-cells in $\Bb$:
$$
\bfig
\putmorphism(-150,50)(1,0)[F(A)`G(A)`\alpha_1(A)]{560}1a
\putmorphism(-150,-320)(1,0)[G(A)`G(A)`=]{600}1a
\putmorphism(-180,50)(0,-1)[\phantom{Y_2}``\alpha_0(A)]{370}1l
\putmorphism(410,50)(0,-1)[\phantom{Y_2}``=]{370}1r
\put(30,-120){\fbox{$\Theta^\alpha_A$}}
\efig
$$
so that for every 1h-cell $f:A\to B$ and every 1v-cell $u:A\to A'$ in $\Aa$ the following identities hold: \\
$(\Theta 0)$ 
$$
\bfig
\putmorphism(-150,500)(1,0)[F(A)`F(B)`F(f)]{600}1a
 \putmorphism(450,500)(1,0)[\phantom{F(A)}`G(B) `\alpha_1(B)]{640}1a

 \putmorphism(-150,50)(1,0)[G(A)`G(B)`G(f)]{600}1a
 \putmorphism(450,50)(1,0)[\phantom{F(A)}`G(B) `=]{640}1a

\putmorphism(-180,500)(0,-1)[\phantom{Y_2}``\alpha_0(A)]{450}1l
\putmorphism(450,500)(0,-1)[\phantom{Y_2}``]{450}1r
\putmorphism(300,500)(0,-1)[\phantom{Y_2}``\alpha_0(B)]{450}0r
\putmorphism(1100,500)(0,-1)[\phantom{Y_2}``=]{450}1r
\put(0,260){\fbox{$(\alpha_0)_f$}}
\put(700,270){\fbox{$\Theta^\alpha_B$}}
\efig
\quad=\quad
\bfig
\putmorphism(-150,500)(1,0)[F(A)`F(B)`F(f)]{600}1a
 \putmorphism(450,500)(1,0)[\phantom{F(A)}`G(B) `\alpha_1(B)]{680}1a
 \putmorphism(-150,50)(1,0)[F(A)`G(A)`\alpha_1(A)]{600}1a
 \putmorphism(450,50)(1,0)[\phantom{F(A)}`G(B) `G(f)]{680}1a

\putmorphism(-180,500)(0,-1)[\phantom{Y_2}``=]{450}1r
\putmorphism(1100,500)(0,-1)[\phantom{Y_2}``=]{450}1r
\put(350,260){\fbox{$\delta_{\alpha_1,f}$}}
\putmorphism(-150,-400)(1,0)[F(A')`G(A') `=]{640}1a

\putmorphism(-180,50)(0,-1)[\phantom{Y_2}``\alpha_0(A)]{450}1l
\putmorphism(450,50)(0,-1)[\phantom{Y_2}``]{450}1l
\putmorphism(610,50)(0,-1)[\phantom{Y_2}``=]{450}0l 
\put(20,-180){\fbox{$\Theta^\alpha_A$}} 
\efig
$$
and \\
$(\Theta 1)$ 
$$
\bfig
 \putmorphism(-150,500)(1,0)[F(A)`G(A)  `\alpha_1(A)]{600}1a
\putmorphism(-180,500)(0,-1)[\phantom{Y_2}`F(A') `F(u)]{450}1l
\putmorphism(-150,-400)(1,0)[G(A')` `=]{480}1a
\putmorphism(-180,50)(0,-1)[\phantom{Y_2}``\alpha_0(A')]{450}1l
\putmorphism(450,50)(0,-1)[\phantom{Y_2}`G(A')`=]{450}1r
\putmorphism(450,500)(0,-1)[\phantom{Y_2}`G(A') `G(u)]{450}1r
\put(0,260){\fbox{$(\alpha_1)_u$}}
\putmorphism(-180,50)(1,0)[\phantom{F(A)}``\alpha_1(A')]{500}1a
\put(20,-170){\fbox{$\Theta^\alpha_{A'}$}}
\efig=
\bfig
 \putmorphism(-150,500)(1,0)[F(A)`F(A) `=]{500}1a
 \putmorphism(450,500)(1,0)[` `\alpha_1(A)]{380}1a
\putmorphism(-180,500)(0,-1)[\phantom{Y_2}`F(A') `F(u)]{450}1l
\put(0,-160){\fbox{$\delta_{\alpha_0,u}$}}
\putmorphism(-150,-400)(1,0)[G(A')`G(A'). `=]{500}1a
\putmorphism(-180,50)(0,-1)[\phantom{Y_2}``\alpha_0(A')]{450}1l
\putmorphism(380,500)(0,-1)[\phantom{Y_2}` `\alpha_0(A)]{450}1l
\putmorphism(380,50)(0,-1)[G(A)` `G(u)]{450}1r
\put(550,290){\fbox{$\Theta^\alpha_A$} }
\putmorphism(940,500)(0,-1)[G(A)`G(A)`=]{450}1r
\putmorphism(470,60)(1,0)[``=]{360}1a
\efig
$$
\end{itemize} 
\end{defn}

Let us denote a $\Theta$-double pseudonatural transformation $\Theta^\alpha$ by 
$ \xymatrix{
\Aa  \ar@{:= }[r]|{\Downarrow \Theta^\alpha} \ar@/^{1pc}/[r]^{F} \ar@/_{1pc}/[r]_{G} & \Bb
}$. 

Horizontal composition of $\Theta$-double pseudonatural transformations 
$\xymatrix{
\Aa \ar@{:= }[r]|{\Downarrow \Theta^\alpha} 
\ar@/^{1pc}/[r]^F \ar@/_{1pc}/[r]_G &  
\Bb \ar@{:= }[r]|{\Downarrow \Theta^\beta} \ar@/^{1pc}/[r]^{F'} \ar@/_{1pc}/[r]_{G'} & \Cc
}$ 
is given by
$$\Theta^{\beta\comp\alpha}_A:= \Theta^\beta_A\comp \Theta^\alpha_A = 
\bfig
 \putmorphism(-180,500)(1,0)[F\s'F(A)`F\s'G(A)  ` F\s'(\alpha_1(A))]{670}1a
 \putmorphism(490,500)(1,0)[\phantom{(B,A)}` `=]{350}1a
\putmorphism(-180,500)(0,-1)[\phantom{Y_2}`F\s'G(A) `F\s'(\alpha_0(A))]{450}1l
\putmorphism(440,500)(0,-1)[\phantom{Y_2}` `]{450}1r
\putmorphism(330,410)(0,-1)[\phantom{Y_2}` `F\s'(id)]{450}0r

\put(580,300){\fbox{$ {F\s'^\bullet}^{-1} $}}  
\putmorphism(490,50)(1,0)[\phantom{(B,A)}`F\s'G(A) `=]{550}1a 

\putmorphism(-180,50)(0,-1)[\phantom{Y_2}``=]{450}1l
\putmorphism(440,50)(0,-1)[\phantom{F\s'F(A')}``=]{450}1r 
\put(-40,300){\fbox{$F\s'(\Theta^\alpha_A)$}}

\putmorphism(-200,50)(1,0)[\phantom{F\s'F(A')}`F\s'G(A)`F\s'(id)]{660}1a 
\putmorphism(980,500)(0,-1)[F\s'G(A)``=]{450}1r 
\putmorphism(980,50)(0,-1)[``=]{450}1r

\putmorphism(-200,-400)(1,0)[F\s'G(A)`F\s'G(A)`=]{660}1a 
\putmorphism(500,-400)(1,0)[\phantom{F(A)}`F\s'G(A)`=]{520}1a 

\putmorphism(1660,-380)(0,-1)[``=]{400}1r
 \putmorphism(1170,-400)(1,0)[`G'G(A) `\beta_1(G(A))]{600}1a 
 \putmorphism(1000,-380)(0,-1)[`G'G(A) `\beta_0(G(A))]{400}1l 
 \putmorphism(1160,-800)(1,0)[`G'G(A) `=]{600}1a 

\put(0,-190){\fbox{$F\s'_{G(A)}$}}
\put(650,-190){\fbox{$1$}}
\put(1210,-620){\fbox{$\Theta^\beta_{G(A)}$}}
\efig
$$
and vertical composition of $\Theta$-double pseudonatural transformations 
$\xymatrix{
\Aa \ar@{:=}@/^1pc/[rr] |{\Downarrow \Theta^\alpha} \ar@/^{2pc}/[rr]^{F}  \ar[rr]^G
  \ar@{:=}@/_1pc/[rr] |{\Downarrow \Theta^\beta}  \ar@/_{2pc}/[rr]_{H}  &&
	\Bb }$ 
is given by
$$\frac{\Theta^\alpha_A}{\Theta^\beta_A}=
\bfig

 \putmorphism(-240,150)(1,0)[F(A)`G(A)` \alpha_1(A)]{500}1a

 \putmorphism(-280,-250)(1,0)[G(A)``=]{420}1a  %
 \putmorphism(210,-250)(1,0)[\phantom{A\ot B}`H(A) `\beta_1(A)]{500}1a

\putmorphism(-230,150)(0,-1)[\phantom{Y_2}``\alpha_0(A)]{400}1l
\putmorphism(230,150)(0,-1)[\phantom{Y_2}`G(A)`=]{400}1l

\put(-160,-40){\fbox{$ \Theta^\alpha_A$ }}
\put(350,-440){\fbox{$ \Theta^\beta_A $}}

\putmorphism(250,-250)(0,-1)[\phantom{Y_2}``\beta_0(A)]{400}1l
\putmorphism(690,-250)(0,-1)[``=]{400}1r

 \putmorphism(250,-640)(1,0)[H(A)` H(A). `=]{450}1a
\efig
$$

The following result is directly proved: 

\begin{prop} \prlabel{teta->double}
A $\Theta$-double pseudonatural transformation $\Theta^\alpha$ gives rise to a double pseudonatural transformation $(\alpha_0,\alpha_1, t^\alpha, r^\alpha)$, 
where 
$$t^\alpha_f=
\bfig
\putmorphism(-150,250)(1,0)[F(A)`F(B)`F(f)]{600}1a
 \putmorphism(450,250)(1,0)[\phantom{F(A)}`G(B) `\alpha_1(B)]{600}1a

 \putmorphism(-150,-200)(1,0)[G(A)` G(B)` G(f)]{600}1a
 \putmorphism(450,-200)(1,0)[\phantom{F(A)}`G(B) ` =]{600}1a

\putmorphism(-180,250)(0,-1)[\phantom{Y_2}``\alpha_0(A)]{450}1l
\putmorphism(450,250)(0,-1)[\phantom{Y_2}``]{450}1r
\putmorphism(300,250)(0,-1)[\phantom{Y_2}``\alpha_0(B)]{450}0r
\putmorphism(1050,250)(0,-1)[\phantom{Y_2}``=]{450}1r
\put(0,10){\fbox{$(\alpha_0)_f$}}
\put(650,20){\fbox{$\Theta^\alpha_B$}}
\efig
\quad\text{and}\quad
r^\alpha_u=
\bfig
 \putmorphism(-150,500)(1,0)[F(A)`G(A)  `\alpha_1(A)]{600}1a
\putmorphism(-180,500)(0,-1)[\phantom{Y_2}`F(A') `F(u)]{450}1l
\putmorphism(-150,-400)(1,0)[G(A')` `=]{480}1a
\putmorphism(-180,50)(0,-1)[\phantom{Y_2}``\alpha_0(A')]{450}1l
\putmorphism(450,50)(0,-1)[\phantom{Y_2}`G(A')`=]{450}1r
\putmorphism(450,500)(0,-1)[\phantom{Y_2}`G(A') `G(u)]{450}1r
\put(0,260){\fbox{$(\alpha_1)_u$}}
\putmorphism(-180,50)(1,0)[\phantom{F(A)}``\alpha_1(A')]{500}1a
\put(20,-170){\fbox{$\Theta^\alpha_{A'}$}}
\efig
$$
for every 1h-cell $f:A\to B$ and 1v-cell $u:A\to A'$. Moreover, the class of all $\Theta$-double pseudonatural transformations 
is a subclass of the class of double pseudonatural transformations. 
\end{prop}

\begin{proof}
By the axiom $(\Theta 1)$, axiom 1. for the horizontal pseudonatural transformation $\alpha_1$ implies 
axiom (T3-1) for $t^\alpha_f$. 
\end{proof}

\medskip

Thus, from the point of view of $\Theta$-double pseudonatural transformations, the axioms (T3-2) and (T3-3) of 
double pseudonatural transformations become redundant. 

Observe also that given a double pseudonatural transformation $\alpha:F\Rightarrow G$ acting between {\em strict} 
double functors, the 2-cells $t^\alpha_{id_A}$ obey the conditions $(\Theta 0)$ and $(\Theta 1)$ for every 0-cell $A$.

\bigskip

The other way around, observe that setting 
$$
\bfig
\putmorphism(-150,150)(1,0)[F(A)`G(A)`\alpha_1(A)]{560}1a
\putmorphism(-150,-220)(1,0)[G(A)`G(A).`=]{600}1a
\putmorphism(-180,150)(0,-1)[\phantom{Y_2}``\alpha_0(A)]{370}1l
\putmorphism(410,150)(0,-1)[\phantom{Y_2}``=]{370}1r
\put(0,-40){\fbox{$t\Theta^\alpha_A$}}
\efig
:=
\bfig
 \putmorphism(-150,420)(1,0)[F(A)`F(A)`=]{500}1a
\putmorphism(-180,420)(0,-1)[\phantom{Y_2}``=]{370}1l
\putmorphism(320,420)(0,-1)[\phantom{Y_2}``=]{370}1r
 \putmorphism(-150,50)(1,0)[F(A)`F(A)`F(id_A)]{500}1a
 \put(-80,250){\fbox{$(F_A)^{-1}$}} 
\putmorphism(330,50)(1,0)[\phantom{F(A)}`G(A) `\alpha_1(A)]{560}1a
 \putmorphism(-170,-350)(1,0)[F(A)`G(A)`G(id_A)]{1070}1a

\putmorphism(-180,50)(0,-1)[\phantom{Y_2}``\alpha_0(A)]{400}1r
\putmorphism(910,50)(0,-1)[\phantom{Y_2}``=]{400}1r
\put(540,-170){\fbox{$t^\alpha_{id_A}$}}
\put(240,-550){\fbox{$G_A$}}

\putmorphism(-180,-350)(0,-1)[\phantom{Y_2}``=]{350}1l
\putmorphism(920,-350)(0,-1)[\phantom{Y_3}``=]{350}1r
 \putmorphism(-180,-700)(1,0)[G(A)` G(A), `=]{1070}1a
\efig
$$
by (T3-3) we get 
$$t^\alpha_f=
\bfig
 \putmorphism(-100,200)(1,0)[F(A)`F(B)`F(f)]{520}1a
 \putmorphism(420,200)(1,0)[\phantom{F(B)}`F(B) `F(id_B)]{520}1a
 \putmorphism(930,200)(1,0)[\phantom{F(B)}`G(B) `\alpha_1(B)]{570}1a
\put(20,20){\fbox{$(\alpha_0)_f$}}
\putmorphism(-100,200)(0,-1)[\phantom{Y_2}``\alpha_0(A)]{400}1l
\put(1100,20){\fbox{$t^\alpha_{id_B}$}}

\putmorphism(430,200)(0,-1)[\phantom{Y_2}``\alpha_0(B)]{400}1r
\putmorphism(1510,200)(0,-1)[\phantom{Y_2}``=]{400}1r

  \putmorphism(-150,-200)(1,0)[G(A)` G(B) `G(f)]{600}1a
\putmorphism(450,-200)(1,0)[\phantom{F(A)}`G(B), ` G(id_B)]{1050}1a
\efig
$$
and analogously, setting 
$$
\bfig
\putmorphism(-150,150)(1,0)[F(A)`G(A)`\alpha_1(A)]{560}1a
\putmorphism(-150,-220)(1,0)[G(A)`G(A).`=]{600}1a
\putmorphism(-180,150)(0,-1)[\phantom{Y_2}``\alpha_0(A)]{370}1l
\putmorphism(410,150)(0,-1)[\phantom{Y_2}``=]{370}1r
\put(0,-40){\fbox{$r\Theta^\alpha_A$}}
\efig
:=
\bfig
\putmorphism(-610,500)(1,0)[F(A)` `=]{370}1a
\putmorphism(-600,500)(0,-1)[` `=]{450}1l
\putmorphism(-610,50)(1,0)[F(A)`F(A) `=]{450}1a
\put(-460,260){\fbox{$F^A$} }

\putmorphism(-150,500)(1,0)[F(A)`G(A) `\alpha_1(A)]{500}1a
 \putmorphism(450,500)(1,0)[` `=]{330}1a
\putmorphism(-180,500)(0,-1)[\phantom{Y_2}` `F(id_A)]{450}1r
\put(0,-160){\fbox{$r^\alpha_{id_A}$}}
\putmorphism(-150,-400)(1,0)[G(A)`G(A) `=]{500}1a
\putmorphism(-180,50)(0,-1)[\phantom{Y_2}``\alpha_0(A)]{450}1l
\putmorphism(380,500)(0,-1)[\phantom{Y_2}` `G(id_A)]{900}1l
\put(460,260){\fbox{$(G^A)^{-1}$} }
  \putmorphism(870,500)(0,-1)[G(A)`G(A),`=]{900}1r
\putmorphism(440,-400)(1,0)[``=]{340}1a
\efig
$$
one gets 
$$r^\alpha_u=
\bfig
 \putmorphism(-150,410)(1,0)[F(A)`G(A)  `\alpha_1(A)]{530}1a
\putmorphism(-160,400)(0,-1)[\phantom{Y_2}`F(A') `F(u)]{380}1l
\putmorphism(370,400)(0,-1)[\phantom{Y_2}`G(A') `G(u)]{380}1r
\putmorphism(-160,50)(0,-1)[\phantom{Y_2}`F(A')`F(id_{A'})]{430}1l
\putmorphism(370,50)(0,-1)[\phantom{Y_2}`G(A').`G(id_{A'})]{820}1r
\put(-60,240){\fbox{$(\alpha_1)_u$}}
\putmorphism(-70,20)(1,0)[``\alpha_1(A')]{330}1a
\put(20,-210){\fbox{$r^\alpha_{id_{A'}} $}}
\putmorphism(-160,-350)(0,-1)[\phantom{Y_2}``\alpha_0(A')]{400}1l
\putmorphism(-180,-750)(1,0)[G(A')` `=]{440}1b
\efig
$$
By successive applications of (T3-2) and axiom 2. for $\alpha_0$ one gets that $t\Theta^\alpha_{\bullet}$ satisfies the axiom $(\Theta 0)$, 
and similarly $r\Theta^\alpha_{\bullet}$ satisfies the axiom $(\Theta 1)$ of \deref{Theta}. Since the 2-cells $t^\alpha_f$ and $r^\alpha_f$ are 
not related, we can not claim that all double pseudonatural transformations are $\Theta$-double pseudotransformations.




\subsection{3-cells of the tricategory } 

We first define modifications for horizontal and vertical pseudonatural transformations. Since we will then define modifications for 
double pseudonatural transformations, for mnemonic reasons we will denote vertical pseudonatural transformations with index 0 and horizontal 
ones with index 1.

\begin{defn} \delabel{modif vert/horiz psnat tr}
A modification between two vertical pseudonatural transformations $\alpha_0$ and $\beta_0$ which act between double psuedofunctors 
$F\Rightarrow G$ is an application $a: \alpha_0\Rrightarrow\beta_0$ such that for each 0-cell $A$ in $\Aa$ there is a horizontally globular 
2-cell $a_0(A):\alpha_0(A)\Rightarrow\beta_0(A)$ which for each 1h-cell $f:A\to B$ satisfies: 
$$
\bfig
\putmorphism(-100,250)(1,0)[F(A)`F(A)`=]{550}1a
 \putmorphism(430,250)(1,0)[\phantom{F(A)}`F(B) `F(f)]{580}1a

 \putmorphism(-100,-200)(1,0)[G(A)`G(A)`=]{550}1a
 \putmorphism(450,-200)(1,0)[\phantom{F(A)}`G(B) `G(f)]{580}1a

\putmorphism(-100,250)(0,-1)[\phantom{Y_2}``\alpha_0(A)]{450}1l
\putmorphism(450,250)(0,-1)[\phantom{Y_2}``]{450}1r
\putmorphism(300,250)(0,-1)[\phantom{Y_2}``\beta_0(A)]{450}0r
\putmorphism(1020,250)(0,-1)[\phantom{Y_2}``\beta_0(B)]{450}1r
\put(0,10){\fbox{$a_0(A)$}}
\put(620,20){\fbox{$(\beta_0)_f$}}
\efig\quad
=\quad
\bfig
\putmorphism(-100,250)(1,0)[F(A)`F(B)`F(f)]{550}1a
 \putmorphism(430,250)(1,0)[\phantom{F(A)}`F(B) `=]{550}1a

 \putmorphism(-100,-200)(1,0)[G(A)`G(B)`G(f)]{550}1a
 \putmorphism(450,-200)(1,0)[\phantom{F(A)}`G(B) `=]{550}1a

\putmorphism(-100,250)(0,-1)[\phantom{Y_2}``\alpha_0(A)]{450}1l
\putmorphism(450,250)(0,-1)[\phantom{Y_2}``]{450}1r
\putmorphism(300,250)(0,-1)[\phantom{Y_2}``\alpha_0(B)]{450}0r
\putmorphism(990,250)(0,-1)[\phantom{Y_2}``\beta_0(B)]{450}1r
\put(0,40){\fbox{$ (\alpha_0)_f$}}
\put(590,30){\fbox{$a_0(B)$}}
\efig
$$
and 

$$
\bfig
 \putmorphism(-400,-400)(1,0)[` `=]{380}1a
 \putmorphism(80,500)(1,0)[F(A)`F(A) `=]{500}1a
\putmorphism(-480,50)(0,-1)[F(A')`G(A')`\alpha_0(A')]{450}1l
\putmorphism(-380,60)(1,0)[``=]{360}1a
\putmorphism(70,500)(0,-1)[\phantom{Y_2}`F(A') `F(u)]{450}1l
\put(200,270){\fbox{$\delta_{\beta_0,u}$}}
\putmorphism(80,-400)(1,0)[G(A')`G(A') `=]{500}1a
\putmorphism(70,50)(0,-1)[\phantom{Y_2}``\beta_0(A')]{450}1r
\putmorphism(580,500)(0,-1)[\phantom{Y_2}` `\beta_0(A)]{450}1r
\putmorphism(580,50)(0,-1)[G(A)` `G(u)]{450}1r
\put(-370,-180){\fbox{$a_0(A')$} }
\efig
=
\bfig
 \putmorphism(-150,500)(1,0)[F(A)`F(A) `=]{500}1a
 \putmorphism(450,500)(1,0)[` `=]{380}1a
\putmorphism(-180,500)(0,-1)[\phantom{Y_2}`F(A') `F(u)]{450}1l
\put(0,-160){\fbox{$\delta_{\alpha_0,u}$}}
\putmorphism(-150,-400)(1,0)[G(A')`G(A'). `=]{500}1a
\putmorphism(-180,50)(0,-1)[\phantom{Y_2}``\alpha_0(A')]{450}1l
\putmorphism(380,500)(0,-1)[\phantom{Y_2}` `\alpha_0(A)]{450}1l
\putmorphism(380,50)(0,-1)[G(A)` `G(u)]{450}1r
\put(520,290){\fbox{$a_0(A)$} }
\putmorphism(940,500)(0,-1)[F(A)`G(A)`\beta_0(A)]{450}1r
\putmorphism(470,60)(1,0)[``=]{360}1a
\efig
$$
\end{defn}

A modification between two horizontal pseudonatural transformations $\alpha_1$ and $\beta_1$ which act between double psuedofunctors 
$F\Rightarrow G$ is an application $a: \alpha_1\Rrightarrow\beta_1$ such that for each 0-cell $A$ in $\Aa$ there is a vertically globular 2-cell 
$a_1(A):\alpha_1(A)\Rightarrow\beta_1(A)$ which for each 1v-cell $u:A\to A'$ satisfies two conditions analogous to those of the above definition.

Now, 3-cells for our tricategory $\DblPs$ will be modifications which we define here:

\begin{defn} \delabel{modif 3}
A modification between two double pseudonatural transformations $\alpha=(\alpha_0,\alpha_1,t^\alpha, r^\alpha)$ and 
$\beta=(\beta_0,\beta_1, t^\beta, r^\beta)$ which act between double pseudofunctors $F\Rightarrow G$ is an application 
$a: \alpha\Rrightarrow\beta$ consisting of a modification $a_0$ for vertical pseudonatural transformations and a 
modification $a_1$ for horizontal pseudonatural transformations, such that for each 0-cell $A$ in $\Aa$ it holds: 
\begin{equation} \eqlabel{modif t}
\bfig
\putmorphism(900,650)(1,0)[F(B)`G(B)  `\alpha_1(B)]{520}1a
\putmorphism(890,650)(0,-1)[\phantom{Y_2}` `=]{380}1l
\putmorphism(1420,650)(0,-1)[\phantom{Y_2}` `=]{380}1r
\put(1020,470){\fbox{$a_1(B)$}}

\putmorphism(-100,250)(1,0)[F(A)`F(A)`=]{550}1a
 \putmorphism(430,250)(1,0)[\phantom{F(A)}`F(B) `F(f)]{480}1a
 \putmorphism(900,250)(1,0)[\phantom{F(A)}`G(B) `\beta_1(B)]{520}1a
\putmorphism(1420,250)(0,-1)[\phantom{Y_2}``=]{450}1r

 \putmorphism(-100,-200)(1,0)[G(A)`G(A)`=]{550}1a
 \putmorphism(450,-200)(1,0)[\phantom{F(A)}`G(B) `G(f)]{980}1a

\putmorphism(-100,250)(0,-1)[\phantom{Y_2}``\alpha_0(A)]{450}1l
\putmorphism(450,250)(0,-1)[\phantom{Y_2}``]{450}1r
\putmorphism(300,250)(0,-1)[\phantom{Y_2}``\beta_0(A)]{450}0r

\put(0,10){\fbox{$a_0(A)$}}
\put(1070,20){\fbox{$t^\beta_f$}}
\efig\quad
=\quad
\bfig
\putmorphism(-80,250)(1,0)[F(A)`F(B)`F(f)]{530}1a
 \putmorphism(450,250)(1,0)[\phantom{F(A)}`G(B) `\alpha_1(B)]{580}1a
 \putmorphism(-80,-150)(1,0)[F(A)`G(A)`G(f)]{1120}1b
\putmorphism(-100,250)(0,-1)[\phantom{Y_2}``\alpha_0(A)]{400}1r
\putmorphism(1020,250)(0,-1)[\phantom{Y_2}``=]{400}1r
\put(350,30){\fbox{$t^\alpha_f$}}
\efig
\end{equation}
and
$$
\bfig
 \putmorphism(-150,410)(1,0)[F(A)`G(A)  `\alpha_1(A)]{490}1a
\putmorphism(-160,430)(0,-1)[\phantom{Y_2}`F(A) `=]{410}1l
\putmorphism(370,430)(0,-1)[\phantom{Y_2}`G(A) `=]{410}1r
\putmorphism(-160,50)(0,-1)[\phantom{Y_2}``F(u)]{420}1l
\putmorphism(370,50)(0,-1)[\phantom{Y_2}`G(A')`G(u)]{820}1r
\put(-60,240){\fbox{$a_1(A)$}}

\putmorphism(-180,20)(1,0)[\phantom{F(A)}``\beta_1(A)]{460}1a
\put(20,-180){\fbox{$r^\beta_u $}}

\putmorphism(-160,-350)(0,-1)[\phantom{Y_2}``\beta_0(A')]{400}1r
\putmorphism(-620,-350)(0,-1)[\phantom{Y_2}``\alpha_0(A')]{400}1l
\putmorphism(-620,-350)(1,0)[F(A')`F(A') `=]{440}1b
\putmorphism(-630,-750)(1,0)[G(A')` `=]{340}1b
\putmorphism(-180,-750)(1,0)[G(A')` `=]{440}1b
\put(-540,-560){\fbox{$a_0(A')$}}

\efig
=
\bfig
 \putmorphism(-150,500)(1,0)[F(A)`G(A)  `\alpha_1(A)]{500}1a
\putmorphism(-180,500)(0,-1)[\phantom{Y_2}`F(A') `F(u)]{400}1l
\putmorphism(-150,-300)(1,0)[G(A')` `=]{400}1a
\putmorphism(-180,100)(0,-1)[\phantom{Y_2}``\alpha_0(A')]{400}1l
\putmorphism(350,500)(0,-1)[\phantom{Y_2}`G(A'). `G(u)]{800}1r
\put(50,100){\fbox{$r^\alpha_u$}}
\efig
$$
\end{defn}

{\em Horizontal composition} of the modifications $a: \alpha\Rrightarrow\beta: F\Rightarrow G$ and $b: \alpha'\Rrightarrow\beta': 
F\s'\Rightarrow G'$, acting between horizontally composable double pseudonatural transformations 
$\alpha'\comp\alpha\Rrightarrow\beta'\comp\beta : F\s'\comp F\Rightarrow G'\comp G$ is given for every 0-cell 
$A$ in $\Aa$ by pairs consisting of
$$(b\comp a)_0(A)=
\bfig
 \putmorphism(-170,470)(1,0)[F\s'F(A)`F\s'F(A)  `=]{570}1a
\putmorphism(-180,470)(0,-1)[\phantom{Y_2}` `=]{400}1l
\putmorphism(-180,80)(0,-1)[\phantom{Y_2}``F\s'(\alpha_0(A))]{400}1l
\putmorphism(370,80)(0,-1)[\phantom{Y_2}``F\s'(\beta_0(A))]{400}1r
\putmorphism(370,470)(0,-1)[\phantom{Y_2}``=]{400}1r  
\put(-40,280){\fbox{$(F\s'_\bullet)^{-1}$}}

\putmorphism(-170,70)(1,0)[F\s'F(A)`F\s'F(A) `F\s'(id)]{570}1a
\putmorphism(-170,-320)(1,0)[F\s'G(A)`F\s'G(A) `F\s'(id)]{570}1b
\put(-140,-130){\fbox{$F\s'(a_0(A))$}}

\putmorphism(-180,-310)(0,-1)[\phantom{Y_2}``=]{400}1l
\putmorphism(370,-310)(0,-1)[\phantom{Y_2}``=]{400}1r
\putmorphism(-170,-710)(1,0)[F\s'G(A)`F\s'G(A) `=]{570}1b
\put(-20,-560){\fbox{$F\s'_\bullet$}}

\putmorphism(-180,-700)(0,-1)[\phantom{Y_2}``\alpha'_0(G(A))]{400}1l
\putmorphism(370,-700)(0,-1)[\phantom{Y_2}``\beta'_0(G(A))]{400}1r
\putmorphism(-170,-1100)(1,0)[G'G(A)`G'G(A) `=]{570}1b
\put(-120,-930){\fbox{$b_0(G(A))$}}

\efig
$$
and 
$$(b\comp a)_1(A)=
\bfig
\putmorphism(-580,250)(0,-1)[\phantom{Y_2}``=]{450}1l
\putmorphism(1520,250)(0,-1)[\phantom{Y_2}``=]{450}1r
\putmorphism(-570,250)(1,0)[F\s'F(A)`\phantom{HF(A)}`=]{550}1a
\putmorphism(-570,-200)(1,0)[F\s'F(A)`\phantom{HF(A)}`=]{550}1a
\put(-460,30){\fbox{$F\s'^\bullet$}}
\put(1040,30){\fbox{$(F\s'^\bullet)^{-1}$}}

 \putmorphism(880,250)(1,0)[\phantom{HF(A)}`F\s'G(A) `=]{550}1a
 \putmorphism(1430,250)(1,0)[\phantom{HF(A)}`F\s'G(A) `\alpha_1'(G(A))]{900}1a
 \putmorphism(880,-200)(1,0)[\phantom{HF(A)}`F\s'G(A) `=]{550}1a
 \putmorphism(1450,-200)(1,0)[\phantom{HF(A)}`F\s'G(A). `\beta_1'(G(A))]{900}1a
\putmorphism(2350,250)(0,-1)[\phantom{Y_2}``=]{450}1l
\put(1680,30){\fbox{$b_1(G(A))$}}

\putmorphism(-20,250)(1,0)[F\s'F(A)`F\s'G(A)` F\s'(\alpha_1(A))]{900}1a
 \putmorphism(-20,-200)(1,0)[F\s'F(A)`F\s'G(A)`F\s'(\beta_1(A))]{900}1a

\putmorphism(-120,250)(0,-1)[\phantom{Y_2}``]{450}1r
\putmorphism(-140,250)(0,-1)[\phantom{Y_2}``F\s'(id)]{450}0r

\putmorphism(940,250)(0,-1)[\phantom{Y_2}``]{450}1l
\putmorphism(960,250)(0,-1)[\phantom{Y_2}``F\s'(id)]{450}0l

\put(180,30){\fbox{$F\s'(a_1(A))$}}

\efig
$$








{\em Vertical composition} of the modifications $a: \alpha\Rrightarrow\beta: F\Rightarrow G$ and $b: \alpha'\Rrightarrow\beta': 
G\Rightarrow H$, acting between vertically composable double pseudonatural transformations $F\stackrel{\alpha}{\Rightarrow} G 
\stackrel{\alpha'}{\Rightarrow}H$ and $F\stackrel{\beta}{\Rightarrow} G \stackrel{\beta'}{\Rightarrow}H$, is given for every 0-cell 
$A$ in $\Aa$ by pairs consisting of
$$(\frac{a}{b})_0(A)=
\bfig
 \putmorphism(-150,470)(1,0)[F(A)`F(A)  `=]{440}1a
\putmorphism(-160,470)(0,-1)[\phantom{Y_2}`G(A) `\alpha_0(A)]{400}1l
\putmorphism(-160,80)(0,-1)[\phantom{Y_2}``\alpha'_0(A)]{400}1l
\putmorphism(300,80)(0,-1)[\phantom{Y_2}`H(A)`\beta'_0(A)]{400}1r
\putmorphism(300,470)(0,-1)[\phantom{Y_2}`G(A) `\beta_0(A)]{400}1r
\put(-80,260){\fbox{$a_0(A)$}}

\putmorphism(-180,70)(1,0)[\phantom{F(A)}``=]{380}1a
\putmorphism(-220,-300)(1,0)[H(A)` `\beta_1(A')]{440}1b
\put(-80,-130){\fbox{$a'_0(A)$}}
\efig
\text{and}\quad
(\frac{a}{b})_1(A)=
\bfig
\putmorphism(-100,250)(1,0)[F(A)`G(A)`\alpha_1(A)]{550}1a
 \putmorphism(430,250)(1,0)[\phantom{F(A)}`H(A) `\alpha'_1(A)]{580}1a

 \putmorphism(-100,-200)(1,0)[F(A)`G(A)`\beta_1(A)]{550}1a
 \putmorphism(450,-200)(1,0)[\phantom{F(A)}`H(A). `\beta'_1(A)]{580}1a

\putmorphism(-100,250)(0,-1)[\phantom{Y_2}``=]{450}1r
\putmorphism(420,250)(0,-1)[\phantom{Y_2}``]{450}1r
 \putmorphism(400,250)(0,-1)[\phantom{Y_2}``=]{450}0r
\putmorphism(1020,250)(0,-1)[\phantom{Y_2}``=]{450}1r
\put(40,10){\fbox{$a_1(A)$}}
\put(560,20){\fbox{$a'_1(A)$}}
\efig
$$

{\em Transversal composition} of the modifications 
$\alpha\stackrel{a}{\Rrightarrow}\beta \stackrel{b}{\Rrightarrow}\gamma : F\Rightarrow G$ is given for every 0-cell 
$A$ in $\Aa$ by pairs consisting of
$$(b\cdot a)_0(A)=
\bfig
\putmorphism(-100,250)(1,0)[F(A)`G(A)`=]{550}1a
 \putmorphism(430,250)(1,0)[\phantom{F(A)}`H(A) `=]{580}1a

 \putmorphism(-100,-200)(1,0)[F(A)`G(A)`=]{550}1a
 \putmorphism(450,-200)(1,0)[\phantom{F(A)}`H(A) `=]{580}1a

\putmorphism(-100,250)(0,-1)[\phantom{Y_2}``]{450}1r
\putmorphism(-120,200)(0,-1)[\phantom{Y_2}``\alpha_0(A)]{450}0r

\putmorphism(420,250)(0,-1)[\phantom{Y_2}``]{450}1r
 \putmorphism(400,200)(0,-1)[\phantom{Y_2}``\beta_0(A)]{450}0r

\putmorphism(1020,250)(0,-1)[\phantom{Y_2}``]{450}1r
\putmorphism(1000,200)(0,-1)[\phantom{Y_2}``\gamma_0(A)]{450}0r

\put(100,100){\fbox{$a_0(A)$}}
\put(660,100){\fbox{$b_0(A)$}}
\efig\quad
\text{and}\quad
(b\cdot a)_1(A)=
\bfig
 \putmorphism(-150,470)(1,0)[F(A)`G(A)  `\alpha_1(A)]{440}1a
\putmorphism(-160,470)(0,-1)[\phantom{Y_2}`F(A) `=]{400}1l
\putmorphism(-160,80)(0,-1)[\phantom{Y_2}`F(A)`=]{400}1l
\putmorphism(300,80)(0,-1)[\phantom{Y_2}`G(A).`=]{400}1r
\putmorphism(300,470)(0,-1)[\phantom{Y_2}`G(A) `=]{400}1r
\put(-80,260){\fbox{$a_1(A)$}}

\putmorphism(-160,70)(1,0)[\phantom{F(A)}``\beta_1(A)]{380}1a
\putmorphism(-100,-300)(1,0)[\phantom{Y_2}` `\gamma_1(A)]{300}1b
\put(-80,-130){\fbox{$b_1(A)$}}
\efig
$$

From the definitions it is clear that vertical and transversal composition of the 3-cells is strictly associative. That the associativity in 
the horizontal direction is also strict we proved in \cite[Section 4.7]{Fem}.

\subsection{A subclass of the 3-cells} 

For $\Theta$-double pseudonatural transformations we define modifications as follows:

\begin{defn}
A modification between two $\Theta$-double pseudonatural transformations $\Theta^\alpha\equiv(\alpha_0,\alpha_1,\Theta^\alpha)$ and 
$\Theta^\beta\equiv(\beta_0,\beta_1, \Theta^\beta)$ which act between double psuedofunctors $F\Rightarrow G$ is an application 
$a: \alpha\Rrightarrow\beta$ consisting of a modification $a_0$ for vertical pseudonatural transformations and a 
modification $a_1$ for horizontal pseudonatural transformations, such that for each 0-cell $A$ in $\Aa$ it holds:
$$
\bfig
\putmorphism(450,450)(1,0)[`G(A) `\alpha_1(A)]{460}1a
\putmorphism(380,470)(0,-1)[F(A)` `=]{420}1l
\putmorphism(920,470)(0,-1)[\phantom{Y_2}``=]{420}1r
\put(500,250){\fbox{$a_1(A)$} }

 \putmorphism(-150,50)(1,0)[F(A)`F(A) `=]{500}1a
 \putmorphism(450,50)(1,0)[` `\beta_1(A)]{380}1a
\putmorphism(-180,50)(0,-1)[\phantom{Y_2}` `\alpha_0(A)]{450}1l
\put(-50,-190){\fbox{$a_0(A)$}}
\putmorphism(-150,-400)(1,0)[G(A)`G(A) `=]{500}1a
\putmorphism(360,50)(0,-1)[\phantom{Y_2}` `\beta_0(A)]{450}1r
\put(660,-180){\fbox{$\Theta^\beta_A$} }
\putmorphism(940,50)(0,-1)[G(A)`G(A)`=]{450}1r
\putmorphism(470,-400)(1,0)[``=]{360}1a
\efig\quad
=
\bfig
\putmorphism(-150,50)(1,0)[F(A)`G(A)`\alpha_1(A)]{560}1a
\putmorphism(-150,-320)(1,0)[G(A)`G(A).`=]{600}1a
\putmorphism(-180,50)(0,-1)[\phantom{Y_2}``\alpha_0(A)]{370}1l
\putmorphism(410,50)(0,-1)[\phantom{Y_2}``=]{370}1r
\put(30,-140){\fbox{$\Theta^\alpha_A$}}
\efig
$$
\end{defn}

It is directly proved that modifications between $\Theta$-double pseudonatural transformations are particular cases of modifications 
between double pseudonatural transformations. This gives a sub-tricategory $\DblPs^\Theta$ of the tricategory $\DblPs$. 

\subsection{The obtained tricategory}  \sslabel{sum up tricat DblPs}

In Sections 4.6 and 4.9 of \cite{Fem} we proved that horizontal associativity of the 2-cells of $\DblPs$ and the interchange law for 2-cells, 
respectively, hold up to isomorphisms, which we gave explicitly. In Section 4.7 of {\em loc.cit.} we proved the strict associativity of the 3-cells, 
and in Section 4.8 we showed that left unity constraints on 2-cells is identity, but for the right one we gave an isomorphism. 
In Section 4.10 we showed that the distinguished modifications from the axiom (TD8) of \cite{GPS} fulfill the required identities, which 
concludes the construction of the tricategory $\DblPs$.

\section{The 2-category $PsDbl$ embeds into our tricategory $\DblPs$} \selabel{Ps embeds}

As we want to propose an alternative notion to intercategories as {\em categories internal to the tricategory $\DblPs$}, which we do in the next 
two sections, so that monoids in the Cartesian monoidal category $(Dbl, \oast)$ fit in it, 
in this section we compare the 2-category $LxDbl$ and our tricategory $\DblPs$. 
As we explained, we can not embed $LxDbl$ (whose 1-cells are lax double functors) into $\DblPs$, instead we will embed the 2-category $PsDbl$ 
of pseudo double categories, pseudo double functors and vertical transformations, used in \cite{Shul}. Apart from 1-cells, 
it differs from $LxDbl$ in that the 
horizontal direction is weak and the vertical one is strict, while in the approach 
of Grandis and Par\'e and in $LxDbl$ it is the other way around. Moreover, 2-cells in $PsDbl$ are vertical rather than horizontal transformations, 
as in $LxDbl$. Thus the 2-category $PsDbl$ is the closest one to $LxDbl$ in the presented context which we could embed into our tricategory $\DblPs$.

The 0-cells of $PsDbl$ are pseudo double categories and not {\em strict} double categories as in $\DblPs$. Though, by Strictification 
Theorem of \cite[Section 7.5]{GP:Limits} every pseudo double category is equivalent by a pseudodouble functor 
to a strict double category. Let $PsDbl^*_3$ be the 3-category defined by adding only the identity 3-cells to the 2-category equivalent 
to $PsDbl$ having strict double categories for 0-cells. Thus $PsDbl^*_3$ consists of strict double categories, pseudo double functors, 
vertical transformations and identity modifications among the latter. Pseudo double functors are in particular double pseudo functors, 
so the only thing it remains to check is how to make a vertical transformation a double pseudonatural transformation, that is, embed 2-cells 
of $PsDbl$ into those of $\DblPs$. 

Before doing this, we prove some more general results. 



\subsection{Bijectivity between strong vertical and strong horizontal transformations} \sslabel{bijective corr}

Recall that a {\em companion} for a 1v-cell $u:A\to A'$ is a 1h-cell  $u_*:A\to A'$ together with certain 2-cells $\Epsilon$ and $\eta$ satisfying 
$[\eta\vert\Epsilon]=\Id_{u_*}$ and $\frac{\eta}{\Epsilon}=\Id_{u}$, 
\cite[Section 1.2]{GP:Adj}, \cite[Section 3]{Shul}. (Here $[\eta\vert\Epsilon]$ denotes the horizontal composition of 2-cells, 
where $\eta$ acts first, and the fraction denotes their vertical composition.) We will say that $u_*$ is a {\em 1h-companion} of $u$. 
Companions are unique up to a unique globular isomorphism \cite[Lemma 3.8]{Shul} and a {\em connection} on a double category is 
a functorial choice of a companion for each 1v-cell, \cite{BS}. We will need a functorial choice of companions only for 1v-cell 
components of vertical pseudonatural transformations, accordingly we will speak about a {\em connection on those 1v-cells}. 

\begin{prop} \prlabel{vertic->horiz}
Let $\alpha_0:F\Rightarrow G$ be a strong vertical transformation between pseudo double functors acting between 
strict double categories $\Aa\to\Bb$ (\cite[Section 7.4]{GP:Limits}). 
The following data define a horizontal pseudonatural transformation $\alpha_1:F\Rightarrow G$:
\begin{itemize}
\item a fixed choice of a 1h-companion of $\alpha_0(A)$, for every 0-cell $A$ of $\Aa$ (with corresponding 2-cells 
$\Epsilon^\alpha_A$ and $\eta^\alpha_A$), 
we denote it by $\alpha_1(A)$; 
\item the 2-cell 
$$(\alpha_1)_u=\quad
\bfig
 \putmorphism(-40,50)(1,0)[``=]{320}1b

\put(550,30){\fbox{$\delta_{\alpha_0,u}$}}
\putmorphism(-150,-400)(1,0)[F(A')`G(A') `\alpha_1(A')]{520}1a
\putmorphism(-150,50)(0,-1)[F(A')``=]{450}1l
\putmorphism(380,500)(0,-1)[\phantom{Y_2}` `F(u)]{450}1l
\putmorphism(950,500)(0,-1)[\phantom{Y_2}`G(A) `\alpha_0(A)]{450}1r
\putmorphism(380,50)(0,-1)[F(A')` `\alpha_0(A')]{450}1r
\putmorphism(350,500)(1,0)[F(A)`F(A)`=]{600}1a
 \putmorphism(950,500)(1,0)[\phantom{F(A)}`G(A) `\alpha_1(A)]{650}1a
 \putmorphism(1060,50)(1,0)[`G(A)`=]{500}1b

\putmorphism(1570,500)(0,-1)[\phantom{Y_2}``=]{450}1r
\put(1250,260){\fbox{$\Epsilon^\alpha_A$}}
\putmorphism(480,-400)(1,0)[`G(A') `=]{500}1a
\putmorphism(950,50)(0,-1)[\phantom{Y_2}``G(u)]{450}1r
\put(0,-160){\fbox{$\eta^\alpha_{A'}$}}
\efig
$$
for every 1v-cell $u:A\to A'$; 
\item the 2-cell 
$$\delta_{\alpha_1,f}=\quad
\bfig
 \putmorphism(-150,250)(1,0)[F(A)`F(A) `=]{500}1a
\put(-100,30){\fbox{$\eta^\alpha_A$}}
\putmorphism(380,250)(0,-1)[\phantom{Y_2}` `\alpha_0(A)]{450}1l
\putmorphism(950,250)(0,-1)[\phantom{Y_2}` `\alpha_0(B)]{450}1r
\putmorphism(350,250)(1,0)[F(A)`F(B)`F(f)]{600}1a
 \putmorphism(950,250)(1,0)[\phantom{F(A)}`G(B) `\alpha_1(B)]{600}1a
 \putmorphism(470,-200)(1,0)[`G(B)`G(f)]{500}1b
 \putmorphism(1060,-200)(1,0)[`G(B)`=]{500}1b
\putmorphism(1570,250)(0,-1)[\phantom{Y_2}``=]{450}1r
\put(550,10){\fbox{$(\alpha_0)_f$}}
\put(1250,30){\fbox{$\Epsilon^\alpha_B$}}
\putmorphism(-150,-200)(1,0)[F(A)`G(A) `\alpha_1(A)]{520}1a
\putmorphism(-150,250)(0,-1)[``=]{450}1l
\efig
$$
for every 1h-cell $f:A\to B$. 
\end{itemize}
\end{prop}

\begin{proof}
To prove axiom 1), the axiom 1) of $\alpha_0$ is used; 
the first part of the axiom 2) works directly, and in the second one use second part of the axiom 3) for $\alpha_0$; 
the first part of the axiom 3) works directly, and in the second one use second part of the axiom 2) for $\alpha_0$; 
in checking of all the three axioms also the rules $\Epsilon\x\eta$ are used. 
\end{proof}

\medskip

Observe that there is a way in the other direction:

\begin{prop} \prlabel{horiz->vertic}
Let $\alpha_1:F\Rightarrow G$ be a strong horizontal transformation between pseudo double functors acting between 
strict double categories $\Aa\to\Bb$. Suppose that for every 0h-cell $A$ the 1h-cell $\alpha_1(A)$ is a 1h-companion 
of some 1v-cell (with corresponding 2-cells $\Epsilon^\alpha_A$ and $\eta^\alpha_A$). 
Fix a choice of such 1v-cells for each $A$ and denote them by $\alpha_0(A)$. 
The following data define a vertical pseudonatural transformation $\alpha_0:F\Rightarrow G$:
\begin{itemize}
\item the 1v-cell $\alpha_0(A)$, for every 0-cell $A$ of $\Aa$;
\item the 2-cell 
$$(\alpha_0)_f=\quad
\bfig
\putmorphism(-150,0)(1,0)[F(A)`F(B)`F(f)]{600}1a
 \putmorphism(450,0)(1,0)[\phantom{F(A)}`G(B) `\alpha_1(B)]{680}1a
 \putmorphism(-150,-450)(1,0)[F(A)`G(A)`\alpha_1(A)]{600}1a
 \putmorphism(450,-450)(1,0)[\phantom{F(A)}`G(B) `G(f)]{680}1a

\putmorphism(-180,0)(0,-1)[\phantom{Y_2}``=]{450}1r
\putmorphism(1100,0)(0,-1)[\phantom{Y_2}``=]{450}1r
\put(350,-240){\fbox{$\delta_{\alpha_1,f}$}}
\putmorphism(450,450)(1,0)[F(B)`F(B) `=]{640}1a

\putmorphism(450,450)(0,-1)[\phantom{Y_2}``=]{450}1l
\putmorphism(1100,450)(0,-1)[\phantom{Y_2}``\alpha_0(B)]{450}1r
\put(680,250){\fbox{$\eta^\alpha_B$}} 

\putmorphism(-150,-450)(0,-1)[\phantom{Y_2}``\alpha_0(A)]{450}1l
\putmorphism(450,-450)(0,-1)[\phantom{Y_2}``=]{450}1l 
 \putmorphism(-150,-900)(1,0)[G(A)`G(A) `=]{580}1a
\put(20,-700){\fbox{$\Epsilon^\alpha_A$}} 

\efig
$$
for every 1h-cell $f:A\to B$; 
\item the 2-cell 
$$\delta_{\alpha_0,u}=
\bfig
 \putmorphism(-150,410)(1,0)[F(A)`F(A)  `=]{530}1a
\putmorphism(-160,400)(0,-1)[\phantom{Y_2}`F(A) `=]{380}1l
\putmorphism(370,400)(0,-1)[\phantom{Y_2}`G(A) `\alpha_0(A)]{380}1r
\putmorphism(-160,50)(0,-1)[\phantom{Y_2}`F(A')`F(u)]{430}1l
\putmorphism(370,50)(0,-1)[\phantom{Y_2}`G(A')`G(u)]{430}1r
\put(-20,240){\fbox{$\eta^\alpha_A$}}
\putmorphism(-70,20)(1,0)[``\alpha_1(A)]{330}1a
\put(-60,-160){\fbox{$(\alpha_1)_u$ }}
\putmorphism(-160,-350)(0,-1)[\phantom{Y_2}``\alpha_0(A')]{400}1l
\putmorphism(-70,-350)(1,0)[``\alpha_1(A')]{330}1a
\putmorphism(370,-350)(0,-1)[\phantom{Y_2}`G(A')`=]{400}1r
\putmorphism(-180,-750)(1,0)[G(A')` `=]{440}1b
\put(-20,-560){\fbox{$\Epsilon^\alpha_{A'}$ }}
\efig
$$
for every 1v-cell $u:A\to A'$. 
\end{itemize}
\end{prop}

By $\Epsilon\x\eta$-relations, there is a 1-1 correspondence between those strong vertical transformations whose 1v-cell components 
have 1h-companions 
and those strong horizontal transformations 
whose 1h-cell components are 1h-companions of some 1v-cells. 

\begin{cor} \colabel{bij}
Suppose that there is a connection on 1v-components of strong vertical transformations. Then there is a bijection between 
strong vertical transformations and those strong horizontal transformations whose 1h-cell components are 1h-companions of some 1v-cells.
\end{cor}

As a direct corollary of \prref{vertic->horiz} we get:

\begin{cor} \colabel{4 identities e-eta}
Suppose that the 1v-components of a strong vertical transformation $\alpha_0:F\Rightarrow G$ have 1h-companions $\alpha_1(A)$, 
for every 0-cell $A$ (with corresponding 2-cells $\Epsilon^\alpha_A$ and $\eta^\alpha_A$), and define the 2-cells 
$(\alpha_1)_u$ and $\delta_{\alpha_1,f}$ as in \prref{vertic->horiz}. The following identities then follow: 
$$
\bfig
\putmorphism(-150,500)(1,0)[F(A)`F(B)`F(f)]{600}1a
 \putmorphism(450,500)(1,0)[\phantom{F(A)}`G(B) `\alpha_1(B)]{640}1a

 \putmorphism(-150,50)(1,0)[G(A)`G(B)`G(f)]{600}1a
 \putmorphism(450,50)(1,0)[\phantom{F(A)}`G(B) `=]{640}1a

\putmorphism(-180,500)(0,-1)[\phantom{Y_2}``\alpha_0(A)]{450}1l
\putmorphism(450,500)(0,-1)[\phantom{Y_2}``]{450}1r
\putmorphism(300,500)(0,-1)[\phantom{Y_2}``\alpha_0(B)]{450}0r
\putmorphism(1100,500)(0,-1)[\phantom{Y_2}``=]{450}1r
\put(0,260){\fbox{$(\alpha_0)_f$}}
\put(700,270){\fbox{$\Epsilon^\alpha_B$}}
\efig
\quad=\quad
\bfig
\putmorphism(-150,500)(1,0)[F(A)`F(B)`F(f)]{600}1a
 \putmorphism(450,500)(1,0)[\phantom{F(A)}`G(B) `\alpha_1(B)]{680}1a
 \putmorphism(-150,50)(1,0)[F(A)`G(A)`\alpha_1(A)]{600}1a
 \putmorphism(450,50)(1,0)[\phantom{F(A)}`G(B) `G(f)]{680}1a

\putmorphism(-180,500)(0,-1)[\phantom{Y_2}``=]{450}1r
\putmorphism(1100,500)(0,-1)[\phantom{Y_2}``=]{450}1r
\put(350,260){\fbox{$\delta_{\alpha_1,f}$}}
\putmorphism(-150,-400)(1,0)[F(A')`G(A'). `=]{640}1a

\putmorphism(-180,50)(0,-1)[\phantom{Y_2}``\alpha_0(A)]{450}1l
\putmorphism(450,50)(0,-1)[\phantom{Y_2}``]{450}1l
\putmorphism(610,50)(0,-1)[\phantom{Y_2}``=]{450}0l 
\put(20,-180){\fbox{$\Epsilon^\alpha_A$}} 
\efig
$$
and
$$
\bfig
\putmorphism(-150,0)(1,0)[F(A)`F(A)`=]{600}1a
 \putmorphism(450,0)(1,0)[\phantom{F(A)}`G(B) `F(f)]{640}1a

 \putmorphism(-150,-450)(1,0)[F(A)`G(A)`\alpha_1(A)]{600}1a
 \putmorphism(450,-450)(1,0)[\phantom{F(A)}`G(B) `G(f)]{640}1a

\putmorphism(-180,0)(0,-1)[\phantom{Y_2}``=]{450}1l
\putmorphism(450,0)(0,-1)[\phantom{Y_2}``]{450}1r
\putmorphism(300,0)(0,-1)[\phantom{Y_2}``\alpha_0(A)]{450}0r
\putmorphism(1100,0)(0,-1)[\phantom{Y_2}``\alpha_0(B)]{450}1r
\put(50,-200){\fbox{$\eta^\alpha_A$}}
\put(650,-200){\fbox{$(\alpha_0)_f$}}
\efig
\quad=\quad
\bfig
\putmorphism(-150,0)(1,0)[F(A)`F(B)`F(f)]{600}1a
 \putmorphism(450,0)(1,0)[\phantom{F(A)}`G(B) `\alpha_1(B)]{680}1a
 \putmorphism(-150,-450)(1,0)[F(A)`G(A)`\alpha_1(A)]{600}1a
 \putmorphism(450,-450)(1,0)[\phantom{F(A)}`G(B) `G(f)]{680}1a

\putmorphism(-180,0)(0,-1)[\phantom{Y_2}``=]{450}1r
\putmorphism(1100,0)(0,-1)[\phantom{Y_2}``=]{450}1r
\put(350,-240){\fbox{$\delta_{\alpha_1,f}$}}
\putmorphism(450,450)(1,0)[F(B)`F(B) `=]{640}1a

\putmorphism(450,450)(0,-1)[\phantom{Y_2}``=]{450}1l
\putmorphism(1100,450)(0,-1)[\phantom{Y_2}``\alpha_0(B)]{450}1l 
\put(600,250){\fbox{$\eta^\alpha_B$}} 
\efig
$$
for every 1h-cell $f:A\to B$; 
$$
\bfig
 \putmorphism(-150,500)(1,0)[F(A)`G(A)  `\alpha_1(A)]{600}1a
\putmorphism(-180,500)(0,-1)[\phantom{Y_2}`F(A') `F(u)]{450}1l
\putmorphism(-150,-400)(1,0)[G(A')` `=]{480}1a
\putmorphism(-180,50)(0,-1)[\phantom{Y_2}``\alpha_0(A')]{450}1l
\putmorphism(450,50)(0,-1)[\phantom{Y_2}`G(A')`=]{450}1r
\putmorphism(450,500)(0,-1)[\phantom{Y_2}`G(A') `G(u)]{450}1r
\put(0,260){\fbox{$(\alpha_1)_u$}}
\putmorphism(-180,50)(1,0)[\phantom{F(A)}``\alpha_1(A')]{500}1a
\put(40,-170){\fbox{$\Epsilon^\alpha_{A'}$}}
\efig=
\bfig
 \putmorphism(-150,500)(1,0)[F(A)`F(A) `=]{500}1a
 \putmorphism(450,500)(1,0)[` `\alpha_1(A)]{380}1a
\putmorphism(-180,500)(0,-1)[\phantom{Y_2}`F(A') `F(u)]{450}1l
\put(0,-160){\fbox{$\delta_{\alpha_0,u}$}}
\putmorphism(-150,-400)(1,0)[F(A'')`F(A'') `=]{500}1a
\putmorphism(-180,50)(0,-1)[\phantom{Y_2}``\alpha_0(A')]{450}1l
\putmorphism(380,500)(0,-1)[\phantom{Y_2}` `\alpha_0(A)]{450}1l
\putmorphism(380,50)(0,-1)[G(A)` `G(u)]{450}1r
\put(550,290){\fbox{$\Epsilon^\alpha_A$} }
\putmorphism(940,500)(0,-1)[G(A)`G(A)`=]{450}1r
\putmorphism(470,60)(1,0)[``=]{360}1a
\efig
$$
and 
$$
\bfig
 \putmorphism(-150,500)(1,0)[F(A)`F(A)  `=]{600}1a
\putmorphism(-180,500)(0,-1)[\phantom{Y_2}`F(A) `=]{450}1l
\putmorphism(-150,-400)(1,0)[F(A')` `\alpha_1(A')]{480}1a
\putmorphism(-180,50)(0,-1)[\phantom{Y_2}``F(u)]{450}1l
\putmorphism(450,50)(0,-1)[\phantom{Y_2}`G(A')`G(u)]{450}1r
\putmorphism(450,500)(0,-1)[\phantom{Y_2}`G(A) `\alpha_0(A)]{450}1r
\put(50,300){\fbox{$\eta^\alpha_A$}}
\putmorphism(-180,50)(1,0)[\phantom{F(A)}``\alpha_1(A)]{520}1a
\put(20,-170){\fbox{$(\alpha_1)_u$}}
\efig=
\bfig
 \putmorphism(-400,-400)(1,0)[` `\alpha_1(A')]{380}1a
 \putmorphism(80,500)(1,0)[F(A)`F(A) `=]{500}1a
\putmorphism(-480,50)(0,-1)[F(A')`F(A')`=]{450}1r
\putmorphism(-380,60)(1,0)[``=]{360}1a
\putmorphism(70,500)(0,-1)[\phantom{Y_2}`F(A') `F(u)]{450}1l
\put(220,270){\fbox{$\delta_{\alpha_0,u}$}}
\putmorphism(80,-400)(1,0)[G(A')`G(A') `=]{500}1a
\putmorphism(70,50)(0,-1)[\phantom{Y_2}``\alpha_0(A')]{450}1r
\putmorphism(580,500)(0,-1)[\phantom{Y_2}` `\alpha_0(A)]{450}1r
\putmorphism(580,50)(0,-1)[G(A)` `G(u)]{450}1r
\put(-280,-160){\fbox{$\eta^\alpha_{A'}$} }
\efig
$$
for every 1v-cell $u:A\to A'$. 
\end{cor}

In the following Proposition $\eta_{F(u)}$ and $\Epsilon_{G(u)}$ are given as in \cite[Lemma 3.16]{Shul}.

\begin{prop} \prlabel{invertible delta}
Given a strong vertical transformation $\alpha_0$ under conditions of \prref{vertic->horiz}. 
Let $u:A\to A'$ be a 1v-cell with a 1h-companion $f=u_*$. Then the inverse of the 2-cell $\delta_{\alpha_1,u_*}$ is given by: 
$$\delta_{\alpha_1,u_*}^{-1}=
\bfig
\putmorphism(-700,500)(1,0)[F(A)`F(A)`=]{530}1a
\putmorphism(-150,500)(0,-1)[\phantom{Y_2}` `F(u)]{450}1l
\putmorphism(-700,500)(0,-1)[\phantom{Y_2}`F(A) `=]{450}1l
\putmorphism(-610,50)(1,0)[``F(u)_*]{360}1b
\put(-630,260){\fbox{$\eta_{F(u)}$}}

 \putmorphism(-40,50)(1,0)[``=]{320}1b

\put(550,30){\fbox{$\delta_{\alpha_0,u}$}}
\putmorphism(-150,-400)(1,0)[F(A')`G(A') `\alpha_1(A')]{520}1a
\putmorphism(-150,50)(0,-1)[F(A')``=]{450}1l
\putmorphism(380,500)(0,-1)[\phantom{Y_2}` `F(u)]{450}1l
\putmorphism(950,500)(0,-1)[\phantom{Y_2}`G(A) `\alpha_0(A)]{450}1r
\putmorphism(380,50)(0,-1)[F(A')` `\alpha_0(A')]{450}1r
\putmorphism(350,500)(1,0)[F(A)`F(A)`=]{600}1a
 \putmorphism(950,500)(1,0)[\phantom{F(A)}`G(A) `\alpha_1(A)]{580}1a
 \putmorphism(1060,50)(1,0)[`G(A)`=]{500}1b

\putmorphism(1520,500)(0,-1)[\phantom{Y_2}``=]{450}1r
\put(1250,260){\fbox{$\Epsilon^\alpha_A$}}
\putmorphism(480,-400)(1,0)[`G(A') `=]{500}1a
\putmorphism(950,50)(0,-1)[\phantom{Y_2}``G(u)]{450}1r
\put(0,-160){\fbox{$\eta^\alpha_{A'}$}}

\putmorphism(1640,50)(1,0)[`G(A')`G(u)_*]{440}1a
\putmorphism(2070,50)(0,-1)[\phantom{Y_2}``=]{450}1r
\putmorphism(1520,50)(0,-1)[\phantom{Y_2}`G(A')`G(u)]{450}1l
\putmorphism(1630,-400)(1,0)[`G(A'). `=]{440}1a
\put(1660,-160){\fbox{$\Epsilon_{G(u)}$}}

\efig
$$
\end{prop}

\begin{proof}
Use axiom 1) for $\alpha_0$ and (6.3) of \cite[Definition 6.1]{Shul1}, together with $\Epsilon\x\eta$-relations.  
\end{proof}

\medskip

Note that the above inverse of $\delta_{\alpha_1,u_*}$ is in fact the image of the known pseudofunctor $V\Dd\to H\Dd$ from the 
vertical 2-category to the horizontal one of a given double category $\Dd$ in which all 1v-cells have 1h-companions. For a horizontally 
globular 2-cell $a$ with a left 1v-cell $u$ and 
a right 1v-cell $v$, the image by this functor of $a$ is given by $[\eta_u \vert a\vert\Epsilon_v]$ (horizontal composition of 2-cells).

\begin{rem}
One could start with a vertical transformation $\alpha_0$ (for which $\delta_{\alpha_0,u}=\Id$ for all 1v-cells $u:A\to A'$) 
and define a horizontal transformation $\alpha_1$ setting $\delta_{\alpha_1,f}=\Id$ for all 1h-cells $f:A\to B$ and 
defining $(\alpha_1)_u$  as in \prref{vertic->horiz}. Though, in order for $\alpha_1$ to satisfy the corresponding axiom 1), 
one needs to assume the first two identities of \coref{4 identities e-eta}. 
\end{rem}

\subsection{Embedding $PsDbl^*_3$ into $DblPs$} \sslabel{embed transf}

It remains to show how to turn vertical transformations into double pseudonatural transformations. We will assume that 
1v-cell components of vertical transformations have 1h-companions. 

Observe that for vertical transformations the 2-cells $\delta_{\alpha_0,u}$ are identities. Moreover,  
we know that vertical transformations are particular cases of vertical pseudonatural transformations (\rmref{hor tr as subcase}), 
and that strong horizontal transformations are particular cases of horizontal pseudonatural transformations. 
By \prref{vertic->horiz} we have that a vertical transformation $\alpha_0$ determines a strong horizontal transformation. 
So far we have axiom (T1) of \deref{double 2-cells}. 
Furthermore, by \prref{invertible delta} we have in particular that for all 1v-cell components $\alpha_0(A)$ of vertical transformations 
the 2-cells $\delta_{\alpha_1,\alpha_1(A)}$ are invertible. Then we have that the axiom (T2) is fulfilled.  
Observe that setting $\Theta^\alpha_A=\Epsilon^\alpha_A$, by the first and third identities in \coref{4 identities e-eta} 
we have a $\Theta$-double pseudonatural transformation between pseudo double functors. Due to \prref{teta->double} we have 
indeed a double pseudonatural transformation, as we wanted. (Actually, thanks to the $\Epsilon\x\eta$-relations, by the first identity in 
\coref{4 identities e-eta}, axiom 1) for the horizontal pseudonatural transformation $\alpha_1$ holds 
if and only if axiom (T3-1) for $t^\alpha_f$ in \prref{teta->double} holds.)

Moreover, we may deduce the following bijective correspondence $t^\alpha_f\leftrightarrow\delta_{\alpha_1,f}$:
$$
t^\alpha_f=
\scalebox{0.86}{
\bfig
\putmorphism(-150,450)(1,0)[F(A)`F(B)`F(f)]{500}1a
 \putmorphism(350,450)(1,0)[\phantom{F(A)}`G(B) `\alpha_1(B)]{580}1a
 \putmorphism(-150,50)(1,0)[F(A)`G(A)`\alpha_1(A)]{500}1a
 \putmorphism(350,50)(1,0)[\phantom{F(A)}`G(B), `G(f)]{580}1a

\putmorphism(-180,450)(0,-1)[\phantom{Y_2}``=]{400}1r
\putmorphism(900,450)(0,-1)[\phantom{Y_2}``=]{400}1r
\put(250,260){\fbox{$\delta_{\alpha_1,f}$}}
\putmorphism(-150,-350)(1,0)[F(A')`G(A') `=]{500}1a

\putmorphism(-180,50)(0,-1)[\phantom{Y_2}``\alpha_0(A)]{400}1r
\putmorphism(350,50)(0,-1)[\phantom{Y_2}``]{400}1l
\putmorphism(510,50)(0,-1)[\phantom{Y_2}``=]{400}0l 
\put(120,-150){\fbox{$\Epsilon^\alpha_A$}} 
\efig}
\qquad\quad
\delta_{\alpha_1,f}=
\scalebox{0.86}{
\bfig
 \putmorphism(-150,250)(1,0)[F(A)`F(A) `=]{500}1a
\put(50,30){\fbox{$\eta^\alpha_A$}}
\putmorphism(380,250)(0,-1)[\phantom{Y_2}` `\alpha_0(A)]{400}1r
\putmorphism(350,250)(1,0)[F(A)`F(B)`F(f)]{500}1a
 \putmorphism(850,250)(1,0)[\phantom{F(A)}`G(B) `\alpha_1(B)]{500}1a
 \putmorphism(450,-150)(1,0)[`G(B)`G(f)]{420}1b
 \putmorphism(960,-150)(1,0)[`G(B),`=]{420}1b
\putmorphism(1370,250)(0,-1)[\phantom{Y_2}``=]{400}1r
\put(800,30){\fbox{$t^\alpha_f$}}
\putmorphism(-150,-150)(1,0)[F(A)`G(A) `\alpha_1(A)]{520}1b
\putmorphism(-150,250)(0,-1)[``=]{400}1r
\efig}
$$
and complete the bijection $t^\alpha_f\leftrightarrow(\alpha_0)_f$:
$$(\alpha_0)_f=
\scalebox{0.86}{
\bfig
\putmorphism(350,400)(1,0)[F(B)`F(B) `=]{540}1a
\putmorphism(350,400)(0,-1)[\phantom{Y_2}``=]{400}1l
\putmorphism(900,400)(0,-1)[\phantom{Y_2}``\alpha_0(B)]{400}1r
\put(530,220){\fbox{$\eta^\alpha_B$}} 

\putmorphism(-150,0)(1,0)[F(A)`F(B)`F(f)]{500}1a
 \putmorphism(350,0)(1,0)[\phantom{F(A)}`G(B) `\alpha_1(B)]{540}1a
 \putmorphism(-150,-400)(1,0)[F(A)`G(A).`G(f)]{1040}1a

\putmorphism(-130,0)(0,-1)[\phantom{Y_2}``=]{400}1r
\putmorphism(900,0)(0,-1)[\phantom{Y_2}``=]{400}1r
\put(530,-190){\fbox{$t^\alpha_f$}} 
\efig}
$$

\begin{rem}
Given the $\Epsilon\x\eta$-relations, by the properties developed in this and the previous subsection, 
axiom 1) for the horizontal pseudonatural transformation $\alpha_1$ holds if and only if 
axiom (T3-1) for $t^\alpha_f$ in \prref{teta->double} holds, if and only if 
axiom 1) for the vertical pseudonatural transformation $\alpha_0$ holds. 
\end{rem}




\section{Tricategorical pullbacks and (co)products} \selabel{3-pull(co)pr}

For a notion of enrichment over a (1-strict) tricategory $V$ we need some notion of a monoidal structure on $V$, 
while for a notion of an internal catgeory in $V$ we need some notion of tricatgeorical pullbacks. We define 
both such notions in this section, where for the monoidal structure we consider tricategorical products.

\subsection{Tricategorical pullbacks}

In this subsection we define tricategorical pullbacks, that is pullbacks in tricategories. We will also call them shortly 3-pullbacks. 

\begin{defn}
A 3-pullback over a 0-cell $S$ with respect to 1-cells $f:M\to S$ and $g:N\to S$ in a tricategory $V$ is given by: 
a 0-cell $P$, 1-cells $p_1:P\to M, p_2:P\to N$ and an equivalence 2-cell $\omega:  gp_2\Rightarrow fp_1$ so that 
   \begin{itemize}
   \item for every 0-cell $T$, 1-cells $q_1: T\to M, q_2:T\to N$ and equivalence 2-cell $\sigma: gq_2\Rightarrow fq_1$ 
	there are a 1-cell $u:T\to P$, equivalence 2-cells $\zeta_1: p_1u\Rightarrow q_1$ and $\zeta_2: q_2\Rightarrow p_2u$ and 
	an isomorphism 3-cell 
	$$\Sigma: \threefrac{\Id_{g}\ot\zeta_2}{\omega\ot\Id_u}{\Id_f\ot\zeta_1} \Rrightarrow\sigma;$$
		\item for all 1-cells $u,v: T\to P$, 2-cells $\alpha:p_1u\Rightarrow p_1v, \beta: p_2u\Rightarrow p_2v$ and a 3-cell 
		$\kappa:\frac{\Id_t\ot\alpha}{\omega\ot\Id_v}\Rrightarrow\frac{\omega\ot\Id_u}{\Id_g\ot\beta}$ 
		there are a 2-cell $\gamma: u\Rightarrow v$ and isomorphism 3-cells $\Gamma_1:\Id_{p_1}\ot\gamma\Rightarrow\alpha, 
		\Gamma_2:\Id_{p_2}\ot\gamma\Rightarrow\beta$; 
		\item for all 2-cells $\gamma,\gamma': u\Rightarrow v$ and a 3-cell $\chi: \omega\ot\gamma\Rrightarrow\omega\ot\gamma'$ 
		such that the following two transversal compositions of 3-cells are equal:
		$$\frac{\Id_t\ot(\Id_{p_1}\ot\gamma)}{\omega\ot\Id_v} \stackrel{ \frac{\Id\ot\Gamma_1}{\Id}} {\Rrightarrow} 
		\frac{\Id_t\ot\alpha}{\omega\ot\Id_v} \stackrel{\kappa}{\Rrightarrow} \frac{\omega\ot\Id_u}{\Id_s\ot\beta} 
		 \stackrel{ \frac{\Id}{\Id\ot(\Gamma'_2)^{-1}}} {\Rrightarrow}
		\frac{\omega\ot\Id_u}{\Id_s\ot(\Id_{p_2}\ot\gamma')}$$
   $$\Downarrow \displaystyle{\frac{a}{\Id}} \hspace{7cm} \Downarrow\displaystyle{\frac{a}{\Id}} $$
    $$ \frac{(\Id_t\ot\Id_{p_1})\ot\gamma}{\omega\ot\Id_v} \hspace{0,2cm}\stackrel{\xi}{\Rrightarrow}\hspace{0,2cm} \omega\ot\gamma 
		\hspace{0,2cm}\stackrel{\chi}{\Rrightarrow}\hspace{0,2cm} \omega\ot\gamma' \hspace{0,2cm}\stackrel{\xi^{-1}}{\Rrightarrow} 
		\hspace{0,2cm} \frac{\omega\ot\Id_u}{(\Id_s\ot\Id_{p_2})\ot\gamma'}$$
		there exists a unique 3-cell $\delta:\gamma\Rrightarrow\gamma'$ such that $\Id_\omega\ot\delta=\chi$. 
	\end{itemize}
\end{defn}

A 3-pullback with notations as in the above Definition we will denote shortly by $(P,M,N,S,p_1,p_2;f,g)$, or 
$(M\times_SN, f,g)$.

\subsection{Tricategorical (co)products} \sslabel{3-(co)prod}

In the literature there are bicategorical (co)products, that is, (co)products in bicategories. In this section we propose 
a definition for their tricategorical companions, we will shortly also call them 3-(co)products. Before defining them let us 
remark what data comprise a 2-product in a bicategory $\K$. A 2-product consists of: 1) a 0-cell $A\times B$ for 0-cells $A,B\in\K$ and 
1-cells $p_1:A\times B\to A, p_2:A\times B\to B$, and 2) for every $X\in\K$ a natural equivalence of categories: 
$F:\K(X,A\times B)\to\K(X,A)\times\K(X,B)$. Observe that the point 2) means that $F$ is an equivalence 1-cell in the 2-category 
$\Cat_2$ of categories, and that the 2-functor $\K(X,-):\K\to\Cat_2$ sends the product 0-cell $A\times B$ in $\K$ to the 
1-product in the 1-category of categories $\Cat_1$. 

With this in mind we define:

\begin{defn}
A 3-product of 0-cells $A$ and $B$ in a tricategory $V$ consists of: 
\begin{itemize}
\item a 0-cell $A\times B$ and 1-cells $p_1:A\times B\to A, p_2:A\times B\to B$, such that
\item for every $X\in V$ there is a biequivalence of bicategories 
$$V(X,A\times B)\simeq V(X,A)\times V(X,B)$$
where on the right hand-side the 2-product in the 2-category $\Bicat_2$ of bicategories, pseudofunctors and icons 
\cite{Lack-i} is meant. 
\end{itemize}
\end{defn}

The second point in the above definition says that there is an equivalence 1-cell in the tricategory $\Bicat_3$ of bicategories 
(bicategories, pseudofunctors, pseudonatural transformations and modifications) up to which the ``trifunctor'' 
$V(X,-):V\to\Bicat_3$ sends the product 0-cell $A\times B$ to the 2-product of bicategories. 

For 3-products of $k>2$ 
0-cells the projection 1-cells to the first $i$ and last $j$ components we will write by $p^k_{1\dots i}$ and 
$p^k_{k-j+1\dots k}$, respectively.

It is useful to unpack the above definition. We will do it for the dual notion of a 3-coproduct in $V$. In this case the 
natural biequivalence of bicategories in question takes the form $V(A\amalg B,X)\simeq V(A,X)\amalg V(B,X)$ and the 
analogous ``trifunctor'' $V(-,X):V\to\Bicat$ is now contravariant. 

\begin{defn}
A 3-coproduct of 0-cells $A$ and $B$ in a tricategory $V$ consists of: 
a 0-cell $A\amalg B$ and 1-cells $\iota_1:A\to A\amalg B, \iota_2:B\to A\amalg B$, such that
 \begin{itemize}
  \item for every 0-cell $T$ and 1-cells $f_1: A\to T, f_2: B\to T$ 
	      there are a 1-cell $u: A\amalg B\to T$ and equivalence 2-cells $\zeta_i: u\iota_i\Rightarrow f_i, i=1,2$; 
  		 \item 
			for all 1-cells $u,v: A\amalg B\to T$ and 2-cells $\alpha: u\iota_1\Rightarrow v\iota_1$ and $\beta: u\iota_2\Rightarrow v\iota_2$, 
	     there are a 2-cell $\gamma:u\Rightarrow v$ and 3-cells $\Gamma_1: \gamma\ot\Id_{\iota_1}\Rrightarrow\alpha$ and 
			$\Gamma_2: \gamma\ot\Id_{\iota_2}\Rrightarrow\beta$; 
				\item 
				for every two 2-cells $\gamma,\gamma':u\Rightarrow v$ and every two 3-cells 
				$\chi_i: \gamma\ot\iota_i\Rrightarrow\gamma'\ot\iota_i, i=1,2$ there is a unique 3-cell $\Gamma: \gamma\Rrightarrow\gamma'$ such that 
					$\chi_i=\Gamma\ot\Id_{\Id_{\iota_i}}, i=1,2$.
       \end{itemize}
\end{defn}

We say that a tricategory $V$ has small 3-(co)products if it has them for any family of 0-cells indexed by (elements of) a set.



\section{Categories internal in 1-strict tricategories} \selabel{int3}

We are interested in internalization in ambient weak $n$-categories, for $n=1,2,3$, that have an underlying 1-category. 
Such ambient weak $n$-categories we call {\em 1-strict}. These can be various ``categories of categories''. 
A folklore example of internal categories are pseudo double categories for which this condition is fulfilled: they are internal 
categories in the 2-category of categories. 
Also the tricategory $\Bicat_3$ of bicategories, pseudofunctors, pseudonatural transformations and modifications is 1-strict.  
In particular, in this paper we are interested in 1-strict tricategories $\DblPs$ and $2Cat_{wk}$, of 2-categories, 
pseudofunctors, weak natural transformations and modifications... 

Observe that while in a 1-strict tricategory 
the associativity and unitality of the composition of 1-cells are strict, and also the 3-cells $\pi_{D,C,B,A},\lambda_{B,A}, 
\rho_{B,A}, \mu_{B,A}$ (modifications from (TD7) and (TD8) from \cite[Definition 2.2]{GPS} evaluated at 1-cells $D,C,B,A$) 
are identities, the horizontal associativity and unitality of 2-cells (and the interchange law) work up to isomorphism.  

Let $V$ be a 1-strict tricategory, we want to define a category internal in $V$. In \cite[Definition 2.11]{DH} an 
internal category in a Gray-category was defined. Therein, the definition of a Gray-category is based on whisker, so that 
instead of a full interchange law there appears an isomorphism 3-cell $sw$ (with an additional rule for whiskering). 
From the point of view of $V$, the 3-cell $sw$ can be defined as the following transversal composition of 3-cells: 
$$ \bigg(\frac{[\alpha\vert\Id]}{[\Id\vert\beta]}\bigg) \stackrel{\xi}{\Rrightarrow} 
\bigg[\big(\frac{\alpha}{\Id}\big) \bigg\rvert \big(\frac{\Id}{\beta}\big)\bigg] \cong 
\bigg[\big(\frac{\Id}{\alpha}\big) \bigg\rvert \big(\frac{\beta}{\Id}\big)\bigg]
\stackrel{\xi^{-1}}{\Rrightarrow}
\bigg(\frac{[\Id\vert\beta]}{[\alpha\vert\Id]}\bigg), $$
for 2-cells $\alpha$ and $\beta$, where the middle isomorphism stands for the composition of one ``vertical'' unity constraint with 
the inverse of the other in the appropriate order, in both coordinates. 
Here $[\alpha\vert\beta]$ denotes 
the horizontal composition $\beta\ot\alpha$, and the fractions denote the vertical one. We consider by the coherence Theorem 
\cite[Theorem 1.5]{GPS} that these unity constraints are identities, so $sw$ will be identity. 
Another difference with respect to \cite[Definition 2.11]{DH} is that therein the authors work with 1-pullbacks (a Gray-category 
is a 1-strict tricategory), while we are working with 3-pullbacks introduced in \seref{3-pull(co)pr}. 

As a matter of fact, we will need only certain 3-pullbacks. For this reason we define iterated 3-pullbacks, analogously to iterated 
2-pullbacks from \cite{GP}. Let $B_1\Doublearrow B_0$ be 1-cells in a 1-strict tricategory $V$. 
Iterated $n$-fold composition of the span $B_1\Doublearrow B_0$ in $V$ can be defined via 3-pullbacks. Such $n$-fold composition 
we call {\em iterated 3-pullbacks}. We denote them by (any of the distributions of the parentheses on) $B_1\times_{B_0}B_1\times_{B_0}\dots\times_{B_0}B_1$ 
($n$ times). We will write this shortly as $B_1^{(n)_0}$, regardless the choice of the distributions, which will be clear from the context. Let 
$B_1^{(0)_0}$ be $B_0$. For the projections $p_i: B_1^{(n)_0} \to B_1$ for $i=1,2\dots n$ we will use lexicographical order. 

\begin{rem} \rmlabel{pb order}
The 3-pullback $(B_1\times_{B_0}B_1,s,t)$ we will consider with the following order of factors: 
$$
\bfig
\putmorphism(-30,400)(1,0)[\phantom{B_1\times_{B_0}B_1}`\phantom{B}`p_2]{500}1a
\putmorphism(500,400)(0,-1)[B_1`B_0.`t]{400}1r
\putmorphism(0,400)(0,-1)[B_1\times_{B_0}B_1`B_1`p_1]{400}1l
\putmorphism(0,0)(1,0)[\phantom{B_1}`\phantom{C}`s]{500}1b
\efig
$$
The labels $s$ and $t$ are suggestive for the case when $B_1$ is a hom-set, then as the diagram indicates, the 3-pullback $B_1\times_{B_0}B_1$ 
is read from right to left, although the projections are labeled in the lexicographical order. 
\end{rem}

In the next definition, to simplify the notation, the unsubscribed symbol $\times$ will stand for $\times_{B_0}$ at many places.

\begin{defn} \delabel{int3}
Let $V$ be a 1-strict tricategory. A category internal in $V$ consists of: 
\begin{enumerate}
\item 1-cells $B_1\Doublearrow B_0$, which we call source and target morphisms, for which the iterated 3-pullbacks $B_1^{(n)_0},n\in\N$ exist; 
\item 1-cells: $B_1\times_{B_0}B_1\stackrel{c}{\to} B_1$ composition and $u: B_0\to B_1$ unit (or identity) morphism; 
\item equivalence 2-cells $a^*: c\ot(id_{B_1}\times_{B_0} c) \Rightarrow c\ot(c\times_{B_0} id_{B_1}), \hspace{0,2cm} 
l^*: c\ot(u\times_{B_0}id_{B_1})\Rightarrow id_{B_1}$ and $r^*: c\ot(id_{B_1}\times_{B_0}u)\Rightarrow id_{B_1}$ in $V$; 
\item 3-cells 
$$\pi^*: \threefrac{\Id_c \ot(\Id_{id_{B_1}}\times a^*)}{a^*\ot\Id_{1\times c\times 1}}{\Id_c \ot (a^*\times \Id_{id_{B_1}})}
\Rrightarrow
\threefrac{a^* \ot\Id_{1\times 1\times c}}{\Id_c \ot\Nat_{(c\times 1)(1\times 1\times c)}}{a^* \ot\Id_{c\times 1\times 1}}$$
$$\mu^*: \Id_c\ot(\Id_{id_{B_1}}\times r^*) \Rrightarrow \frac{a^*\ot \Id_{id_{B_1}\times u\times id_{B_1}}}{\Id_c\ot(l^*\times id_{B_1})}$$

$$\lambda^*: \Id_c\ot(\Id_{id_{B_1}}\times l^*) \Rrightarrow \threefrac{a^*\ot \Id_{id_{B_1}\times id_{B_1}\times u}}
{\Id_c\ot\Nat_{(c\times id_{B_1})\ot(id_{B_1}\times id_{B_1}\times u)}}{l^*\ot\Id_c}$$

$$\rho^*: \frac{a^*\ot\Id_{u\times id_{B_1}\times id_{B_1}}}{\Id_c\ot (r^*\times id_{B_1})}\Rrightarrow 
\threefrac{\Id_c\ot\Nat_{(c\times id_{B_1})\ot(u\times id_{B_1}\times id_{B_1})}}{r^*\ot\Id_c}
{\Nat_{id_{B_1}\ot c}}$$

$$\epsilon^*: l^*\ot\Id_u \Rrightarrow \frac{\Id_c\ot\Nat_{(id_{B_1}\times u)\ot u}}{r^*\ot\Id_u}$$
%
which satisfy axioms (IT-1) - (IT-5) in the Appendix and symmetric versions of (IT-1), (IT-3) and (IT-4) 
(here the 2-cells $\nu$ are all identities, see the Remark below); 
\end{enumerate}
the above data should moreover satisfy the following compatibility conditions: 
$$sp_2=tp_1, \hspace{0,3cm} su=id_{B_0}=tu, \hspace{0,3cm} sc=sp_1, \hspace{0,3cm} tc=tp_2,$$ 
$$id_s\ot l^*=id_s=id_s\ot r^*, \hspace{0,3cm} id_t\ot l^*=id_t=id_t\ot r^*, \hspace{0,3cm} id_s\ot a^*=id_{sp_1}, 
\hspace{0,3cm} id_t\ot a^*=id_{tp_3},$$ 
$$\Id_{id_{s}}\ot \pi^*=\Id_{id_{sp_1}}, \hspace{0,3cm} \Id_{id_{t}}\ot \pi^*=\Id_{id_{tp_4}},$$
$$\Id_{id_{s}}\ot \mu^*=\Id_{id_{s}}\ot\lambda^*= \Id_{id_{s}}\ot\rho^*=\Id_{id_{sp_1}}, \hspace{0,3cm} 
\Id_{id_{t}}\ot \mu^*=\Id_{id_{t}}\ot\lambda^*= \Id_{id_{t}}\ot\rho^*=\Id_{id_{tp_2}},$$ 
where $p_i, i=1,2,3,4$, are 1-cells projections from the corresponding pullbacks in $V$. 
\end{defn}

\begin{rem} \rmlabel{items int3}
When writing out the 3-cells (and the axioms) in our definition, 
the following should be kept in mind. 
\begin{enumerate} [a)]
\item 
We will identify 1-cells $(id_{B_1}\times u)\ot c$ (acting on $B_1\times_{B_0}B_1$) and $c\times u$ (acting on 
$(B_1\times_{B_0}B_1)\times_{B_0}B_0$), by suppressing the isomorphism 1-cells (for the associativity and unity of the 3-pullback) between 
their domains. We do similar for $u$ and their symmetric counterparts. Recall that by 1-strictness of $V$ one has $c\ot id_{B_1}=c$.  
\item We explain the (co)domains of the 2-cells $\Id_\bullet\times a^*, \Id_\bullet\times r^*, \Id_\bullet\times l^*$ 
and their symmetric counterparts in item 4 in the definition above. Given a 2-cell $\alpha:G\ot F\Rightarrow G'\ot F'$, by abuse of notation, 
by $\alpha\times\id_{B_1}$ we will mean the induced 2-cell $(G\times 1)\ot(F\times 1)\Rightarrow (G'\times 1)\ot(F'\times 1)$. 
(Observe that by the 3-pullback property, between $(G\times 1)\ot(F\times 1)$ and $(G\ot F)\times 1$ there exists a (possibly non-isomorphism) 2-cell $\gamma$.)   
\item The naturality identity 2-cells we will sometimes draw explicitly and denote them all by $\nu$, or we will just write ``='' between two equal 
compositions of 1-cells. Here we refer to the 1-cells of the form $G\times F=(G\ot 1)\times(1\ot F)=(1\ot F)\times(G\ot 1)$. 

\item In order to simplify the diagrams and the definition, we could want the following two vertical compositions of 
horizontal compositions of 2-cells to be equal:
$$
\xymatrix@C=8pt{  \ar@/^1pc/[rr]^{}\ar@{}[rr]| {\Downarrow \Id} \ar@/_1pc/[rr]_{ } &  &  
\ar@/^1pc/[rr]^{}\ar@{}[rr]|{\Downarrow \Id} \ar@/_1pc/[rr]_{ } &  & }  
\qquad 
\xymatrix@C=8pt{  \ar@/^1pc/[rr]^{}\ar@{}[rr]| {\Downarrow \Id} \ar@/_1pc/[rr]_{ } &  &  
\ar@/^1pc/[rr]^{}\ar@{}[rr]|{\Downarrow \alpha} \ar@/_1pc/[rr]_{ } &  & }  \vspace{-0,2cm}$$
$$
\xymatrix@C=8pt{  \ar@/^1pc/[rr]^{}\ar@{}[rr]| {\Downarrow \Id} \ar@/_1pc/[rr]_{ } &  &  
\ar@/^1pc/[rr]^{}\ar@{}[rr]|{\Downarrow \alpha} \ar@/_1pc/[rr]_{ } &  & }  
\qquad   
\xymatrix@C=8pt{  \ar@/^1pc/[rr]^{}\ar@{}[rr]| {\Downarrow \Id} \ar@/_1pc/[rr]_{ } &  &  
\ar@/^1pc/[rr]^{}\ar@{}[rr]|{\Downarrow \Id} \ar@/_1pc/[rr]_{ } &  & } 
$$
When $V=\DblPs$, applying \prref{horiz comp 2-cells} and \prref{vertic comp 2-cells} 
one can see that the two compositions above differ by a modification given by 
the globular 2-cells $\delta_{\alpha_i, id}$ for $i=1,2$. Thus one could restrict to a full sub tricategory of $V$ whose 1-cells 
are double pseudofunctors $F$ which applied to the identity 1h- and 1v-cells give identities. Then one could also consider that their 
distinguished 2-cells $F^A$ and $F_A$ (see the next section) are identities (for all 0-cells $A$ of the domain strict 
double category of $F$), thus the unity constraints for the horizontal composition would both be identities (see \cite[Section 4.8]{Fem}), 
and one could also consider that 2-cells of the sub tricategory are those double pseudonatural tranformations $\alpha$ of $V$ 
whose associated globular 2-cells $\delta_{\alpha_i, id}$ for $i=1,2$ are identities (see the end of \deref{hor psnat tr}). 
\end{enumerate}
\end{rem}

\begin{rem} 
Let us comment the axioms (IT-1) - (IT-5). 
We do it for the case of the full sub tricategory of $V$ from point e) in the above Remark, let us denote it by $V^*$. 
Although the 3-cell $sw$ is identity in our context, we will mention it, as it helps to better understand technically 
how the compositions of 3-cells are made in the axioms. 

By $n$-fold fractions we denote vertical composition of $n$ 3-cells (observe that we consider vertical associativity of 2-cells as identity). 
All the drawings of 2-cells (bicategory diagrams), and accordingly the 3-cells acting between them, are read from top to bottom and from left to right, including the horizontal 
composition of 3-cells $\alpha\ot\beta$, (first acts $\alpha$, then $\beta$) which otherwise is read from right to left. 
In one entry of an $n$-fraction vertical lines present transversal composition of 3-cells (read from left to right). 
Moreover, in one such entry may appear: 
$\frac{\alpha}{\alpha'}\vert \beta\vert \beta'\vert \beta''$ where all the named cells are 3-cells. 
This means that instead of writing separate drawings for four transversally composed 3-cells, 
we condense them into one 3-cell written this way. We usually do this when applying the distinguished 3-cells $sw, a,\xi$ 
from the ambient tricategory $V^*$ (associativity of 2-cells and interchangers). 

(IT-1) comprises of $\lambda^*_u, \Epsilon^*, \mu_u^*, sw$ and the first 1-cell in its domain 2-cell 
is $id_{B_1}\ot u\ot u$ (in the symmetric version it is $u\ot u\ot id_{B_1}$).

(IT-2) comprises of $\lambda^*, \rho^*, sw, a, \xi$ and the first 1-cell in its domain 2-cell 
is $u\ot id_{B_1}\ot u$. 

(IT-3) comprises of $\lambda^*, \pi^*, sw, a, \xi$ and the first 1-cell in its domain 2-cell 
is $id_{B_1}\ot id_{B_1}\ot id_{B_1}\ot u$ (in the symmetric version it is $u\ot id_{B_1}\ot id_{B_1}\ot id_{B_1}$). 
It corresponds to the {\em normalization in the first and the fourth coordinate}. 

(IT-4) comprises of $\mu^*, \lambda^*, \pi^*, sw, a, \xi$ and the first 1-cell in its domain 2-cell 
is $id_{B_1}\ot id_{B_1}\ot u\ot id_{B_1}$ (in the symmetric version it is $id_{B_1}\ot u\ot id_{B_1}\ot id_{B_1}$).
It corresponds to the {\em normalization in the second and the third coordinate}. 

(IT-5) comprises of $\pi^*, sw, \xi$. It corresponds to the {\em 4-cocycle condition on $a^*$}. 

\noindent Observe in these axioms that the 3-cells $sw, a,\xi$ are the distinguished 3-cells from the ambient tricategory $V^*$. 
\end{rem}

\section{Categories internal in $\DblPs$} \selabel{int in DblPs}

An internal category in $\DblPs$ consists of strict double categories $\Dd_0$ and $\Dd_1$, strict double functors 
$S,T: \Dd_1\to\Dd_0$, double pseudo functors $U: \Dd_0\to\Dd_1, M:\Dd_1\times_{\Dd_0}\Dd_1\to\Dd_1$, double pseudonatural 
transformations $a^*, l^*, r^*$ and double modifications $\pi^*, \mu^*, \lambda^*, \rho^*, \Epsilon^*$, satisfying the corresponding 
axioms from the previous section. Both double pseudo functors $U$ and $M$ are equipped with distinguished globular 2-cells 
(set $F$ for any of $U$ and $M$): 
$$
\scalebox{0.86}{
\bfig
\putmorphism(-550,200)(1,0)[F(A)`F(C)`F(gf)]{1000}1a

 \putmorphism(-550,-200)(1,0)[F(A)`F(B)`F(f)]{500}1a
 \putmorphism(-50,-200)(1,0)[\phantom{F(A)}`F(C)`F(g)]{500}1a

\putmorphism(-580,200)(0,-1)[\phantom{Y_2}``=]{400}1l
\putmorphism(450,200)(0,-1)[\phantom{Y_2}``=]{400}1l
\put(-150,10){\fbox{$F_{gf}$}}

\efig}
\quad
\scalebox{0.86}{
\bfig
 \putmorphism(-150,200)(1,0)[F(A)`F(A)`F(id_A)]{500}1a 
\putmorphism(340,200)(0,-1)[\phantom{Y_2}``=]{400}1l

 \putmorphism(-150,-200)(1,0)[F(A)`F(A)`=]{500}1a 
\putmorphism(-180,200)(0,-1)[\phantom{Y_2}``=]{400}1r
\put(0,0){\fbox{$F_A$}}
\efig}
\quad
\scalebox{0.86}{
\bfig
 \putmorphism(-180,400)(1,0)[F(A)`F(A)`=]{500}1a
\putmorphism(-180,400)(0,-1)[\phantom{Y_2}`F(A') `F(u)]{400}1l
\putmorphism(-180,-400)(1,0)[F(A'')`F(A'') `=]{500}1a

\putmorphism(-180,0)(0,-1)[\phantom{Y_2}``F(v)]{400}1l
\putmorphism(310,400)(0,-1)[\phantom{Y_2}` `F(vu)]{800}1l
\put(-40,180){\fbox{$F^{vu}$}}
\efig}
\quad
\scalebox{0.86}{
\bfig
 \putmorphism(-150,200)(1,0)[F(A)`F(A)`=]{500}1a 
\putmorphism(340,200)(0,-1)[\phantom{Y_2}``F(id_A)]{400}1r

 \putmorphism(-150,-200)(1,0)[F(A)`F(A)`=]{500}1a 
\putmorphism(-180,200)(0,-1)[\phantom{Y_2}``=]{400}1r
\put(0,0){\fbox{$F^A$}}
\efig}
$$
satisfying the following axioms, where $f,g,h$ are composable 1h-cells, $u,v,w$ are composable 1v-cells, $a$ and $b$ are 2-cells 
composable horizontally and $a$ and $a'$ are 2-cells composable vertically (note that here, for simplicity of the notation, we are denoting 
both 1h- and 1v-composition by juxtaposition, the difference is clear from the letters denoting 1-cells). Coherence in the 1h-direction: 
$$
\scalebox{0.86}{
\bfig
\putmorphism(-550,0)(1,0)[F(A)`F(C)`F(gf)]{1000}1a
 \putmorphism(450,0)(1,0)[\phantom{F(A)}`F(D) `F(h)]{600}1a

 \putmorphism(-550,-400)(1,0)[F(A)`F(B)`F(f)]{500}1a
 \putmorphism(-50,-400)(1,0)[\phantom{F(A)}`F(C)`F(g)]{500}1a
 \putmorphism(450,-400)(1,0)[\phantom{F(A)}`F(D) `F(h)]{600}1a

\putmorphism(-580,0)(0,-1)[\phantom{Y_2}``=]{400}1l
\putmorphism(450,0)(0,-1)[\phantom{Y_2}``=]{400}1l
\putmorphism(1050,0)(0,-1)[\phantom{Y_2}``=]{400}1r
\put(-150,-190){\fbox{$F_{gf}$}}
\put(650,-190){\fbox{$\Id$}}

\putmorphism(-550,400)(1,0)[F(A)`F(D) `F(h(gf))]{1600}1a

\putmorphism(-580,400)(0,-1)[\phantom{Y_2}``=]{400}1l
\putmorphism(1050,400)(0,-1)[\phantom{Y_3}``=]{400}1r
\put(140,200){\fbox{$F_{h(gf)}$}}

\efig}
=
\scalebox{0.86}{
\bfig
\putmorphism(-550,0)(1,0)[F(A)`F(B)`F(f)]{500}1a
 \putmorphism(-50,0)(1,0)[\phantom{F(A)}`F(D) `F(hg)]{1100}1a

 \putmorphism(-550,-400)(1,0)[F(A)`F(B)`F(f)]{500}1a
 \putmorphism(-50,-400)(1,0)[\phantom{F(A)}`F(C)`F(g)]{500}1a
 \putmorphism(450,-400)(1,0)[\phantom{F(A)}`F(D) `F(h)]{600}1a

\putmorphism(-580,0)(0,-1)[\phantom{Y_2}``=]{400}1l
\putmorphism(-50,0)(0,-1)[\phantom{Y_2}``=]{400}1r
\putmorphism(1050,0)(0,-1)[\phantom{Y_2}``=]{400}1r
\put(-350,-190){\fbox{$\Id$}}
\put(350,-190){\fbox{$F_{hg}$}}

\putmorphism(-550,400)(1,0)[F(A)`F(D) `F((hg)f)]{1600}1a

\putmorphism(-580,400)(0,-1)[\phantom{Y_2}``=]{400}1l
\putmorphism(1050,400)(0,-1)[\phantom{Y_3}``=]{400}1r
\put(60,200){\fbox{$F_{(hg)f}$}}
\efig}
$$

$$
\scalebox{0.86}{
\bfig
 \putmorphism(-150,250)(1,0)[F(A)`F(A)`F(id_A)]{600}1a 
 \putmorphism(410,250)(1,0)[\phantom{A\ot B}`F(B) `F(f)]{580}1a
\putmorphism(440,250)(0,-1)[\phantom{Y_2}``=]{400}1l

 \putmorphism(-150,-150)(1,0)[F(A)`F(A)`=]{600}1a 
 \putmorphism(410,-150)(1,0)[\phantom{A\ot B}`F(B) `F(f)]{580}1a
\putmorphism(-180,250)(0,-1)[\phantom{Y_2}``=]{400}1r
\putmorphism(1000,250)(0,-1)[\phantom{Y_2}``=]{400}1r
\put(0,0){\fbox{$F_A$}}
\put(640,30){\fbox{$\Id$}}
\efig}
\quad=
\scalebox{0.86}{
\bfig
 \putmorphism(-150,250)(1,0)[F(A)`F(A)`F(id_A)]{600}1a 
 \putmorphism(410,250)(1,0)[\phantom{A\ot B}`F(B) `F(f)]{580}1a

 \putmorphism(-150,-150)(1,0)[F(A)`F(B)`F(f id_A)]{1160}1b 
\putmorphism(-180,250)(0,-1)[\phantom{Y_2}``=]{400}1r
\putmorphism(1000,250)(0,-1)[\phantom{Y_2}``=]{400}1r
\put(320,0){\fbox{$F_{f id_A}$}}
\efig}
$$

$$
\scalebox{0.86}{
\bfig
 \putmorphism(-150,250)(1,0)[F(A)`F(B)`F(f)]{600}1a 
 \putmorphism(410,250)(1,0)[\phantom{A\ot B}`F(B) `F(id_B)]{580}1a
\putmorphism(440,250)(0,-1)[\phantom{Y_2}``=]{400}1l

 \putmorphism(-150,-150)(1,0)[F(A)`F(B)`F(f)]{600}1a 
 \putmorphism(410,-150)(1,0)[\phantom{A\ot B}`F(B) `=]{580}1a
\putmorphism(-180,250)(0,-1)[\phantom{Y_2}``=]{400}1r
\putmorphism(1000,250)(0,-1)[\phantom{Y_2}``=]{400}1r
\put(640,50){\fbox{$F_B$}}
\put(60,50){\fbox{$\Id$}}
\efig}
\quad=
\scalebox{0.86}{
\bfig
 \putmorphism(-150,250)(1,0)[F(A)`F(B)`F(f)]{600}1a 
 \putmorphism(410,250)(1,0)[\phantom{A\ot B}`F(B) `F(id_B)]{580}1a

 \putmorphism(-150,-150)(1,0)[F(A)`F(B),`F(id_B f)]{1160}1b 
\putmorphism(-180,250)(0,-1)[\phantom{Y_2}``=]{400}1r
\putmorphism(1000,250)(0,-1)[\phantom{Y_2}``=]{400}1r
\put(320,0){\fbox{$F_{ id_B f}$}}
\efig}
$$
coherence in the 1v-direction: 
$$
\scalebox{0.86}{
\bfig
 \putmorphism(-150,400)(1,0)[F(A)`F(A)`=]{600}1a
 \putmorphism(450,400)(1,0)[\phantom{F(A)}`F(A) `=]{600}1a
\putmorphism(-180,400)(0,-1)[\phantom{Y_2}`F(A') `F(u)]{400}1l
\put(620,-280){\fbox{$F^{w(vu)}$}}
\putmorphism(-180,-400)(1,0)[F(A'')` `=]{500}1a

\putmorphism(-180,0)(0,-1)[\phantom{Y_2}``F(v)]{400}1l
\putmorphism(450,400)(0,-1)[\phantom{Y_2}`F(A'') `F(vu)]{800}1r
\put(40,0){\fbox{$F^{vu}$}}
\putmorphism(450,-400)(0,-1)[\phantom{Y_2}``F(w)]{400}1r
\put(40,-620){\fbox{$\Id$}}

\putmorphism(-180,-400)(0,-1)[\phantom{Y_2}``F(w)]{400}1l
\putmorphism(-210,-800)(1,0)[F(A''')` `=]{480}1a
\putmorphism(420,-800)(1,0)[F(A''')` `=]{450}1a

\putmorphism(1000,400)(0,-1)[`G(B')`F(w(vu))]{1200}1r
\efig}
=
\scalebox{0.86}{
\bfig
 \putmorphism(-150,400)(1,0)[F(A)`F(A)`=]{600}1a
 \putmorphism(450,400)(1,0)[\phantom{F(A)}`F(A) `=]{600}1a
\putmorphism(-180,400)(0,-1)[\phantom{Y_2}`F(A') `F(u)]{400}1l
\put(610,-180){\fbox{$F^{(wv)u}$}}

\putmorphism(-180,0)(0,-1)[\phantom{Y_2}`F(A'')`F(v)]{400}1l
\putmorphism(450,0)(0,-1)[\phantom{Y_2}``F(wv)]{800}1r
\putmorphism(450,400)(0,-1)[\phantom{Y_2}`F(A') `F(u)]{400}1r
\put(0,180){\fbox{$\Id$}}

\putmorphism(-180,-400)(0,-1)[\phantom{Y_2}``F(w)]{400}1l
\putmorphism(-200,-800)(1,0)[F(A''')` `=]{480}1a
\putmorphism(420,-800)(1,0)[F(A''')` `=]{450}1a

\putmorphism(-180,0)(1,0)[\phantom{(B, \tilde A)}``=]{500}1a
\putmorphism(1000,400)(0,-1)[`F(A''')`F((wv)u)]{1200}1r
\put(0,-470){\fbox{$F^{wv}$}}
\efig}
$$

$$
\scalebox{0.86}{
\bfig
\putmorphism(-280,400)(1,0)[F(A)`F(A)` =]{450}1a
 \putmorphism(-280,0)(1,0)[F(A)`F(A)` =]{450}1a
 \putmorphism(-280,-400)(1,0)[F(A')`F(A')` =]{450}1a

\putmorphism(-280,400)(0,-1)[\phantom{Y_2}``=]{400}1l
 \putmorphism(-280,20)(0,-1)[\phantom{F(A)}` `F(u)]{400}1l

\putmorphism(180,400)(0,-1)[\phantom{Y_2}``]{400}1l
\putmorphism(210,360)(0,-1)[\phantom{Y_2}``F(id_A)]{400}0l
\putmorphism(180,20)(0,-1)[\phantom{Y_2}``F(u)]{400}1l
\put(-260,220){\fbox{$F^A$}}
\put(-220,-210){\fbox{$\Id$}}
\efig}
=
\scalebox{0.86}{
\bfig
\putmorphism(-280,400)(1,0)[F(A)`F(A)` =]{450}1a
 \putmorphism(-280,-400)(1,0)[F(A')`F(A')` =]{450}1a

\putmorphism(-280,400)(0,-1)[\phantom{Y_2}``F(u id_A)]{800}1l

\putmorphism(180,400)(0,-1)[\phantom{Y_2}`F(A)`F(id_A)]{400}1l
\putmorphism(180,20)(0,-1)[\phantom{Y_2}``F(u)]{400}1l
\put(-220,0){\fbox{$F_{u id_A}$}}
\efig}
\quad\text{and}\quad
\scalebox{0.86}{
\bfig
\putmorphism(-280,400)(1,0)[F(A)`F(A)` =]{450}1a
 \putmorphism(-280,0)(1,0)[F(A')`F(A')` =]{450}1a
 \putmorphism(-280,-400)(1,0)[F(A')`F(A')` =]{450}1a

\putmorphism(-280,400)(0,-1)[\phantom{Y_2}``F(u)]{400}1l
 \putmorphism(-280,20)(0,-1)[\phantom{F(A)}` `=]{400}1l

\putmorphism(180,400)(0,-1)[\phantom{Y_2}``F(u)]{400}1l
\putmorphism(180,20)(0,-1)[\phantom{Y_2}``]{400}1l
\putmorphism(210,90)(0,-1)[\phantom{Y_2}``F(id_{A'})]{400}0l
\put(-220,200){\fbox{$\Id$}}
\put(-260,-280){\fbox{$F^{A'}$}}
\efig}
=
\scalebox{0.86}{
\bfig
\putmorphism(-280,400)(1,0)[F(A)`F(A)` =]{450}1a
 \putmorphism(-280,-400)(1,0)[F(A')`F(A'),` =]{450}1a

\putmorphism(-280,400)(0,-1)[\phantom{Y_2}``F(id_{A'} u)]{800}1l

\putmorphism(180,400)(0,-1)[\phantom{Y_2}`F(A')`F(u)]{400}1l
\putmorphism(180,20)(0,-1)[\phantom{Y_2}``F(id_{A'})]{400}1l
\put(-220,0){\fbox{$F_{id_{A'} u}$}}
\efig}
$$ 
coherence for the composition of and unity 2-cells, horizontally: 
$$
\scalebox{0.86}{
\bfig
\putmorphism(-150,500)(1,0)[F(A)`F(C)`F(gf)]{1100}1a
 \putmorphism(-150,50)(1,0)[F(A)`F(B)`F(f)]{520}1a
 \putmorphism(370,50)(1,0)[\phantom{F(A)}`F(C) `F(g)]{540}1a
\put(320,270){\fbox{$F_{gf}$}}

\putmorphism(360,50)(0,-1)[\phantom{Y_2}``F(v)]{450}1r
\putmorphism(-180,500)(0,-1)[\phantom{Y_2}``=]{450}1l
\putmorphism(950,500)(0,-1)[\phantom{Y_2}``=]{450}1r
\put(0,-170){\fbox{$F(a)$}}
\put(600,-170){\fbox{$F(b)$}}

\putmorphism(-180,-400)(1,0)[F(A')`F(B') `F(f')]{530}1a
 \putmorphism(390,-400)(1,0)[\phantom{A'}` F(C') `F(g')]{530}1a

\putmorphism(-180,50)(0,-1)[\phantom{Y_2}``F(u)]{450}1l
\putmorphism(950,50)(0,-1)[\phantom{Y_3}``F(w)]{450}1r

\efig}
=
\scalebox{0.86}{
\bfig
\putmorphism(-150,500)(1,0)[F(A)`F(C)`F(gf)]{1050}1a
\putmorphism(-150,50)(1,0)[F(A')`F(C')`F(g'f')]{1050}1a

\putmorphism(-180,500)(0,-1)[\phantom{Y_2}``F(u)]{450}1r
\putmorphism(920,500)(0,-1)[\phantom{Y_2}``F(w)]{450}1r
\put(300,260){\fbox{$F(a\vert b)$}}
\put(310,-180){\fbox{$F_{g'f'}$}}

\putmorphism(-150,-400)(1,0)[F(A')`F(B') `F(f')]{530}1a
 \putmorphism(360,-400)(1,0)[\phantom{F(A')}` F(C') `F(g')]{550}1a

\putmorphism(-180,50)(0,-1)[\phantom{Y_2}``=]{450}1l
\putmorphism(920,50)(0,-1)[\phantom{Y_3}``=]{450}1r

\efig}
$$

$$
\scalebox{0.86}{
\bfig
\putmorphism(-250,500)(1,0)[F(A)`F(A)` =]{500}1a
 \putmorphism(-250,50)(1,0)[F(A)`F(A)` F(id_A)]{500}1a
 \putmorphism(-250,-400)(1,0)[F(A')`F(A')` F(id_{A'})]{500}1a

\putmorphism(-280,500)(0,-1)[\phantom{Y_2}``=]{450}1l
 \putmorphism(-280,70)(0,-1)[\phantom{F(A)}` `F(u)]{450}1l

\putmorphism(250,500)(0,-1)[\phantom{Y_2}``=]{450}1r
\putmorphism(250,70)(0,-1)[\phantom{Y_2}``F(u)]{450}1r
\put(-120,290){\fbox{$F_A^{-1}$}}
\put(-160,-150){\fbox{$F(\Id_u)$}}
\efig}
=
\scalebox{0.86}{
\bfig
\putmorphism(-250,500)(1,0)[F(A)`F(A)` =]{500}1a
 \putmorphism(-250,50)(1,0)[F(A')`F(A')` =]{500}1a
 \putmorphism(-250,-400)(1,0)[F(A')`F(A')` F(id_{A'})]{500}1a

\putmorphism(-280,500)(0,-1)[\phantom{Y_2}``F(u)]{450}1l
 \putmorphism(-280,70)(0,-1)[\phantom{F(A)}` `=]{450}1l

\putmorphism(250,500)(0,-1)[\phantom{Y_2}``F(u)]{450}1r
\putmorphism(250,70)(0,-1)[\phantom{Y_2}``=]{450}1r
\put(-160,260){\fbox{$F(\Id_u)$}}
\put(-120,-150){\fbox{$F_{A'}^{-1}$}}
\efig}
$$ 
and vertically: 
$$
\scalebox{0.86}{
\bfig
 \putmorphism(-150,500)(1,0)[F(A)`F(A) `=]{500}1a
 \putmorphism(450,500)(1,0)[` `F(f)]{380}1a
\putmorphism(-180,500)(0,-1)[\phantom{Y_2}`F(A') `F(u)]{450}1l
\put(20,250){\fbox{$F^{vu}$}}
\putmorphism(-150,-400)(1,0)[F(A'')`F(A'') `=]{500}1a
\putmorphism(-180,50)(0,-1)[\phantom{Y_2}``F(v)]{450}1l
\putmorphism(380,500)(0,-1)[\phantom{Y_2}` `F(vu)]{900}1l
\put(560,45){\fbox{$F(\frac{a}{a'})$} }
\putmorphism(920,500)(0,-1)[F(B)`F(B'')`F(v'u')]{900}1r
\putmorphism(450,-400)(1,0)[``F(h)]{360}1a
\efig}= 
\scalebox{0.86}{
\bfig
 \putmorphism(-150,500)(1,0)[F(A)`F(B) `F(f)]{600}1a
 \putmorphism(450,500)(1,0)[\phantom{(B,A)}` `=]{450}1a
\putmorphism(-150,500)(0,-1)[\phantom{Y_2}` `F(u)]{450}1l
\put(620,50){\fbox{$F^{v'u'}$}}
\putmorphism(-150,-400)(1,0)[F(A'')` `F(h)]{500}1a
\putmorphism(-150,50)(0,-1)[\phantom{Y_2}``F(v)]{450}1l
\putmorphism(450,500)(0,-1)[\phantom{Y_2}`F(B') `F(u')]{450}1r
\putmorphism(450,50)(0,-1)[\phantom{Y_2}`F(B'')`F(v')]{450}1r
\put(40,260){\fbox{$F(a)$}}
\putmorphism(-150,50)(1,0)[F(A')``F(g)]{500}1a
\putmorphism(1000,500)(0,-1)[F(B)`F(B'')`F(v'u')]{900}1r
\putmorphism(480,-400)(1,0)[\phantom{F(A)}`\phantom{F(A)}`=]{500}1b
\put(40,-170){\fbox{$F(a')$}}
\efig}
$$

$$
\scalebox{0.86}{
\bfig
 \putmorphism(-150,250)(1,0)[F(A)`F(A)`=]{600}1a 
 \putmorphism(450,250)(1,0)[\phantom{F(A)}`F(B) `F(f)]{680}1a
\putmorphism(500,250)(0,-1)[\phantom{Y_2}``F(id_A)]{450}1l

 \putmorphism(-150,-200)(1,0)[F(A)`F(A)`=]{600}1a 
 \putmorphism(450,-200)(1,0)[\phantom{F(A)}`F(B) `F(f)]{680}1a
\putmorphism(-180,250)(0,-1)[\phantom{Y_2}``=]{450}1r
\putmorphism(1100,250)(0,-1)[\phantom{Y_2}``F(id_B)]{450}1r
\put(0,0){\fbox{$F^A$}}
\put(650,30){\fbox{$F(\Id_f)$}}
\efig}
=
\scalebox{0.86}{
\bfig
 \putmorphism(-100,250)(1,0)[F(A)`F(B)`F(f)]{550}1a
 \putmorphism(450,250)(1,0)[\phantom{F( B)}`F(B) `=]{550}1a

 \putmorphism(-100,-200)(1,0)[F(A)`F(B)`F(f)]{550}1a
 \putmorphism(450,-200)(1,0)[\phantom{F( B)}`F(B). `=]{550}1a 
\putmorphism(-120,250)(0,-1)[\phantom{Y_2}``=]{450}1r
\putmorphism(450,250)(0,-1)[\phantom{Y_2}``=]{450}1l
\putmorphism(1000,250)(0,-1)[\phantom{Y_2}``F(id_B)]{450}1r
\put(40,40){\fbox{$\Id_{F(f)}$}}
\put(650,30){\fbox{$F^B$}}
\efig}
$$

The above three coherences in the 1v-direction for $U$ and $M$ correspond to axioms (21)-(26) of \cite[Section 3]{GP}, respectively. 
The analogous six coherences in the 1h-direction do not appear there. The two (horizontally) globular 2-cells $F^{vu}$ and 
$F^A$ for $U$ and $M$ correspond to natural transformations (17)-(20): $U^A=\tau, U^{vu}=\mu, M^A=\delta, M^{vu}=\chi$, 
and the above two coherences for the composition of and unity 2-cells 
in the vertical direction for $U$ and $M$ correspond to naturalities of (17)-(20). One can analogously formulate 
natural transformations in the horizontal direction, introducing additional two (vertically) globular 2-cells $F_{gf}$ and 
$F_A$ for $U$ and $M$ and the above two coherences for the composition of and unity 2-cells in the horizontal direction, which 
correspond to their naturalities. (To formulate these natural transformations in the horizontal direction change the roles of 
vertical and horizontal cells in the definition of two categories determining a strict double category.) For the sake of 
comparing this structure to intercategories, for mnemotechnical reasons we could denote 
these distinguished (vertically) globular 2-cells as follows: $U_A=\tau', U_{gf}=\mu', M_A=\delta', M_{gf}=\chi'$.

\medskip

Summing up, for the double pseudo functors $U$ and $M$ we have eight globular 2-cells: 
$$U_{gf}, U_A, U^{vu}, U^A, \quad M_{gf}, M_A, M^{vu}, M^A,$$
which satisfy in total 20 axioms named above. We will denote their actions as follows. 
Let us denote the image under $M:\Dd_1\times_{\Dd_0}\Dd_1\to\Dd_1$ of $(y,x)$ by $(x\vert y)$ for any of the four types of cells 
$(y,x)\in\Dd_1\times_{\Dd_0}\Dd_1$. Moreover, let us denote by $\Id^h_x$ the image under $U:\Dd_0\to\Dd_1$ of any 
of the four types of cells $x$ in $\Dd_0$. Now for 1h-cells $g, g',f,f'$ and 1v-cells $u, u', v, v'$ of $\Dd_1$ (for the action 
of $M$), respectively of $\Dd_0$ (for the action of $U$) we will write: 
\begin{equation} \eqlabel{globular on vertical}
\chi:\frac{(u\vert u')}{(v\vert v')}\Rightarrow (\frac{u}{v}\vert \frac{u'}{v'}), \quad 
\delta: id^v_{(A\vert A')}\Rightarrow(id^v_A\vert id^v_{A'}), \quad 
\mu: \frac{\Id^h_u}{\Id^h_v}\Rightarrow\Id^h_{\frac{u}{v}}, \quad 
\tau: id^v_{\Id^h_A}\Rightarrow \Id^h_{id^v_A}
\end{equation} 

\begin{equation} \eqlabel{globular on horiz}
\begin{cases}
(gf\vert g'f')\\
\quad \Downarrow \chi' \\
(g\vert g')(f\vert f')
\end{cases}  
	\qquad
\begin{cases}
(id^h_A\vert id^h_{A'})\\
\quad \Downarrow \delta' \\
id^h_{(A\vert A')}
\end{cases}  
	\qquad
\begin{cases}
\Id^h_{gf}\\
 \Downarrow \mu' \\
\Id^h_g\Id^h_f
\end{cases}  
	\qquad
\begin{cases}
\Id^h_{id^h_A}\\
 \Downarrow \tau' \\
id^h_{\Id^h_A}
\end{cases}  
\end{equation} 
here $id^v_A$ denotes the identity 1v-cell on $A$ (observe that the composition in the juxtapositions is read from right to left, 
while in $(-\vert-)$ it is done the other way around!).

\medskip 

A double pseudonatural transformation $\alpha: F\Rightarrow G$ between double pseudo functors $F$ and $G$ consists of 
a vertical pseudonatural transformation $\alpha_0: F\Rightarrow G$ and a horizontal pseudonatural transformation 
$\alpha_1: F\Rightarrow G$, both of which by \deref{hor psnat tr} are given by two distinguished globular 2-cells 
$\delta_{\alpha_0,u}$ and $\delta_{\alpha_1,f}$ and satisfy 5 axioms (two of them are trivial and one is simplified in the context of intercategories), 
two distinguished 2-cells $t^\alpha_f$ and $r^\alpha_u$ for every 1v-cell $u$ and 1h-cell $f$, which have to satisfy 6 axioms in total, 
by \deref{double 2-cells}. 
Comparing such a structure of a double pseudonatural transformation with the context of intercategories, that is, 
comparing 2-cells of the tricategory $\DblPs$ and the 2-category $LxDbl$, one finds that in the latter context only $\alpha_1$ appears 
(with $\delta_{\alpha_1,f}$ trivial), being the resting data $\alpha_0$, four 2-cells and 6 axioms new in our context. 

Thus each of double pseudonatural transformations 
$a^*: M(\Id\times_{\Dd_0} M) \to M(M\times_{\Dd_0} \Id), 
l^*: M(U\times_{\Dd_0}-)\to -$ and $r^*: M(-\times_{\Dd_0}U)\to -$ is equipped with 6 distinguished 2-cells for every 1v-cell $u$ and 1h-cell $f$ 
and satisfies 16 axioms. This makes 18 distinguished 2-cells and 48 axioms. As commented in \ssref{2-cells Theta}, if double 
pseudonatural transformations come from $\Theta$-double pseudonatural transformations (the 2-cells $t^\alpha_f$ and $r^\alpha_u$ come from a 
2-cell $\Theta^\alpha_A$), as indicated in \prref{teta->double}, then two axioms become trivially fulfilled for each 
double pseudonatural transformation, reducing the amount of axioms to 42. The 6 conditions (27)-(32) from \cite[Section 3]{GP} for 
horizontal transformations, corresponding to our $a^*, l^*, r^*$, together with the corresponding three naturality conditions, so 9 in total, 
are substituted by 42 or 48 axioms in our context.

\medskip

Instead of writing out all the axioms for all of the transformations here, let us just record the following. 
For the double pseudonatural transformation $a^*: M(\Id\times_{\Dd_0} M) \to M(M\times_{\Dd_0} \Id)$, which we can also write as 
$a^*: ((-\vert -)\vert-) \Rightarrow (-\vert(-\vert -))$, let us shorten: 
$L=(-\vert -)\vert -=((-\vert -)\vert-)$ and $R=-\vert (-\vert -)=(-\vert(-\vert -))$. 
The 1v- and 1h-composition in $\Dd_1$ we will denote by fractions and juxtapositions: $\frac{u}{v}$ and $gf$, respectively. 
Then the distinguished globular 2-cells for the double pseudonatural transformations $L$ and $R$ are given by:  
$$L^{vu}=\big(\frac{(u\vert u')\vert u''}{(v\vert v')\vert v''}\stackrel{\chi_{\bullet\bullet,\bullet}}{\Rightarrow}
\frac{(u\vert u')}{(v\vert v')}\vert \frac{u''}{v''}\stackrel{\chi\vert 1}{\Rightarrow}
\big(\frac{u}{v}\vert \frac{u'}{v'} \big)\vert \frac{u''}{v''} \big), \quad R^{vu}=\big(\frac{u\vert(u'\vert u'')}{v\vert(v'\vert v'')}\stackrel{\chi_{\bullet,\bullet\bullet}}{\Rightarrow}
\frac{u}{v}\vert \frac{(u'\vert u'')}{(v'\vert v'')} \stackrel{1\vert\chi}{\Rightarrow}
\frac{u}{v}\vert \big(\frac{u'}{v'} \vert \frac{u''}{v''}\big) \big)$$ 
$$L^A= \big( \Id^v_{(A\vert A')\vert A''} \stackrel{\delta_{\bullet\bullet,\bullet}}{\Rightarrow}
[\Id^v_{(A\vert A')}\vert\Id^v_{A''}] \stackrel{[\delta\vert 1]}{\Rightarrow}
[\Id^v_A\vert\Id^v_{A'}]\Id^v_{A''}\big) $$
$$R^A= \big( \Id^v_{A\vert (A'\vert A'')} \stackrel{\delta_{\bullet,\bullet\bullet}}{\Rightarrow}
[\Id^v_A\vert\Id^v_{(A'\vert A'')}] \stackrel{[1\vert\delta]}{\Rightarrow}
\Id^v_A[\Id^v_{A'}\vert\Id^v_{A''}]\big)$$ 

 $$L_{gf} =
    \begin{cases}
      \big((f\vert g)\vert(f'\vert g')\big)\vert(f''\vert g'') \\
       \qquad\quad \Downarrow \chi'\vert\Id \\
       (f'f\vert g'g)\vert(f''\vert g'') \\
			 \qquad\quad \Downarrow \chi'\\
			 f''(f'f)\vert g''(g'g)
    \end{cases}  
\quad
 R_{gf} =
    \begin{cases}
      (f\vert g)\vert\big((f'\vert g')\vert(f''\vert g'')\big) \\
       \qquad\quad \Downarrow \Id\vert\chi'\\
       (f\vert g)\vert(f''f'\vert g''g') \\
			 \qquad\quad \Downarrow \chi' \\
			 f''(f'f)\vert g''(g'g)
    \end{cases}$$  

\medskip 

 $$L_A =
    \begin{cases}
      (id^h_A\vert id^h_{A'})\vert id^h_{A''} \\
       \qquad \Downarrow \delta'\vert\Id\\
       id^h_{A\vert A'}\vert id^h_{A''} \\
			 \quad \Downarrow \delta' \\
			 id^h_{(A\vert A')\vert A''}
    \end{cases}  
	\qquad
		R_A =
    \begin{cases}
      id^h_A\vert(id^h_{A'}\vert id^h_{A''}) \\
       \qquad \Downarrow \Id\vert\delta'\\
       id^h_A\vert id^h_{A'\vert A''} \\
			 \quad \Downarrow \delta' \\
			 id^h_{A\vert(A'\vert A'')}.
    \end{cases}$$
\noindent In \cite[Section 4.2]{Fem1} we wrote out a half of the axioms for the double pseudonatural transformation $a^*$.

\subsection{(Pseudo)monoid in B\"ohm's $(Dbl,\ot)$ as a category internal in $\DblPs$}

From our discussion from the end of \ssref{beyond} we see that in order to view a monoid $\Aa$ in $(Dbl,\ot)$ as a category internal in 
$\DblPs$, the double pseudo functor $\oast:\Aa\times\Aa\to\Aa$ is a good candidate for a desired composition on the pullback 
($M:\Dd_1\times_{\Dd_0}\Dd_1\to\Dd_1$, with $\Dd_1=\Aa$ and $\Dd_0=1$). 

Recall that $m: \Aa\ot\Aa\to\Aa$ is a strict double functor on the Gray type monoidal product on $(Dbl,\ot)$, while 
$\oast:\Aa\times\Aa\to\Aa$ is a double pseudo functor on the Cartesian product of double categories. 
Let us set $f\oast g=m\big((1\ot g)(f\ot 1)\big)$ (recall the discussion from \ssref{Gabi's monoid}). Since 
$m$ is strictly multiplicative in both directions, we find $m\big((1\ot g)(f\ot 1)\big)=m(1\ot g)m(f\ot 1)$, which yields 
$(1\oast g)(f\oast 1)=f\oast g$ (taking $h'=1,k=1$ in the computation in \ssref{Gabi's monoid} we recover the same identity). 

Now direct computation shows: $h\oast(g\oast f)=(h\oast g)\oast f$ in both vertical and horizontal direction of 1-cells: 
use the distributive law of the tensor with respect to the composition of 1-cells 
in the Gray type tensor product $\Aa\ot\Aa$ (see the description of this tensor product after \deref{H dbl}), 
the fact that associativity of the latter compositions is strict and that $m$ is strictly associative \cite[(iii) of Section 4.3]{Gabi}).   
This yields an analogous result on 
0- and double cells, then for the double pseudonatural transformation $a^*: \oast(\Id\times\oast) \to \oast(\oast\times\Id)$
we may set to be identity: $(a^*_0)_{C,B,A}=id^v_{(A\vert B)\vert C}$ and $(a^*_1)_{C,B,A}=id^h_{(A\vert B)\vert C}$, 
$(a^*_0)_{f'',f',f}=\Id_{(f\vert f')\vert f''}=t^{a^*}_f=t^{a^*}_{f'',f',f}$ and $(a^*_1)_{u'',u',u}=\Id_{(u\vert u')\vert u''}=
r^{a^*}_u=r^{a^*}_{u'',u',u}$, and $L^A=L^{A'',A',A}=1_{(A\vert A')\vert A''}$, the same for $R^A$, 
here $C,B,A$ are 0-cells, $u'',u',u$ 1v-cells, and $f'',f',f$ 1h-cells of $\Aa$. 
Observe that 
it is $M=\oast, M(y,x)=y\oast x= (x\vert y)$. 

Let $I$ denote the image 0-cell of the strict double functor $u:*\to\Aa$. Observe that: $m(A,I)=\oast(A,I)=A\oast I$ and similarly 
the other way around, for any 0-cell $A\in\Aa$. Now by \cite[(iii) of Section 4.3]{Gabi} we deduce that left and right unity constraints 
$l^*$ and $r^*$ for $\oast:\Aa\times\Aa\to\Aa$ are identities. As a matter of fact, as a monoid in a 1-category it can not have 
2- and 3-cells for the constraints, so we have that a monoid $\Aa$ in $(Dbl,\ot)$ is not only a category internal in $\DblPs$, 
but even a category internal in the underlying 1-category of $\DblPs$, which is the category 
from \cite[Section 6]{Shul1}. 

\medskip

Let us consider a monoidal 2-category made out of the monoidal category $(Dbl,\ot)$ from \cite{Gabi} by adding as 2-cells 
vertical transformations, whose 1v-cell components have 1h-companions (recall \ssref{bijective corr}). 
We denote this 2-category by $(Dbl_2,\ot)$. Let us now consider peudomonoids in this 2-category. We repeat the analogous arguments 
as in the above computations. The difference appears when computing associativity on the 1-cells: now $m$ is not strictly associative, rather there 
is an isomorphism $a^m_0: \ot(id\times\ot)\to\ot(\ot\times id)$. We have to take into account the form of (horizontal and vertical) 1-cells in 
$\Aa\ot\Aa\ot\Aa$, we find: $h\oast(g\oast f)=\big(h\oast(1\oast 1)\big)[\big(1\oast (g\oast 1)\big)\cdot(1\oast (1\oast f)\big)]$ and 
$(h\oast g)\oast f=[\big((h\oast 1)\oast 1\big)\cdot\big((1\oast g)\oast 1\big)]\big((1\oast 1)\oast f\big)$, where the square brackets 
may be omitted, and the dot denotes the composition of 1-cells (in the corresponding direction). 
Then we define the 2-cell $(a^*_0)_{h,g,f}$ as the following 2-cell: 
\begin{equation} \eqlabel{a^*_0}
(a^*_0)_{h,g,f}=
\bfig
 \putmorphism(-180,200)(1,0)[``(f\vert 1)\vert 1]{600}1a
 \putmorphism(420,200)(1,0)[` `(1\vert g)\vert 1]{600}1a
 \putmorphism(1010,200)(1,0)[` `(1\vert 1)\vert h]{600}1a
\put(-120,30){\fbox{$(a^m_0)_{f,1,1}$}}
\putmorphism(-180,220)(0,-1)[\phantom{Y_2}``(a^m_0)_{A,B,C}]{440}1l
\put(600,60){\fbox{$(a^m_0)_{1,g,1}$}}
\put(1200,60){\fbox{$(a^m_0)_{1,1,h}$}}

\putmorphism(420,220)(0,-1)[\phantom{Y_2}``]{440}1r
\putmorphism(240,170)(0,-1)[\phantom{Y_2}``(a^m_0)_{A',B,C}]{400}0r
\putmorphism(1010,220)(0,-1)[\phantom{Y_2}``]{440}1r
\putmorphism(990,160)(0,-1)[\phantom{Y_2}``(a^m_0)_{A',B',C}]{400}0r
\putmorphism(1600,220)(0,-1)[\phantom{Y_2}``(a^m_0)_{A',B',C'}]{440}1r

  \putmorphism(-180,-200)(1,0)[` `f\vert( 1\vert 1)]{600}1a
\putmorphism(420,-200)(1,0)[`` 1\vert( g\vert 1)]{600}1a
\putmorphism(1010,-200)(1,0)[`` 1\vert( 1\vert h)]{600}1a
\efig
\end{equation}
so that on 0-cells we have: $(a^*_0)_{A,B,C}=(a^m_0)_{A,B,C}$. (On the right hand-side of the identity \equref{a^*_0} the indexes are read 
from the left to the right, to accompany the notation of the 1h-cells used here.) In \ssref{embed transf} 
we proved for what here are 2-cells of $(Dbl_2,\ot)$ that they can be turned into 
2-cells in the tricategory $\DblPs$. Let $a^m_1$ denote the obtained (strong) horizontal transformation, and $t^m_{h,g,f}$ and 
$r^m_{w,v,u}$ the obtained distinguished 2-cells making $a^m=(a^m_0, a^m_1, t^m, r^m)$ a double pseudonatural transformation. 
We define the 2-cell $(a^*_1)_{w,v,u}$, for 1v-cells $u,v,w$, in the analogous way as we did for $(a^*_0)_{h,g,f}$ above. 
The 2-cells $t^m, r^m$ are constructed due to \prref{teta->double} as follows: 
\begin{equation} \eqlabel{t^m,r^m}
t^m_{\tilde f}=
\scalebox{0.86}{\bfig
\putmorphism(-150,250)(1,0)[F(\tilde A)`F(\tilde B)`F(\tilde f)]{600}1a
 \putmorphism(450,250)(1,0)[\phantom{F(A)}`G(\tilde B) `a^m_1(\tilde B)]{600}1a

 \putmorphism(-150,-200)(1,0)[G(\tilde A)` G(\tilde B)` G(\tilde f)]{600}1a
 \putmorphism(450,-200)(1,0)[\phantom{F(A)}`G(\tilde B) ` =]{600}1a

\putmorphism(-180,250)(0,-1)[\phantom{Y_2}``a^m_0(\tilde A)]{450}1l
\putmorphism(450,250)(0,-1)[\phantom{Y_2}``]{450}1r
\putmorphism(300,250)(0,-1)[\phantom{Y_2}``a^m_0(\tilde B)]{450}0r
\putmorphism(1050,250)(0,-1)[\phantom{Y_2}``=]{450}1r
\put(0,10){\fbox{$(a^m_0)_{\tilde f}$}}
\put(650,20){\fbox{$\Epsilon^m_{\tilde B}$}}
\efig}
\quad\text{and}\quad
r^m_{\tilde u}=
\scalebox{0.86}{\bfig
 \putmorphism(-150,500)(1,0)[F(\tilde A)`G(\tilde A)  `a^m_1(\tilde A)]{600}1a
\putmorphism(-180,500)(0,-1)[\phantom{Y_2}`F(\tilde{A'}) `F(\tilde u)]{450}1l
\putmorphism(-150,-400)(1,0)[G(\tilde{A'})` `=]{480}1a
\putmorphism(-180,50)(0,-1)[\phantom{Y_2}``a^m_0(\tilde{A'})]{450}1l
\putmorphism(450,50)(0,-1)[\phantom{Y_2}`G(\tilde{A'})`=]{450}1r
\putmorphism(450,500)(0,-1)[\phantom{Y_2}`G(\tilde{A'}) `G(\tilde u)]{450}1r
\put(0,260){\fbox{$(a^m_1)_{\tilde u}$}}
\putmorphism(-180,50)(1,0)[\phantom{F(A)}``a^m_1(\tilde{A'})]{500}1a
\put(20,-170){\fbox{$\Epsilon^m_{\tilde{A'}}$}}
\efig}
\end{equation}
where $F=\ot(id\times\ot)$ and $G=\ot(\ot\times id)$, $\tilde f$ and $\tilde u$ are 1h- and 1v-cell in $\Aa\times\Aa\times\Aa$, respectively, 
and $\Epsilon^m_A$ is the 2-cell from the data that $a^m_0(A)$ is a companion of $a^m_1(A)$. 
We construct $t^*$ and $r^*$ by the same recipe: substitute $(a^m_0)_{\tilde f}$ from \equref{t^m,r^m} by $(a^*_0)_{h,g,f}$ from \equref{a^*_0}, and 
set $\Epsilon^*_{C',B',A'}=\Epsilon^m_{C',B',A'}$ to define $t^*_{h,g,f}$, analogously for $r^*_{w,v,u}$. Then $a^*=(a^*_0, a^*_1, t^*, r^*)$ 
constitutes a 2-cell in $\DblPs$. 

For the unity constraints $l^*, r^*$ the argument is simpler. Since $A\oast I$ is an image both by $m:\Aa\ot\Aa\to\Aa$ and by $\oast:\Aa\times\Aa\to\Aa$, 
as we argued above, 
we just set $l^*=l^m$ and $r^*=r^m$, being the right hand-sides unity constraints for $m$. Analogously as above, these vertical transformations 
can be made into double pseudonatural transformations, hence  $l^*$ and $r^*$ are indeed 2-cells in $\DblPs$. 

For the 3-cells in \deref{int3} we take to be identities and get that a pseudomonoid in $(Dbl_2,\ot)$ is indeed a category internal 
in $\DblPs$. 

\medskip

In order to have an example with non-trivial 3-cells from \deref{int3}, one can take a ``weak pseudomonoid'' in the tricategory $(Dbl_3,\ot)$, which 
is obtained from the 2-category $(Dbl_2,\ot)$ by adding invertible vertical modifications as 3-cells, {\em i.e.} invertible modifications of vertical transformations. 


Let us now prove that invertible vertical modifications give rise to invertible horizontal modifications, so that together they make (invertible) 
3-cells in the tricategory $\DblPs$. Then the 3-cells constraints for $m$, which are $\pi^m, \mu^m, \lambda^m, \rho^m$, can be upgraded to 3-cells 
$\pi^*, \mu^*, \lambda^*, \rho^*$ corresponding to the desired 3-cells in \deref{int3}, and we would have this desired example. 

Recall that vertical modifications are given by 2-cells $b_0(A)$ as on the left hand-side below, 
then let the inverses of horizontal modifications be given via the 2-cells $b_1^{-1}(A)$ on the right hand-side below: 
$$
\scalebox{0.86}{
\bfig
\putmorphism(-100,250)(1,0)[F(A)`F(A)`=]{550}1a
 \putmorphism(-100,-200)(1,0)[G(A)`G(A)`=]{550}1a

\putmorphism(-100,250)(0,-1)[\phantom{Y_2}``\alpha_0(A)]{450}1l
\putmorphism(450,250)(0,-1)[\phantom{Y_2}``]{450}1r
\putmorphism(300,250)(0,-1)[\phantom{Y_2}``\beta_0(A)]{450}0r
\put(0,10){\fbox{$b_0(A)$}}
\efig}
\qquad\quad\qquad
b_1^{-1}(A)=\scalebox{0.86}{
\bfig
 \putmorphism(-150,250)(1,0)[F(A)`F(A) `=]{500}1a
\put(-100,30){\fbox{$\eta^\alpha(A)$}}
\putmorphism(380,250)(0,-1)[\phantom{Y_2}` `]{450}1l
\putmorphism(520,250)(0,-1)[\phantom{Y_2}` `\alpha_0(A)]{450}0l
\putmorphism(950,250)(0,-1)[\phantom{Y_2}` `\beta_0(A)]{450}1r
\putmorphism(350,250)(1,0)[F(A)`F(A)`=]{600}1a
 \putmorphism(950,250)(1,0)[\phantom{F(A)}`G(A) `\beta_1(A)]{600}1a
 \putmorphism(470,-200)(1,0)[`G(A)`=]{500}1b
 \putmorphism(1060,-200)(1,0)[`G(A)`=]{500}1b
\putmorphism(1570,250)(0,-1)[\phantom{Y_2}``=]{450}1r
\put(530,10){\fbox{$b_0(A)$}}
\put(1250,30){\fbox{$\Epsilon^\beta_A$}}
\putmorphism(-150,-200)(1,0)[F(A)`G(A) `\alpha_1(A)]{520}1a
\putmorphism(-150,250)(0,-1)[``=]{450}1l
\efig}
$$
(in the obvious way $b_1(A)$ is given via $b^{-1}_0(A)$; recall that $\eta$ and $\Epsilon$ come from the data of companions). 
It is straightforward to prove that this defines horizontal modifications (one uses $\Epsilon\x\eta$-properties and the construction 
of a horizontal transformation out of a vertical one from \prref{vertic->horiz}; recall that 
for vertical transformations $\alpha_0$ the distinguished 2-cells $\delta_{\alpha_0,u}$ are identities, for 1v-cells $u$). 
Finally, the two compatibility conditions between a horizontal and a vertical modification from \deref{modif 3} are 
directly proved. In the second condition one uses the third identity in \coref{4 identities e-eta} which is fulfilled in this context. 
This finishes the proof that a ``weak pseudomonoid'' in the tricategory $(Dbl_3,\ot)$ is a category internal in the tricategory $\DblPs$. 

\medskip

The examples of intercategories from \cite{GP:Fram} which do not rely on laxness of the double functors in $LxDbl$, as duoidal categories do, 
are all examples of categories internal in the tricategory $\DblPs$ (so that 3-cells for the internal structure are trivial). These are {\em e.g.} 
monoidal double categories of \cite{Shul}, cubical bicategories of \cite{GG}, Verity double bicategories from \cite{Ver}, Gray categories \cite{Gray}.

\subsection{Geometric interpretation of a category internal in $\DblPs$}

Let us denote this structure formally by 
$$
 \Dd_1\times_{\Dd_0}\Dd_1\triarrows \Dd_1\tripplearrow \Dd_0 
$$
and the functor components of the double pseudo functors $U: \Dd_0\to\Dd_1$ and $M:\Dd_1\times_{\Dd_0}\Dd_1\to\Dd_1$ by 
$U_i, M_i, i=0,1$. Then we may obtain a similar grid of categories and functors to $(*)$ from \cite[Section 4]{GP}, 
the difference is that now the three columns in the grid are {\em strict} double categories and the rows differ in that not 
only the functors $U_1$ and $M_1$, but now also $U_0$ and $M_0$, are equipped with natural transformations expressing their 
lax multiplicativity and lax unitality.

Let us see a geometrical representation of this alternative notion to intercategories on a cube. 
Considering source and target functors, as well as arrows from morphisms to objects in the categories $(\Dd_0)_0, (\Dd_0)_1, (\Dd_1)_0, (\Dd_1)_1$ 
constituting double categories $\Dd_0$ and $\Dd_1$, 
one sees that the objects of $(\Dd_0)_0$ are the lowest and morphisms of $(\Dd_1)_1$ are the highest in this hierarchy, so we may present 
the former by vertices of a cube and the latter by the whole cube. For the rest of gadgets there is a choice, we will 
fix the one as in \cite[Section 4]{GP}, so that we have: \\


vertices - objects of $\Dd_0$ \qquad

horizontal arrows - objects of $\Dd_1$, 

vertical arrows - 1v-cells of $\Dd_0$, 

transversal arrows - 1h-cells of $\Dd_0$,

horizontal cells - 1h-cells of $\Dd_1$, 

lateral cells - 2-cells of $\Dd_0$, 

basic cells - 1v-cells of $\Dd_1$ and 

cube - 2-cells of $\Dd_1$. \vspace{-3,8cm} 

$ \hspace{6cm}
\bfig
 \putmorphism(-20,520)(2,1)[``]{380}1a 
  \putmorphism(540,800)(1,0)[\bullet`\bullet`]{1000}1a 
 \putmorphism(1000,520)(2,1)[``]{350}1a 
\put(400,560){\fbox{$1h \hspace{0,14cm} of\hspace{0,14cm} \Dd_1$}}
 \putmorphism(-120,1050)(1,0)[``vertices: objects \hspace{0,14cm} of\hspace{0,14cm} \Dd_0]{1000}0b 

 \putmorphism(-150,400)(1,0)[\bullet`\bullet`objects \hspace{0,14cm} of\hspace{0,14cm} \Dd_1]{1000}1b 
\putmorphism(840,400)(0,-1)[\phantom{Y_2}`\bullet`]{800}1l
\putmorphism(1100,400)(0,-1)[\phantom{Y_2}``1v \hspace{0,14cm} of\hspace{0,14cm} \Dd_0]{800}0l

 \putmorphism(-170,-400)(1,0)[\bullet`\bullet`]{1020}1a 
\putmorphism(-180,400)(0,-1)[\phantom{Y_2}``]{800}1r
\put(0,0){\fbox{$1v \hspace{0,14cm} of\hspace{0,14cm} \Dd_1$}}

 \putmorphism(-170,-640)(1,0)[``cube: 2\x cells \hspace{0,14cm} of\hspace{0,14cm} \Dd_1]{1020}0a 
\putmorphism(1530,810)(0,-1)[\phantom{Y_2}`\bullet`]{840}1l
 \putmorphism(1030,-300)(2,1)[``]{320}1a 
 \putmorphism(830,-350)(1,0)[``1h \hspace{0,14cm} of\hspace{0,14cm} \Dd_0]{820}0a 
\put(880,200){\fbox{$2\x cells \hspace{0,14cm} of\hspace{0,14cm} \Dd_0$}}
\efig
$

\noindent From here we see that vertical and transversal arrows compose in their respective directions,  
horizontal cells compose in the transversal direction, basic cells compose in the vertical direction, and 
lateral cells both in vertical and transversal directions. All of them compose strictly associatively and unitary. 
The pullback $\Dd_1\times_{\Dd_0}\Dd_1$ can be represented by horizontal connecting of cubes, and accordingly the 
functor $M:\Dd_1\times_{\Dd_0}\Dd_1\to\Dd_1$ corresponds to the horizontal composition of cubes. 

The globular 2-cells \equref{globular on vertical} of $\Dd_1$ are thus cubes whose only non-identity cells are the basic ones, 
and we will consider that they map from the back towards the front. They compose in the transversal direction. 
On the other hand, the globular 2-cells \equref{globular on horiz} of $\Dd_1$ are cubes whose only non-identity cells 
are the horizontal ones, they map from top to bottom, and compose in the vertical direction. 

The double pseudofunctor $U$ applied to a 2-cell $a$ of $\Dd_0$ gives a cube $\Id^h_a$ which is horizontal 
identity cube on the lateral cell $a$, and the rest of the cells are identities on the corresponding 1h- and 
1v-cells at the borders of $a$. 

A 2-cell in $\Dd_1$ is a cube whose lateral cells are identities, top and bottom correspond to its source and target 1h-cells, 
while front and back basic cells correspond to its source and target 1v-cells. 

For all the laws described in \seref{int in DblPs} observe that horizontal composition of 2-cells in $\Dd_1$ corresponds to the 
transversal composition of cubes, and that vertical composition of 2-cells in $\Dd_1$ corresponds to the vertical composition of 
cubes.

\section{Enriched categories as internal categories in 1-strict tricategories} \selabel{enrich-int}

In the first subsection of this section we introduce categories enriched over 1-strict tricategories. In the second 
subsection we prove the result from the above title and in the third one we discuss examples 
in lower dimensions that can be seen as its consequences. 
The next section we dedicate to illustrate this result on the tricategory of tensor categories.

\subsection{Categories enriched over 1-strict tricategories} 

For enrichment we need some kind of a monoidal product in the ambient tricategory. We will consider tricategorical products from 
\ssref{3-(co)prod} (with the terminal object). By a terminal object in a tricategory we mean a 0-cell $I$ so that for any 0-cell 
$T$ there is a unique 1-cell $t:T\to I$, and all the 2-cells $t\Rightarrow t$ are the identity one. In the following definition 
for the terminal object in the ambient tricategory $V$ we will write just $*$.

\begin{defn} \delabel{enr3}
Let $V$ be a 1-strict tricategory with 3-products. We say that $\Tau$ is a category enriched over $V$ if it consists of: 
\begin{enumerate}
\item a set of objects $\Ob\Tau$ of $\Tau$;
\item for all $A,B\in\Ob\Tau$ a 0-cell $\Tau(A,B)$ in $V$; 
\item for all $A,B,C\in\Ob\Tau$ a 1-cell $\circ:\Tau(B,C)\times\Tau(A,B)\to\Tau(A,C)$ in $V$ called {\em composition};
\item for all $A\in\Ob\Tau$ a 1-cell $I_A:*\to\Tau(A,A)$ in $V$ called {\em unit}; 
\item equivalence 2-cells in $V$:  
$a^\dagger: -\circ(-\circ-) \to (-\circ-)\circ-$, and for all 
$A,B\in\Ob\Tau: \hspace{0,24cm} l^\dagger: I_B\cdot 1_{\Tau(A,B)}\to 1_{\Tau(A,B)}$ and   
$r^\dagger: 1_{\Tau(A,B)}\cdot I_A\to 1_{\Tau(A,B)}$; 
\item 3-cells 
$\pi^\dagger, \mu^\dagger, l^\dagger, r^\dagger$ and $\Epsilon^\dagger$ analogous to those in item 4. of \deref{int3} 
and which satisfy the analogous Axioms as the latter ones. 
\end{enumerate}
\end{defn}

The formal differences in the cells and Axioms in \deref{int3} and the above one are the following. In the vertices 
of the diagrams the iterated 3-pullbacks $B_1^{(n)_0}$ are replaced by $\Tau(\bullet,\bullet)^{\times n}$ for natural numbers $n$, 
1-cells $c$ and $u$ are replaced by $\comp$ and $I_\bullet$, respectively, and supraindeces * are replaced by supraindeces $\dagger$.

\medskip

\begin{lma}
There exist equivalence 1-cells in $V$ between the following 3-coproducts:
$$\amalg_{A\in\Ob\Tau}\amalg_{B\in\Ob\Tau}\Tau(A,B) \simeq \amalg_{B\in\Ob\Tau}\amalg_{A\in\Ob\Tau}\Tau(A,B) \simeq  
\amalg_{A,B\in\Ob\Tau} \Tau(A,B).$$
\end{lma}

\begin{ex} \exlabel{GPS}
Let $\Bicat_3$ denote the tricategory of bicategories, pseudofunctors, pseudonatural transformations and modifications, 
it is clearly 1-strict. A tricategory from \cite[Definition 2.2]{GPS} is a category enriched in $\Bicat_3$ such that  
$\mu^\dagger$ and $\Epsilon^\dagger$ are identities. Of course, the latter can not be used as a definition of the notion 
of a tricategory. Rather, one says that a tricategory is a category {\em weakly enriched over the category} $\Bicat_1$ of 
bicategories and pseudofunctors. More general, instead of saying ``a category enriched over a 1-strict tricategory $V$'' 
one could say ``a category {\em weakly enriched over the underlying category} of $V$''. 
\end{ex}

\subsection{Enriched categories as internal categories in 1-strict tricategories}

Let $V$ be a tricategory with a terminal object $\I$, finite tricategorical products and tricategorical pullbacks. Then 
observe that a 3-product $X\times Y$ is in particular a 3-pullback $\big(X\times_\I Y; t_X,t_Y\big)$, where $t_X,t_Y$ are the unique 
morphisms into $\I$. Moreover, a 3-product $X\times Y\times Z$ is a 3-pullback $\big((X\times Y)\times_Y (Y\times Z); p_2,p_1\big)$. 
In particular, for $Y=Y_1\times\dots\times Y_k$ for any natural $k$, a 3-product $X\times Y_1\times\dots\times Y_k\times Z$ is 
a 3-pullback $\big((X\times Y_1\times\dots\times Y_k)\times_{(Y_1\times\dots\times Y_k)} (Y_1\times\dots\times Y_k\times Z); 
p_2,p_1\big)$. 

In this section we deal with ``hands on enrichment'' and for this we found it easier to use lexicographical order when 
writing 3-products and 3-pullbacks (contrary to \rmref{pb order}).

\begin{prop} \prlabel{biequiv a_n}
Let $V$ be a 1-strict tricategory with finite 3-products, a terminal object $\I$ and small tricategorical coproducts. 
Assume that $\Tau$ is a category enriched over $V$,  set $T_0=\amalg_{A\in\Ob\Tau} \I_A$ - the coproduct of copies of 
the terminal object indexed by the objects of $\Tau$, and $T_1=\amalg_{A,B\in\Ob\Tau} \Tau(A,B)$, and suppose that $V$ 
has iterated 3-pullbacks $T_1^{(n)_0}$. If additionally the following conditions are fulfilled: 
\begin{enumerate} 
\item for every natural $n\geq 2$ the trifunctors $\amalg_{B_1,\dots,B_{n-1}}$ preserve the following 3-pullbacks:
$\big(\amalg_A\Tau(A,B_1)\big)\times\Tau(B_1,B_2)\times\dots\times\Tau(B_{n-2},B_{n-1})\times\big(\amalg_C\Tau(B_{n-1},C)\big)$; 
\item the trifunctors $X\times-$ and $-\times X$, for $X\in\Ob\Tau$,  
preserve the coproducts $\amalg_A\Tau(A,B)$ and $\amalg_B\Tau(A,B)$; 
\end{enumerate}
then the resulting 3-pullbacks in 1. are $\amalg_{A, B_1\dots,B_{n-1}, C\in\Ob\Tau} \Tau(A,B_1)\times\dots\times\Tau(B_{n-1},C)$ and 
for every natural $n\geq 2$ there are equivalence 1-cells in $V$: 
$$a^n_{A_1,\dots,A_{n+1}}: \amalg_{A_1,\dots,A_{n+1}\in\Ob\Tau} \Tau(A_1,A_2)\times\dots\times\Tau(A_{n},A_{n+1})
\stackrel{\simeq}{\to} 
\underbrace{T_1\times_{T_0}\dots \times_{T_0} T_1}_n$$
(with all possible distributions of parentheses).
\end{prop}

\begin{proof}
We will do the proof for the cases when $n$ equals 2 and 3, the higher cases are proven in analogous way. 
For $n=2$ we start by a 3-pullback $\big(\amalg_A\Tau(A,B)\big)\times\big(\amalg_C\Tau(B,C)\big)$ (over $\I$),
and act by the trifunctor $\amalg_B$ on it.  
By (1) we get the following 3-pullback, where in the first coordinate we apply the preservation property (2), and in the 
second and third the corresponding equivalences of the coproducts (in the rest of coordinates by abuse of notation we do not 
change the notation of the 1-cells for simplicity reasons): 
$$(\amalg_{A,B,C}\Tau(A,B)\times\Tau(B,C), \,\, 
\amalg_{A,B}\Tau(A,B), \,\, \amalg_{B,C}\Tau(B,C), \,\,
\amalg_B \I_B, \,\, \amalg_B p_1, \,\, \amalg_B p_2;\,\, \amalg_B t, \,\, \amalg_B t).$$ 
On the other hand, by construction this 3-pullback is $(T_1\times_{T_0}T_1,s,t)$. Thus there is an equivalence 
$$a^2_{A,B,C}: \amalg_{A,B,C}\Tau(A,B)\times\Tau(B,C)\stackrel{\simeq}{\to}T_1\times_{T_0}T_1.$$

\medskip 

For $n=3$ we start with a 3-pullback 
\begin{equation} \eqlabel{tri pb}
(\big[\big(\amalg_A\Tau(A,B)\big)\times\Tau(B,C)\big]\times_{\Tau(B,C)}\big[\Tau(B,C)\times\big(\amalg_D\Tau(C,D)\big)\big], 
p_2,p_1 \big),
\end{equation}
which can be rewritten as the 3-product: 
\begin{multline*}
(\big(\amalg_A\Tau(A,B)\big)\times\Tau(B,C)\times\big(\amalg_D\Tau(C,D)\big), \,\, 
\big(\amalg_A\Tau(A,B)\big)\times\Tau(B,C), \,\, \Tau(B,C)\times\big(\amalg_D\Tau(C,D)\big), \\
\Tau(B,C), p^3_{12}, p^3_{23}; p_2,p_1 \big).
\end{multline*}
We act on it by the trifunctor $\amalg_B\amalg_C\simeq\amalg_{B,C}\simeq\amalg_C\amalg_B$ and get by (1) a 3-pullback, 
which by (2) has the form: 
\begin{multline*}
(\amalg_{A,B,C,D}\big(\Tau(A,B)\times\Tau(B,C)\times\Tau(C,D)\big), \,\, \amalg_{A,B,C}\big(\Tau(A,B)\times\Tau(B,C)\big), \,\, 
\amalg_{B,C,D}\big(\Tau(B,C)\times\Tau(C,D)\big), \\
\amalg_{B,C}\Tau(B,C), \,\, \amalg_{B,C}p_{12}, \,\, \amalg_{B,C}p_{23}, \,\, \amalg_{B,C}p_2, \,\, \amalg_{B,C}p_1\big)
\end{multline*}
and by construction (see \equref{tri pb}) it is indeed the 3-pullback 
$T_1\times_{T_0}T_1\times_{T_0}T_1$ (we differentiate the two distributions of the parentheses).  
This yields equivalences 
$$a^{3,L}_{A,B,C,D}: \amalg_{A,B,C,D}\big(\Tau(A,B)\times\Tau(B,C)\big)\times\Tau(C,D)\stackrel{\simeq}{\to}(T_1\times_{T_0}T_1)\times_{T_0}T_1$$
and 
$$a^{3,R}_{A,B,C,D}: \amalg_{A,B,C,D}\Tau(A,B)\times\big(\Tau(B,C)\times\Tau(C,D)\big)\stackrel{\simeq}{\to}T_1\times_{T_0}(T_1\times_{T_0}T_1).$$

For general $n$ one obtains that by construction $\amalg_{A, B_1\dots,B_{n-1}, C\in\Ob\Tau} \Tau(A,B_1)\times\dots\times\Tau(B_{n-1},C)$ 
is $\big(\underbrace{T_1\times_{T_0}\dots \times_{T_0} T_1}_{n-1}\big) \hspace{0,14cm}
\times_{\underbrace{T_1\times_{T_0}\dots \times_{T_0} T_1}_{n-2}} \hspace{0,14cm}
\big(\underbrace{T_1\times_{T_0}\dots \times_{T_0} T_1}_{n-1} \big)$, which is equivalent to 
$\underbrace{T_1\times_{T_0}\dots \times_{T_0} T_1}_n$. 
\qed\end{proof}

Let us fix the notation for the associated equivalence 2-cells for the above equivalences: 
\begin{equation} \eqlabel{2-isos} 
\alpha^{3,L}:(a^{3,L})^{-1}\comp a^{3,L}\stackrel{\simeq}{\Rightarrow}\Id, 
\hspace{0,14cm} \alpha^{3,R}:(a^{3,R})^{-1}\comp a^{3,R}\stackrel{\simeq}{\Rightarrow}\Id, \hspace{0,14cm} 
\alpha^2: (a^2)^{-1}\comp a^2\stackrel{\simeq}{\Rightarrow}\Id 
\end{equation}
(by $(\bullet)^{-1}$ here we denote a quasi-inverse).

\begin{prop} \prlabel{weak-int}
In the conditions of the previous Proposition, a category $\Tau$ enriched over $V$ is a particular case of a category internal in $V$. 
\end{prop}

\begin{proof}
By the 3-coproduct property, 
the composition from the enrichment $\circ:\Tau(A,B)\times\Tau(B,C)\to\Tau(A,C)$ induces a 1-cell $\crta\circ$ up to an equivalence 
2-cell $\kappa$, 
and the 1-cell $\circ\times\id$ induces a 1-cell $\crta\circ_{12}$ up to an equivalence 2-cell $\zeta_L$, 
as indicated in the two left squares in the diagram:  
$$
\bfig
 \putmorphism(-400,800)(1,0)[\big(\Tau(A,B)\text{x}\Tau(B,C)\big)\text{x}\Tau(C,D)`\amalg_{A,B,C,D} \big(\Tau(A,B)\text{x}\Tau(B,C)\big)\text{x}\Tau(C,D)`
\iota^4_{A,B,C,D}]{1660}1a 
\putmorphism(1960,800)(1,0)[\phantom{Y_2}`(T_1\times_{T_0}T_1)\times_{T_0}T_1`a^{3,L}_{A,B,C,D}]{740}1a
\putmorphism(1300,800)(0,-1)[\phantom{Y_2}``\crta\circ_{12}]{400}1l
\putmorphism(-180,800)(0,-1)[\phantom{Y_2}``\circ\times\id]{380}1l
\putmorphism(2500,800)(0,-1)[\phantom{Y_2}`` c\times_{T_0}\id]{400}1r
\put(580,580){$\Swarrow\zeta_L$}
\put(1900,580){$\Swarrow\zeta'_L$}

 \putmorphism(-150,400)(1,0)[\Tau(A,C)\times\Tau(C,D)`\amalg_{A,C,D\in\Ob\Tau} \Tau(A,C)\times\Tau(C,D)`\iota^3_{A,C,D}]{1300}1a 
\putmorphism(1300,400)(0,-1)[\phantom{Y_2}``\crta\circ]{400}1l
\putmorphism(1730,400)(1,0)[\phantom{Y_2}`T_1\times_{T_0}T_1`a^2_{A,C,D}]{800}1a

 \putmorphism(-30,0)(1,0)[`\amalg_{A,D\in\Ob\Tau} \Tau(A,D)`\iota^2_{A,D}]{1380}1a 
\putmorphism(-180,400)(0,-1)[\phantom{Y_2}`\Tau(A,D)`\circ]{380}1l
\putmorphism(2500,400)(0,-1)[\phantom{Y_2}`T_1.` c]{400}1r
\putmorphism(1730,0)(1,0)[\phantom{Y_2}` `=]{690}1a
\put(580,180){$\Swarrow\kappa$}
\put(1900,180){$\Swarrow\xi$}
\efig
$$
Using the equivalences $a^{3,L}, a^2$ and their quasi-inverses, the 1-cells $\crta\circ$ and $\crta\circ_{12}$ 
induce 1-cells $c:T_1\times_{T_0}T_1\to T_1$ and $c\times_{T_0}\id$ up to equivalence 2-cells $\xi$ and $\zeta'_L$ in $V$, respectively,  
in the two right squares above. 
Here $a^2, a^{3,L}$ are biequivalences from the above Proposition, and $\iota$'s are the corresponding 3-coproduct structure embeddings. 
Observe that $\xi$ and $\zeta'_L$ are given through the 2-cells \equref{2-isos} (horizontally composed with suitable identity 2-cells). 


\medskip

Now one may draw an analogous diagram to the above one in a parallel plane above it, where now $\id\times\circ$ induces a 1-cell 
$\crta\circ_{23}$ up to an equivalence 2-cell $\zeta_R$, and $\crta\circ_{23}$ induces a 1-cell $\id\times_{T_0}c$ up to an equivalence 
2-cell $\zeta'_R$. 
From the enrichment we have an equivalence 2-cell $a^\dagger$ up to which the pentagon for the associativity of $\circ$ commutes. One can draw this 
pentagon transversally to the plane of the paper so that it connects the two diagrams, in the two planes, at their extreme left edges, adding a fifth edge 
for the associativity in the top 0-cell. The latter associativity 1-cell induces associativity 1-cells between the 3-coproduct and 3-pullback 
0-cells, by the property of a 3-coproduct and via the equivalences $a^{3,L},a^{3,R}$, respectively. 
Now the 2-cell $a^\dagger$ induces an equivalence 2-cell $\crta a$, 
up to which $\crta\circ$ is associative, so to say, it connects the two diagrams transversally at the level of the 3-coproduct vertices. 
Finally, $\crta a$ induces an equivalence 2-cell $a^*$ up to which $c$ is associative, connecting the two diagrams transversally at their 
extreme right edges. 

To understand how $\crta a$ and $a^*$ are defined, observe that the three 2-cells $a^\dagger$ and yet-to-be-defined $\crta a$ and $a^*$ divide this three-dimensional diagram into two horizontal prisms with pentagonal bases in the transversal direction. Now consider the two prisms as 
3-cells between the following two pairs of compositions of 2-cells: 

$
\bfig
\putmorphism(-10,390)(1,0)[``]{800}1a
\putmorphism(-20,420)(1,-1)[``]{350}1l 
  \putmorphism(60,420)(-1,-1)[``]{350}1l
\putmorphism(720,420)(1,-1)[``]{330}1l
\putmorphism(240,120)(1,0)[``]{780}1a
\putmorphism(-260,170)(1,-1)[``]{330}1l
  \putmorphism(300,170)(-1,-1)[``]{330}1l
  \putmorphism(1060,150)(-1,-1)[``]{330}1l
\putmorphism(0,-120)(1,0)[``]{800}1a
\put(0,120){$a^\dagger$}
\put(460,220){$\zeta_L$}
\put(400,-20){$\kappa_{A,C,D}$}
\efig
\hspace{3cm} 
\bfig
\putmorphism(-10,390)(1,0)[``]{800}1a
\putmorphism(-20,420)(1,-1)[``]{350}1l 
  \putmorphism(60,420)(-1,-1)[``]{350}1l
\putmorphism(720,420)(1,-1)[``]{330}1l
\putmorphism(240,120)(1,0)[``]{780}1a
\putmorphism(-260,170)(1,-1)[``]{330}1l
  \putmorphism(300,170)(-1,-1)[``]{330}1l
  \putmorphism(1060,150)(-1,-1)[``]{330}1l
\putmorphism(0,-120)(1,0)[``]{800}1a
\put(0,120){$\crta a$}
\put(460,220){$\zeta'_L$}
\put(400,-20){$\xi_{A,C,D}$}
\efig
$
$$\Downarrow \hspace{8cm} \Downarrow$$

$
\bfig
\putmorphism(200,430)(1,0)[``]{760}1a
  \putmorphism(190,390)(1,0)[\simeq``]{760}0a
\putmorphism(110,360)(1,0)[``]{730}1a
  \putmorphism(100,380)(1,0)[`\simeq`]{750}0a

  \putmorphism(200,400)(-1,-1)[``]{330}1l
\putmorphism(890,450)(1,-1)[``]{340}1l %
\putmorphism(-70,130)(1,0)[``]{720}1a
\putmorphism(-120,170)(1,-1)[``]{330}1l 
  \putmorphism(1210,220)(-1,-1)[``]{400}1l 
\putmorphism(130,-120)(1,0)[``]{760}1a

\putmorphism(570,170)(1,-1)[``]{330}1l
  \putmorphism(900,400)(-1,-1)[``]{330}1l
\put(850,120){$\crta a$}
\put(340,220){$\zeta_R$}
\put(280,-20){$\kappa_{A,B,D}$}
\efig
\hspace{3cm} 
\bfig
\putmorphism(200,430)(1,0)[``]{760}1a
  \putmorphism(190,390)(1,0)[\simeq``]{760}0a
\putmorphism(110,360)(1,0)[``]{730}1a
  \putmorphism(100,380)(1,0)[`\simeq`]{750}0a

  \putmorphism(200,400)(-1,-1)[``]{330}1l
\putmorphism(890,450)(1,-1)[``]{340}1l %
\putmorphism(-70,130)(1,0)[``]{720}1a
\putmorphism(-120,170)(1,-1)[``]{330}1l 
  \putmorphism(1210,220)(-1,-1)[``]{400}1l 
\putmorphism(130,-120)(1,0)[``]{760}1a

\putmorphism(570,170)(1,-1)[``]{330}1l
  \putmorphism(900,400)(-1,-1)[``]{330}1l
\put(850,120){$a^*$}
\put(340,220){$\zeta'_R$}
\put(280,-20){$\xi_{A,B,D}$}
\efig
$

\noindent 
Now we define a 2-cell whose domain and codomain coincide with those of the 
``horizontal composition of the place where we wrote $\crta a$ and the identity 2-cell 
of the left upper 1-cell $\iota^4_{A,B,C,D}$'' as the appropriate composition of the 
resting six 2-cells (and their quasi-inverses) in the left prism. Then by the property of a 3-coproduct there is a 2-cell 
$\crta a$ at the place we want to have it, namely, so that $\crta a\ot\Id_{\iota^4}$ is isomorphic to the mentioned composition. 
The 2-cell $a^*$ is obtained similarly, with the difference that instead of the 3-coproduct property one uses the fact that 
all horizontal 1-cells in the right prism above are equivalences, thus $a^*$ is given as a suitable composition of the 2-cells 
$\alpha^2, \crta a, \xi^{-1}$ and $(\zeta_R')^{-1}$ (recall that by $(\bullet)^{-1}$ we denote the corresponding quasi-inverse 2-cells).

From the enrichment we have an invertible 3-cell 
$$\pi^\dagger: \threefrac{\Id_{\circ} \ot(\Id_{id_T}\times a^\dagger)}{a^\dagger\ot\Id_{1\times\circ\times 1}}{\Id_{\circ} \ot 
(a^\dagger\times \Id_{id_T})} \Rrightarrow
\threefrac{a^\dagger \ot\Id_{1\times 1\times\circ}}{\Id_{\circ} \ot\Nat_{(\circ\times 1)(1\times 1\times \circ)}}
{a^\dagger \ot\Id_{\circ\times 1\times 1}}$$
satisfying a septagonal identity (here $T$ stands for 0-cells of the form $\Tau(A,B), A,B\in V$). 
Let us denote the domain and the codomain 2-cells of $\pi$ by $L(a^\dagger)$ and $R(a^\dagger)$, respectively. 
We next show that $\pi^\dagger$ induces a 3-cell $\pi^*:L(a^*)\Rrightarrow R(a^*)$, where now $L,R$ have the obvious adjusted meaning. 
Consider the following two 2-cells in $V$ whose upper 
1-cells are the domain and codomain of $L(a^\dagger)$, respectively, and whose lower 1-cells are the domain and codomain of $L(a^*)$, respectively: 
$$
\Omega^L=
\bfig
 \putmorphism(-150,500)(1,0)[T^4` `\circ\times\id\times\id]{600}1a
\putmorphism(-180,500)(0,-1)[\phantom{Y_2}`\amalg T^4 `\iota^5]{450}1r
\put(-30,290){\fbox{$(\zeta^4_{12})^{-1}$}}
\put(-30,-170){\fbox{$(\zeta'^4_{12})^{-1}$}}
\putmorphism(-210,-400)(1,0)[T^4_1` `c\times\id\times\id]{600}1a
\putmorphism(-180,50)(0,-1)[\phantom{Y_2}``a^4]{420}1r
\putmorphism(380,500)(0,-1)[\phantom{Y_2}` `\iota^4]{450}1r
\putmorphism(380,50)(0,-1)[` `a^3]{450}1r
\putmorphism(920,500)(0,-1)[\phantom{Y_2}` `\iota^3]{450}1r
\putmorphism(920,50)(0,-1)[\phantom{Y_2}` `a^2]{450}1r

\putmorphism(430,500)(1,0)[``\circ\times\id]{500}1a
 \putmorphism(830,500)(1,0)[\phantom{F(A)}` `\circ]{500}1a
  \putmorphism(-110,50)(1,0)[` `\crta\circ_{12}^4]{530}1a
\putmorphism(430,50)(1,0)[``\crta\circ_{12}^3]{500}1a
 \putmorphism(920,50)(1,0)[``\crta\circ]{440}1a
\putmorphism(1300,500)(0,-1)[\phantom{Y_2}``\iota^2]{450}1r
\put(530,290){\fbox{$(\zeta^3_{12})^{-1}$}}
\put(1050,290){\fbox{$\kappa^{-1}$}}
\putmorphism(410,-400)(1,0)[``c\times \id ]{500}1a
\putmorphism(900,-400)(1,0)[` `c]{440}1a
\putmorphism(1300,50)(0,-1)[\phantom{Y_2}``=]{450}1r
\put(530,-170){\fbox{$(\xi'^3_{12})^{-1}$}}
\put(1050,-170){\fbox{$\xi^{-1}$}}
\efig
\Omega_L= 
\bfig
 \putmorphism(-150,500)(1,0)[T^4` `\id\times\id\times\circ]{600}1a
\putmorphism(-180,500)(0,-1)[\phantom{Y_2}`\amalg T^4 `\iota^5]{450}1r
\put(-30,290){\fbox{$(\zeta^4_{34})^{-1}$}}
\put(-30,-170){\fbox{$(\zeta'^4_{34})^{-1}$}}
\putmorphism(-210,-400)(1,0)[T^4_1` `\id\times\id\times c]{600}1a
\putmorphism(-180,50)(0,-1)[\phantom{Y_2}``a^4]{420}1r
\putmorphism(380,500)(0,-1)[\phantom{Y_2}` `\iota^4]{450}1r
\putmorphism(380,50)(0,-1)[` `a^3]{450}1r
\putmorphism(920,500)(0,-1)[\phantom{Y_2}` `\iota^3]{450}1r
\putmorphism(920,50)(0,-1)[\phantom{Y_2}` `a^2]{450}1r

\putmorphism(430,500)(1,0)[``\id\times\circ]{500}1a
 \putmorphism(830,500)(1,0)[\phantom{F(A)}` `\circ]{500}1a
  \putmorphism(-110,50)(1,0)[` `\crta\circ_{34}^4]{530}1a
\putmorphism(430,50)(1,0)[``\crta\circ_{23}^3]{500}1a
 \putmorphism(920,50)(1,0)[``\crta\circ]{440}1a
\putmorphism(1300,500)(0,-1)[\phantom{Y_2}``\iota^2]{450}1r
\put(530,290){\fbox{$(\zeta^3_{23})^{-1}$}}
\put(1050,290){\fbox{$\kappa^{-1}$}}
\putmorphism(410,-400)(1,0)[``\id\times c ]{500}1a
\putmorphism(900,-400)(1,0)[` `c]{440}1a
\putmorphism(1300,50)(0,-1)[\phantom{Y_2}``=]{450}1r
\put(530,-170){\fbox{$(\xi'^3_{23})^{-1}$}}
\put(1050,-170){\fbox{$\xi^{-1}$}}
\efig
$$
Analogously, one defines $\Omega^R$ and $\Omega_R$ for the corresponding 2-cells for $R(a^\dagger)$ 
and $R(a^*)$. Then let us consider the following two 3-cells which we draw here in the form of two cylinders with 
oval horizontal bases: 
\begin{figure}[hbt!]
\centering
\includegraphics[width=140mm]{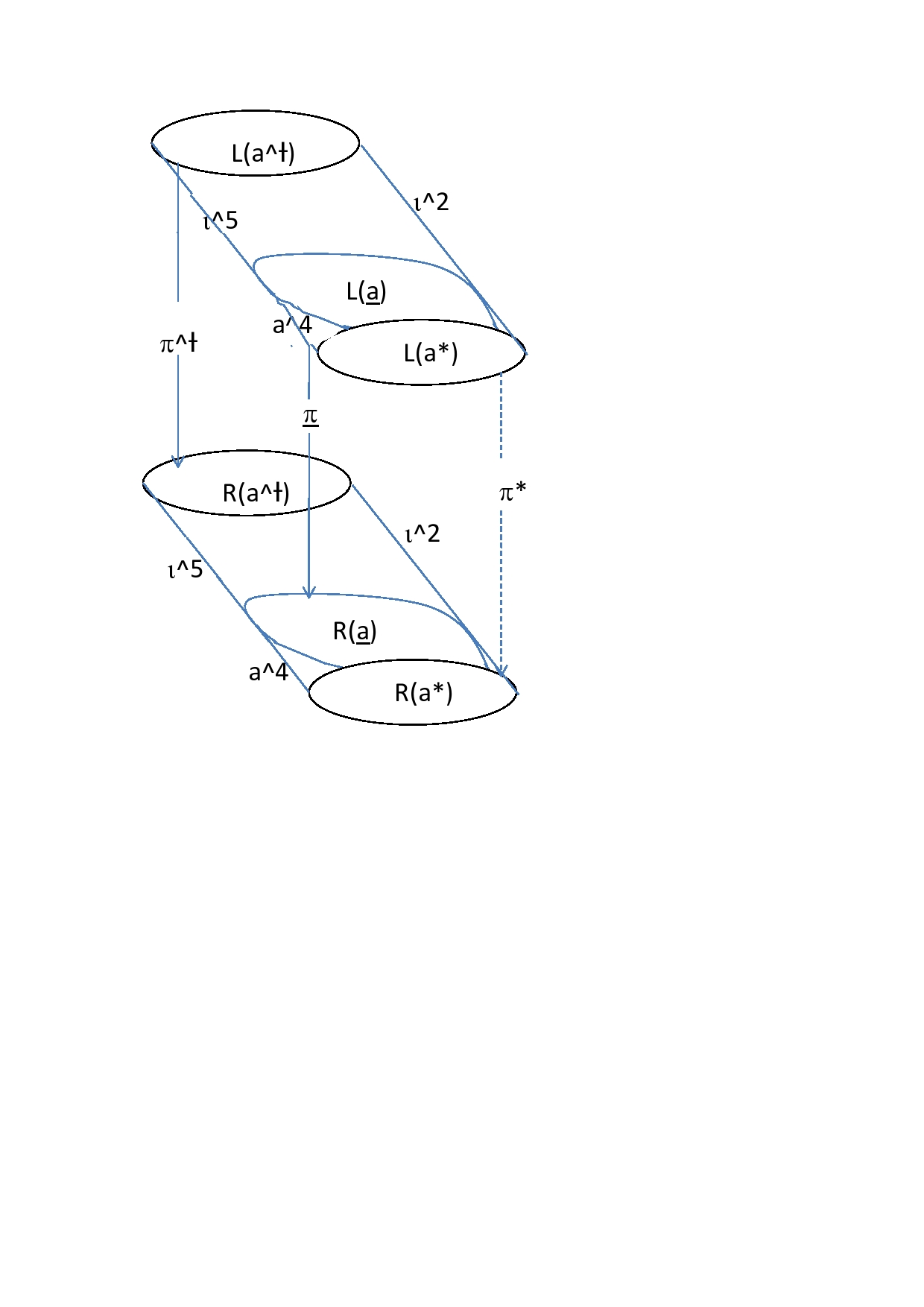}
\end{figure}
\bigskip


\bigskip

\bigskip

\vspace{7cm}
\noindent The upper 2-cell of the upper cylinder is $\Omega^L$ and the lower $\Omega_L$, while the upper 2-cell of the lower cylinder 
is $\Omega^R$ and the lower $\Omega_R$. Since the above prism with pentagonal basis, by which the 2-cell $a^\dagger$ induced a 2-cell $a^*$, 
is an isomorphism 3-cell (transversal composition of two isomorphism 3-cells), then these two cylinders with oval bases (each of which is 
an analogous transversal composition of two 3-cells, as the before mentioned composition of isomorphism 3-cells), are also isomorphism 3-cells. 
Then $\pi^\dagger$ induces a 3-cell $\crta\pi$, and the latter induces the wanted 3-cell $\pi^*$. Namely, start from the 2-cell 
$L(\crta a)\ot\Id_{\iota^5}$, it is isomorphic to the 2-cell $\Id_{\iota^2}\ot L(a^\dagger)$, map the latter by the 3-cell 
$\Id_{\Id_{\iota^2}}\ot\pi^\dagger$ to the 2-cell $\Id_{\iota^2}\ot R(a^\dagger)$, which is isomorphic to the 2-cell $R(\crta a)\ot\Id_{\iota^5}$, 
then this resulting 3-cell, which is a certain conjugation of $\Id_{\Id_{\iota^2}}\ot\pi^\dagger$, determines a unique 3-cell $\crta\pi$, 
by the 3-coproduct property, so that $\Id_{\Id_{\iota^2}}\ot\pi^\dagger$ equals $\crta\pi\ot\Id_{\Id_{\iota^5}}$. For $\pi^*$: start with the 2-cell 
$L(a^*)$, it is isomorphic to $L(\crta a)\ot \Id_{(a^4)^{-1}}$, map the latter by $\crta\pi\ot\Id_{\Id_{a^4}}$ to $R(\crta a)\ot \Id_{(a^4)^{-1}}$, 
which is isomorphic to $R(a^*)$. This composition 3-cell is the wanted $\pi^*$. (Recall that $(a^4)^{-1}$ denotes a quasi-inverse.) 
Since $\pi^\dagger$ satisfies a septagonal identity, so does $\pi^*$ in its desired form. 

\medskip

This proves the existence of a composition $c:T_1\times_{T_0}T_1\to T_1$ associative up to an equivalence for a structure of a category 
internal in $V$. 

In an analogous way the unit 1-cell $I_A:\I\to\Tau(A,A)$ from enrichment induces a unit 1-cell $u: T_0\to T_1$, and the 2-cells 
$l^\dagger,r^\dagger$ for the unity law in the enrichment induce 2-cells $l^*,r^*$ for the unity law in an internal category in $V$. 
The induction of the associated 3-cells $\lambda^*$ and $\rho^*$, but also of $\mu^*$ and $\Epsilon^*$, from the 3-cells from the enrichment 
$\lambda^\dagger, \rho^\dagger, \mu^\dagger$ and $\Epsilon^\dagger$, respectively, goes the analogous way as we proved it above for $\pi^\dagger$. 
\qed\end{proof}

Observe that by the construction of $T_0$ in the above proof, if $V$ is a 1-strict tricategory whose 0-cells are bicategories, 
like in \exref{GPS}, then 0-cells of $T_0$ are the same as 0-cells of $\Tau$. Moreover, 1-cells of $T_0$ are the identities on 
its 0-cells and 2-cells are identities on the latter, {\em i.e.} the object of objects $T_0$ is discrete. 

\subsection{Examples of enrichment and internalization in lower dimensions}

In the examples where $V$ is some kind of ``category of categories'', for the existence of the iterated $n$-pullbacks, $n=1,2,3$, 
it is sufficient to require that source and target 1-cells $s,t$ be {\em strict} functors. We illustrate this by showing it for the 2-category 
$PsDbl_2$ of pseudodouble categories, pseudodouble functors and vertical transformations, \cite{Shul, GG} (then it also applies 
to the 2-category $Cat_2$ of categories, functors and natural transformations). Namely, as in \cite[Proposition 2.1]{GP}, 
we have: 

\begin{prop}
The set-theoretical pullback of strict double functors $F:\A\to\C$ and $G:\B\to\C$ determines on objects, 1-cells and 2-cells 
a pseudo double category $\A\times_\C\B$ which is the 2-pullback of $F$ and $G$ in $PsDbl_2$. The projections onto $\A$ and $\B$ are 
strict double functors. 
\end{prop}

The construction in the proposition from the last subsection can be carried out in 2-categories: consider 3-cells to be identities, 
and consider equivalence 2-cells to be bijective. Then we obtain:

\begin{cor} \colabel{weak-int}
Let $V$ be a Cartesian monoidal 2-category 
with a terminal object $\I$, finite 2-products and small 2-coproducts preserved by the pseudofunctors $-\times X, X\times -:V\to V$ 
for every $X\in V$. Let $\Tau$ be a category enriched over $V$, set 
$T_0=\amalg_{A\in\Tau}\I_A$ and $T_1=\amalg_{A,B\in \Tau}\Tau(A,B)$, suppose that the iterated 
2-pullbacks $T_1^{(n)_0}$ exist and that the pseudofunctors $\amalg_{B_1,\dots,B_{n-1}}$ map the 2-products: 
$\big(\amalg_A\Tau(A,B_1)\big)\times\Tau(B_1,B_2)\times\dots\times\Tau(B_{n-2},B_{n-1})\times\big(\amalg_C\Tau(B_{n-1},C)\big)$ 
to iterated 2-pullbacks: $\underbrace{T_1\times_{T_0}\dots \times_{T_0} T_1}_n$.  
Then $\Tau$ is a particular case of a category internal to $V$. 
\end{cor}

\begin{ex}
A category enriched over the 2-category $V=Cat_2$ 
is a bicategory, and it is well-known that a bicategory embeds into a pseudodouble category, which is a category internal in $Cat_2$. 
\end{ex}

\begin{ex}
A category enriched over $PsDbl_2$ is a locally cubical bicategory from \cite[Definition 11]{GG}. 
A category internal to $PsDbl_2$ is a version of an intercategory. \coref{weak-int} applied to $V=PsDbl_2$ uses the argumentation 
similar to \cite[Section 3.5]{GP:Fram}, where a locally cubical bicategory is shown to be a particular case of an intercategory. 
\end{ex}

Truncating the result of \coref{weak-int} to 1-categories one recovers a version of the results in \cite[Appendix]{Ehr:III} and \cite{Power}. 
As a particular case of this we have the following. 
A Gray-category is a category enriched over the monoidal category $Gray$ with the Gray monoidal product. This product is defined 
as an image of a cubical functor defined on the Cartesian product of two 2-categories. In \cite[Section 5.2]{GP:Fram} it is shown 
how a Gray-category can be seen as an intercategory, a category internal in $LxDbl$. As an intermediate step one can see how a 
Gray-category is made a category internal in $Gray$. 


In the above three examples we can embed the 1-category $Gray$ and the 2-categories $Cat_2$ and $PsDbl_2$ to our tricategory $\DblPs$ 
and we get three examples of categories internal in $\DblPs$.


\section{Tricategory of tensor categories: enrichment and internalization} 

Apart from our search for an alternative framework to intercategories and what we developed in \seref{int in DblPs}, 
we had another motivating example to introduce internal categories in a tricategory in \seref{int3}. 
Namely, analogously to the double category of rings, in one dimension higher we have a $(1\times 2)$-category of tensor categories 
(for this name see {\em e.g.} \cite{Shul}). It is an internal category in a suitable tricategory $V$, so that the category of objects 
consists of tensor categories, tensor functors and tensor natural transformations (thus the vertical direction is strictly 
associative and unital), while the category of morphisms consists of bimodule categories, bimodule functors, and bimodule natural transformations. Since the associativity for the relative tensor product of bimodule categories is an equivalence (and not an isomorphism!),  the horizontal direction of this $(1\times 2)$-category is tricategorical in nature. 
Clearly, the tricategory $\Tens$ of tensor categories, with 0-cells tensor categories, 1-cells bimodule categories, 2-cells bimodule 
functors, and 3-cells bimodule natural transformations lies in this $(1\times 2)$-category. 

In the first subsection below we will show that the tricategory $\Tens$ is enriched over the tricategory $2Cat_{wk}$, of 2-categories, 
pseudofunctors, weak natural transformations and modifications. Note that $2\x\Cat_{wk}$ is 1-strict. In the second subsection we will 
present an internal category structure for $\Tens$ in $2Cat_{wk}$ richer than the one coming from \prref{weak-int}, that is, where the object 
of objects $T_0$ is not descrete.


\subsection{$\Tens$ as an enriched category} 

Let us recall and discuss the structure of a tricategory $\Tens$ of tensor categories. 
\begin{enumerate} 
\item For every two tensor categories $\C$ an $\D$ we have a 2-category $\Bimod(\C,\D)$; 
\item given two $\C\x\D$-bimodule categories $\M, \N$ there is a category $\Bimod(\C,\D)\big(\M,\N\big)={}_\C\Fun_\D(\M,\N)$ whose 
composition of morphisms is given by the vertical composition of $\C\x\D$-bimodule natural transformations, which we denote by $\cdot$ 
(this is the transversal composition of 3-cells in $\Tens$); 
\item given a third $\C\x\D$-bimodule category $\Ll$ there is a functor 
$\comp: {}_\C\Fun_\D(\N,\Ll)\times{}_\C\Fun_\D(\M,\N)\to{}_\C\Fun_\D(\M,\Ll)$ 
given by the composition of $\C\x\D$-bimodule functors and $\C\x\D$-bimodule natural transformations; 
thus the horizontal composition of 2-cells in $\Bimod(\C,\D)$ is given by the usual horizontal composition of natural transformations 
(this is the vertical composition of 2- and 3-cells in $\Tens$); 
by the functor properties of $\comp$, on this 2-category level we have the strict interchange law: 
$(\zeta'\cdot\omega')\comp(\zeta\cdot\omega)=(\zeta'\comp\zeta)\cdot(\omega'\comp\omega)$ for accordingly composable natural transformations 
$\omega, \omega', \zeta, \zeta'$; 
\item given a third tensor category $\E$ there is a pseudofunctor $\Del_\D: \Bimod(\D,\E)\times\Bimod(\C,\D)\to\Bimod(\C,\E)$, so 
the composition of 1-cells and the horizontal composition of 2- and 3-cells in $\Tens$ is given by the relative tensor 
product of bimodule categories. Let $(\M,\N),(\M', \N')\in\Bimod(\D, \E)\times\Bimod(\C, \D)$, then set for the hom-set 
$$\Bimod(\D, \E)\times\Bimod(\C, \D)\big((\M,\N),(\M', \N')\big)={}_\C\Fun_\E^{\D\x bal}\big((\M,\N),(\M', \N')\big),$$ 
which is the category of $\D$-balanced $\C\x\E$-bimodule functors and natural transformations. Then there is a functor 
$\widetilde\s : {}_\C\Fun_\E^{\D\x bal}\big((\M,\N),(\M', \N')\big) \to {}_\C\Fun_\E(\N\Del_\D\M, \N'\Del_\D\M')$ and there are natural isomorphisms 
$\widetilde G\comp\widetilde F \iso\widetilde{G\comp F}$ and $\Id_{\N\Del_\D\M}\iso\widetilde{\Id_{(\M,\N)}}$ for all 
$F\in{}_\C\Fun_\E^{\D\x bal}\big((\M,\N),(\M', \N')\big)$ and $G\in{}_\C\Fun_\E^{\D\x bal}\big((\M',\N'),(\M'', \N'')\big)$  
(this corresponds to the bimodule case of \cite[Proposition 3.3.2]{Schaum1}). In particular, 
the latter natural isomorphisms imply that we have the interchange law at this level holding up to an isomorphism: 
$(\F'\Del_\D\G')\comp(\F\Del_\D\G)\iso(\F'\comp\F)\Del_\D(\G'\comp\G)$ for according bimodule functors, and also: 
$\Id_{\N\Del_\D\M}\iso\Id_\N\Del_\D\Id_\M$. The above functor property implies in particular: 
$(\zeta'\comp\omega')\Del_\D(\zeta\comp\omega)=(\zeta'\Del_\D\zeta)\comp(\omega'\Del_\D\omega)$ for according bimodule natural 
transformations, and $\Id_{\G\Del_\D\F}=\Id_\G\Del_\D\Id_\F$; 
\item for 0-, 1- and 2-cells $\C, \M$ and $\F$ respectively there are identity 1-, 2- and 3-cells $\C, \id_\M$ and $\Id_\F$, respectively; 
\item there are pseudonatural equivalences $a,l,r$ so that concretely for the corresponding bimodule categories one has equivalence functors: 
$a_{\M,\N,\Ll}: (\M\Del_\C\N)\Del_\D\Ll\stackrel{\cong}{\to}\M\Del_\C(\N\Del_\D\Ll), 
l_\N:\C\Del_\C\N\stackrel{\cong}{\to}\N$ and $r_\N:\N\Del_\D\D\stackrel{\cong}{\to}\N$ (observe that the respective naturalities hold up to 
natural isomorphisms); 
\item there are modifications $\pi, \mu, \lambda$ and $\rho$ which evaluated at bimodule categories give natural isomorphisms 
$$\pi: (\id_\K\Del_\C a_{\N,\M,\Ll})\comp a_{\K, \N\Del_\D\M, \Ll}\comp(a_{\K, \N,\M}\Del_\E\id_\Ll) \Rrightarrow a_{\K,\N,\M\Del_\E\Ll}\comp
a_{\K\Del_\C\N,\M,\Ll},$$ 
$$\mu_{\M,\D,\Ll}: r_\M\Del_\D\id_\N\Rrightarrow(\id_\M\Del_\D l_\N)\comp a_{\M,\D \N}, $$
$$\lambda_{\C,\M,\N}: l_\M\Del_\D\id_\N\Rrightarrow l_{\M\Del_\D\N}\comp a_{\C,\M,\N}, $$ 
$$\rho_{\C,\M,\E}:(\id_\M\Del_\D r_\N)\comp a_{\M,\N,\E}\Rrightarrow r_{\M\Del_\D\N},$$ 
similar to those in (vi)-(ix) of \cite[Theorem 3.6.1]{Schaum1} and they satisfy three axioms analogous to those in (x) of {\em loc.cit.}. 
\end{enumerate}

\begin{rem} \rmlabel{iso not =}
To see 
that the functors on the two sides in the isomorphism $(\F'\Del_\D\G')\comp(\F\Del_\D\G)\iso
(\F'\comp\F)\Del_\D(\G'\comp\G)$ are a priori not equal, observe the following. As functors acting on the relative tensor product, 
they both are given up to an isomorphism by the defining functors $(\F'\times\G')\comp(\F\times\G)$ and $(\F'\comp\F)\times(\G'\comp\G)$, 
respectively, which are clearly equal between themselves. Since both functors are determined up to an isomorphism by the same functor, 
they only can be isomorphic between themselves, and one can not claim that they are equal. This applies to the point 4. above. 
By the same reason naturalities in the point 6. above hold only up to an isomorphism. 
\end{rem}

\begin{rem} 
For a fixed tensor category $\C$ it was proved in \cite{Gr} that $\Bimod(\C,\C)$ forms a monoidal 2-category in the sense of 
\cite{KV}, which is a non-semistrict monoidal bicategory, namely, it is weaker than a Gray monoid. Though, \cite{Gr} follows the approach of 
\cite{ENO} where the relative tensor product of bimodule categories is defined in such a way that a functor from such tensor product is 
defined {\em uniquely} by a balanced functor, whereas in \cite{Del} it is defined {\em up to a unique isomorphism}. This has for a 
consequence that many of the structure isomorphisms in $\Bimod(\C,\C)$ in \cite{Gr} result to be identities (coherence 3-cells: 
$\pi$ for the associativity constraint, \cite[Proposition 4.9]{Gr}, and $\lambda$ and $\rho$ for the left and right unity constraints 
\cite[Proposition 4.11]{Gr}), and moreover the associativity constraint $a$ itself is an isomorphism instead of being an equivalence 
(see the proof of \cite[Proposition 4.4]{Gr}). Substituting tensor categories by fusion categories (semisimple tensor categories), 
in \cite{Schaum1} it was proved that these form a tricategory (in a weaker sense than in \cite{Gr}, as we just pointed out). 
Semisimplicity does not influence the arguments of the proof, so we may take it as a proof that $\Tens$ is a tricategory. 
Note that the author uses the term ``2-functor'' for a pseudofunctor, \cite[Definition A 3.6]{Schaum1}. 
\end{rem}

\medskip

From the items 1, 4, 5, 6 and 7 above it is clear that $\Tens$ is a category enriched over the tricategory of 2-categories, 
pseudofunctors, weak natural transformations and modifications, which we denoted earlier by $2\x\Cat_{wk}$.

\subsection{Internal category structure for $\Tens$}

Now let us explain the $(1\times 2)$-category structure for tensor categories, {\em i.e.} of a category internal in $2\x\Cat_{wk}$. 
To do so we will give 2-categories $C_0$ and $C_1$, pseudofunctors $s,t,u$ and $c$, weak natural equivalences $\alpha^*, \lambda^*$ 
and $\rho^*$ and modifications $\pi^*, \mu^*, \lambda^*, \rho^*,\Epsilon^*$. As we announced at the beginning of this section, 
let $C_0$ be the 2-category of tensor categories, tensor functors and tensor natural transformations, 
and let $C_1$ be the 2-category of bimodule categories, bimodule functors and bimodule natural transformations. Fix tensor categories 
$\C$ and $\D$. To give a source and target 2-functors $s,t: C_1\to C_0$, let $\M$ be a $C\x\D$-bimodule category, $\F$ a $C\x\D$-bimodule functor, 
and $\omega$ a $C\x\D$-bimodule natural transformation. Set $s(\M)=\C, t(\M)=\D, s(\F)=\id_\C, t(\F)=\id_\D$ and $s(\omega)=\Id_{\id_\C}$ and 
$t(\omega)=\Id_{\id_\D}$ - the identity functors on $\C$ and $\D$ are obviously tensor functors, and the identity natural transformations on 
these two identity functors are obviously tensor ones. It is also clear that thus defined source and target functors are strict 2-functors. 
To define the identity 2-functor $u: C_0\to C_1$, take tensor categories $\C,\D$, tensor functors $F,G:\C\to\D$ and a tensor natural transformation 
$\zeta:F\to G$, and for $C, C', C''\in\C$ let $C\rtr C'$ denote the left action of $C$ on $C'$ and $C'\ltr C''$ the right action of $C''$ on $C'$. 
Set $u(\C)=\C$ as a $C$-bimodule category, $u(F)=F$ as a $C$-bimodule functor where $\D$ is a $\C$-bimodule category 
through $F$, that is: $C\rtr D\ltr C'=F(C)\ot D\ot F(C')$ for an object $D\in\D$ and where $\ot$ denotes the tensor product in $\D$ (a well-known fact), 
then $F$ is clearly $C$-bilinear. Finally, set $u(\zeta)=\zeta$, then similarly as for functors, $\zeta$ is a $\C$-bilinear natural transformation. 
To see that $u$ is indeed a 2-functor, take a further tensor category $\E$ and a tensor functor $G:\D\to\E$, then it is clear that $GF$ as a 
$\C$-bimodule functor is equal to the composition of $G$ as a $\D$-bimodule functor and $F$ as a $\C$-bimodule functor. 

The rest of the structure (a pseudofunctor $c$, pseudonatural equivalences  $\alpha^*, \lambda^*, \rho^*$ and modifications $\pi^*, \mu^*, 
\lambda^*, \rho^*,\Epsilon^*$) are given as in \prref{weak-int}. That $c$ is a pseudofunctor and not a 2-functor follows from 
\rmref{iso not =}. For this reason the tricategory $\Tens$ is an internal category in the tricategory $2\x\Cat_{wk}$, rather than 
in the Gray 3-category $2CAT_{nwk}$, as conjectured in \cite[Example 2.14]{DH} (1-cells in $2CAT_{nwk}$ are 2-functors, while in $2\x\Cat_{wk}$ 
these are pseudofunctors).

\medskip

Observe that the tricategory of 2-categories $2\x\Cat_{wk}$ embeds into the tricategory $\DblPs$. Thus the 
$(1\times 2)$-category of tensor categories is also an example of our alternative notion to intercategories (with non-trivial 3-cells 
involved in the internalization).

\bigskip

{\bf Acknowledgements.} 

I am profoundly thankful to Gabi B\"ohm for helping me 
understand the problem of non-fitting of monoids in her monoidal category $Dbl$ into intercategories, for suggesting me to try her $Dbl$ 
as the codomain for the embedding in Section 2, and for many other richly nurturing discussions. 
My deep thanks also go to the referee of the previous version of this paper. 
This research was partly developed during my sabbatical year from the Instituto de Matem\'atica Rafael Laguardia of the 
Facultad de Ingeneir\'ia of the Universidad de la Rep\'ublica in Montevideo (Uruguay). My thanks to ANII and PEDECIBA Uruguay 
for financial support. The work was also supported by the Serbian Ministry of Education, Science and
Technological Development through Mathematical Institute of the Serbian Academy of Sciences and Arts.


\pagebreak


\begin{figure}[hbt!]
\centering
\includegraphics[width=140mm]{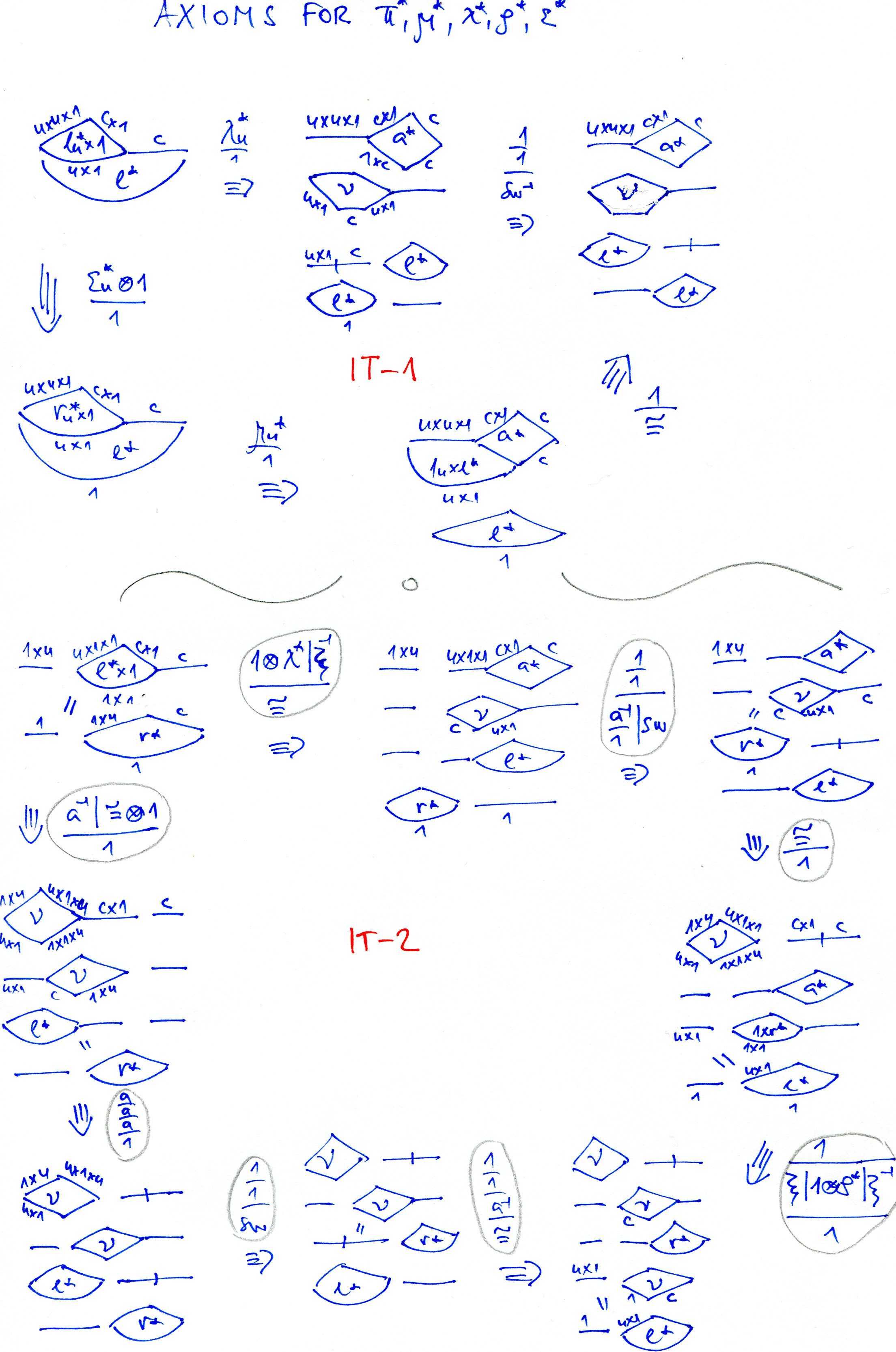}
\end{figure}

\begin{figure}[hbt!]
\centering
\includegraphics[width=140mm]{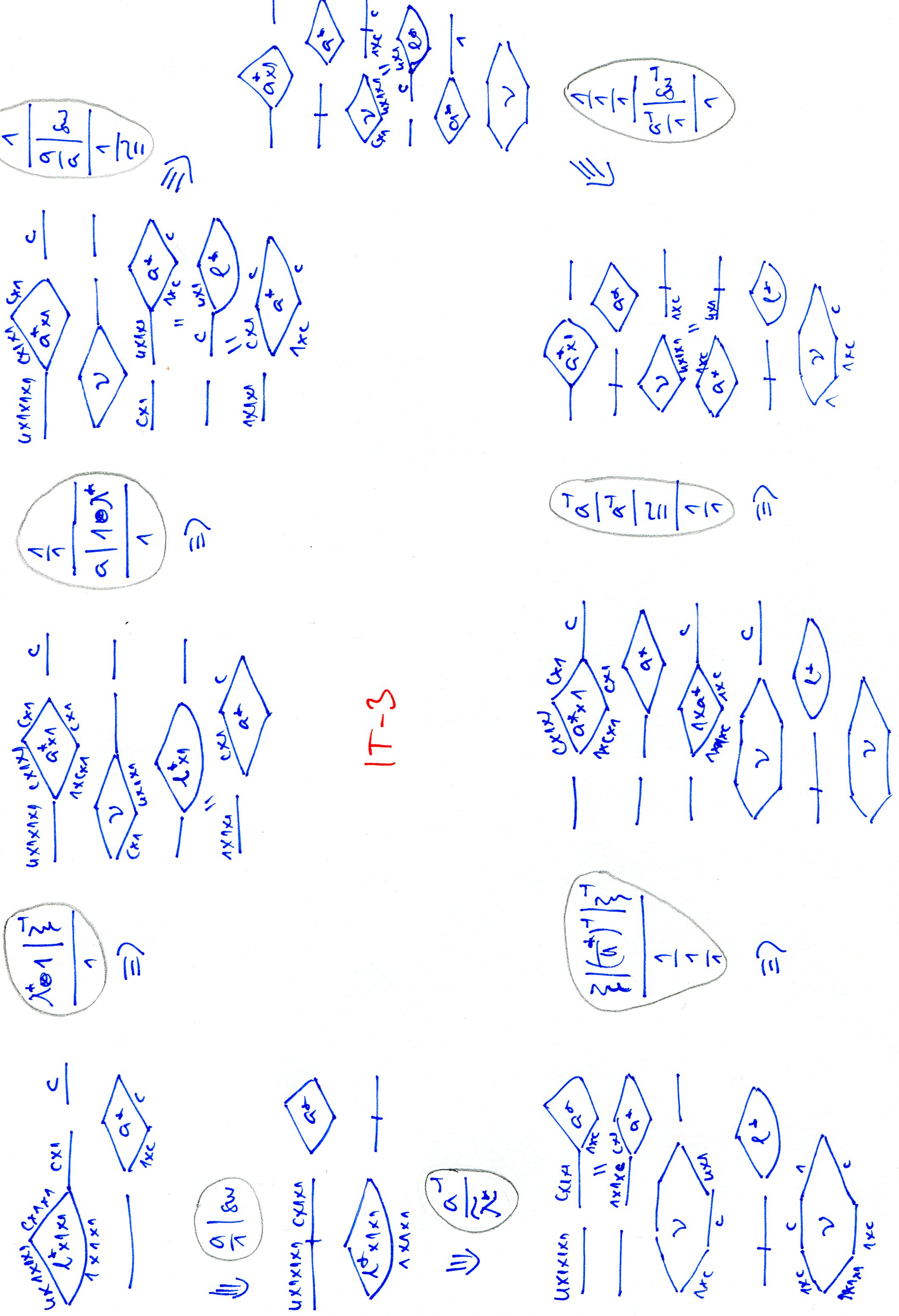}
\end{figure}

\begin{figure}[hbt!]
\centering
\includegraphics[width=140mm]{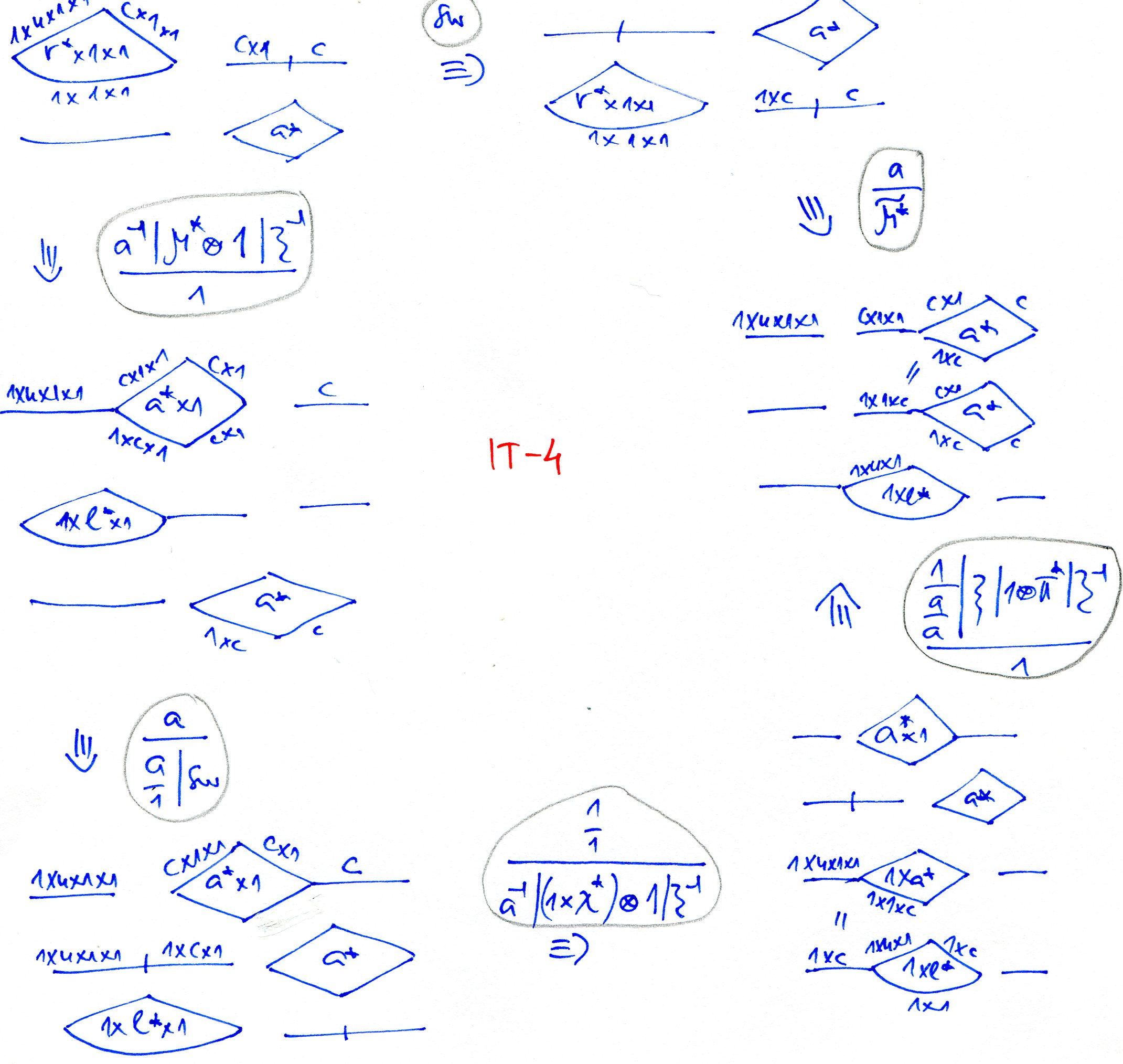}
\end{figure}

\begin{figure}[hbt!]
\centering
\includegraphics[width=140mm]{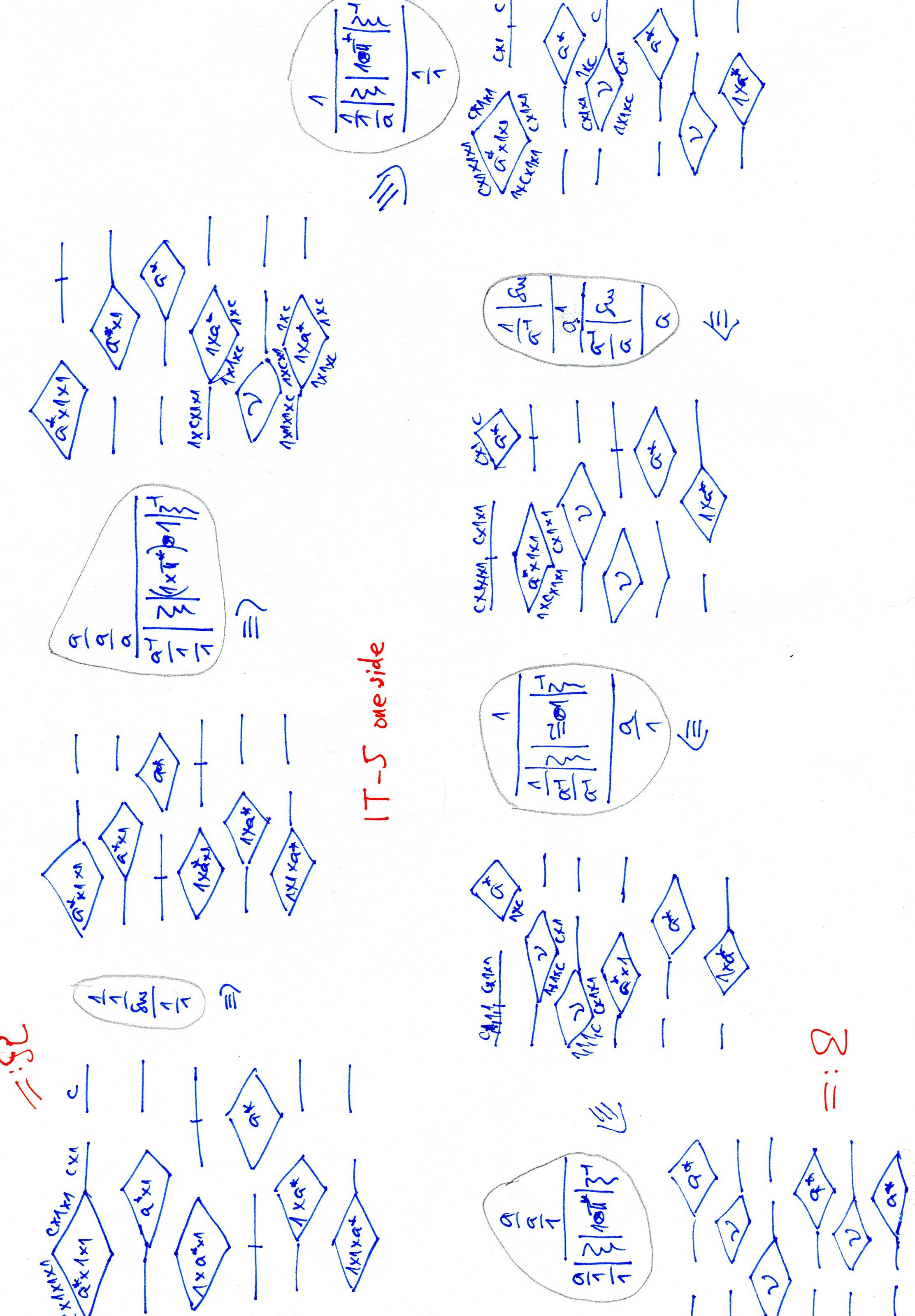}
\end{figure}

\begin{figure}[hbt!]
\centering
\includegraphics[width=140mm]{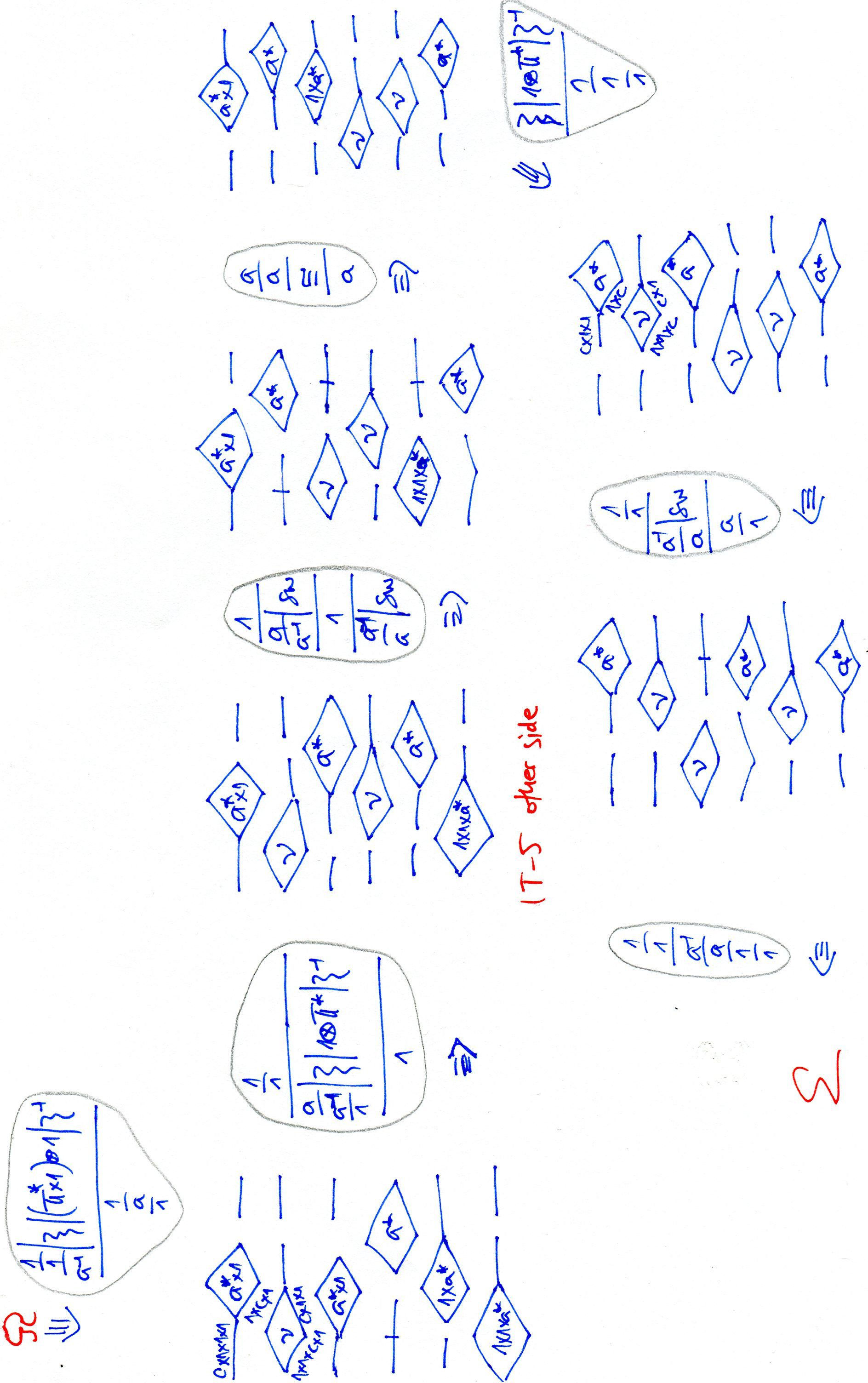}
\end{figure}
\bigskip

\end{document}